%% file: paper-v12.tex
\documentclass[twoside,11pt]{article}

\usepackage{graphicx, subfigure}
\usepackage[top=1in,bottom=1in,left=1in,right=1in]{geometry}
\usepackage[sort&compress]{natbib} 
\usepackage{amsfonts, amsmath, amssymb, amsthm, constants, bbm}
\usepackage{MnSymbol}
\usepackage{mathtools}
\usepackage{enumitem}
\usepackage{multirow}
\usepackage{booktabs}

\usepackage{color}
\definecolor{darkred}{RGB}{100,0,0}
\definecolor{darkgreen}{RGB}{0,100,0}
\definecolor{darkblue}{RGB}{0,0,150}

\usepackage{hyperref}
\hypersetup{colorlinks=true, linkcolor=darkred, citecolor=darkgreen, urlcolor=darkblue}
\usepackage{url}

\input notation.tex

\definecolor{purple}{rgb}{0.4,.1,.9}

\newcommand\blfootnote[1]{%
  \begingroup
  \renewcommand\thefootnote{}\footnote{#1}%
  \addtocounter{footnote}{-1}%
  \endgroup
}

\begin{document}
\thispagestyle{empty}

\title{Template Matching and Change Point Detection by M-estimation}
\author{Ery Arias-Castro \and Lin Zheng}
\date{}
\maketitle

\blfootnote{The authors are with the Department of Mathematics, University of California, San Diego, USA.  Contact information is available \href{http://math.ucsd.edu/\~eariasca}{here} and \href{http://www.math.ucsd.edu/people/graduate-students/}{here}.}

\begin{abstract}
We consider the fundamental problem of matching a template to a signal. We do so by M-estimation, which encompasses procedures that are robust to gross errors (i.e., outliers). Using standard results from empirical process theory, we derive the convergence rate and the asymptotic distribution of the M-estimator under relatively mild assumptions. We also discuss the optimality of the estimator, both in finite samples in the minimax sense and in the large-sample limit in terms of local minimaxity and relative efficiency. Although most of the paper is dedicated to the study of the basic shift model in the context of a random design, we consider many extensions towards the end of the paper, including more flexible templates, fixed designs, the agnostic setting, and more. 
\end{abstract}

%\tableofcontents

\section{Introduction} 
\label{sec:intro}

A basic task in signal processing is the matching of a template, aka filter or pattern, to a noisy signal \citep{brunelli2009template, turin1960introduction}. There has been an extensive amount of research is this very broad area, as applications are many, from locating blood vessels and tumors in medical imaging to more sophisticated tasks such as locating a human face or `recognizing' objects like cars in images \citep{hjelmaas2001face, belongie2002shape, lowe1999object, serre2005object}.

\paragraph{Scan statistic}
%{\em Scan statistic.}
In statistics, this template matching problem has been considered in different variants. Such is the literature on the scan statistic \citep{glaz2012scan, glaz2009scan, glaz2001scan}, where the focus has been on the detection of the presence of the filter somewhere in the noisy signal, rather than the estimation of its location (when present), for which theory has been developed, including first-order performance bounds \citep{MR2604703,morel,arias2005near,cluster, arias2018distribution} with a minimax decision theory perspective, as well as more refined results studying and even establishing the limit distribution \citep{naus2004multiple, pozdnyakov2005martingale, glaz2004multiple, wang2014variable, haiman2006estimation, proksch2018multiscale, konig2020multidimensional, sharpnack2016exact}. Some of this has some nontrivial intersections with the study of the maximum of various types of random walks and similar processes \citep{shao1995conjecture, jiang, boutsikas, sieg95, kabluchko2011extremes}.
In this literature, there are comparatively very few papers that tackle the problem of estimating the location of the template: \cite{jeng2010optimal} establish a consistency result for the location of very short intervals, while \cite{kou2017identifying} shows that the scan statistic is, as a location estimator, consistent near the signal-to-noise ratio required for mere detection.

%\begin{quote}
%{\em We study the behavior of M-estimators for matching a filter, obtaining their limit distributions and rates of convergence.}
%\end{quote}

\paragraph{Change point detection}
%{\em Change-point analysis.}
Discontinuities are features of great importance in practice, and it is no coincidence that the template that is most often considered in the scan statistic literature is the indicator of an interval --- or multiple intervals as in \citep{jeng2010optimal, hall2010innovated, frick2014multiscale} --- or a rectangle or other shape as in \citep{cluster, MR2604703}. A discontinuity is sometimes called a change point, and there is also a large body of work in that area. While some of the work on change point detection happens under the umbrella of sequential analysis where data are streaming in \citep{siegmund2013sequential}, what we are discussing here is most closely related to the `offline' or `a posteriori' setting where all data are readily available \citep{truong2020selective}. This spans, in itself, a very large literature \citep{chen2011parametric, csorgo1997limit, brodsky2013nonparametric, basseville1993detection}. First-order performance bounds are available, for example in \citep{frick2014multiscale, he2010asymptotic}, but distributional limits are more scarce, at least when it comes to the location of the discontinuities. \cite{hinkley1970inference} derives the asymptotic distribution of the maximum likelihood ratio when everything else about the model is known. \cite{yao1989least} obtain the asymptotic distribution of the least squares estimator of a piecewise constant signal, and this is extended to other U-type statistics in \citep{doring2011convergence, mauer2018least, ferger2001exponential, bai1997estimation}. 
\citep{ferger1994change, dumbgen1991asymptotic} consider estimating the location of a change of distribution in a sequence of random variables.
We note that all this work is done in the context of a deterministic design corresponding to a regular grid (as in a signal processing setting).

\paragraph{Alignment}
%{\em Alignment.}
The problem of matching a (clean/noiseless) template to a signal is closely related to the problem of matching two or more noisy signals, sometimes referred to as aligning or registering or synchronizing the signals. While this literature is also very large on the side of methodology and applications \citep{zitova2003image, hajnal2001medical, sotiras2013deformable}, there is a sizable literature that develops theory for such problems. Some of these theoretical developments were made in the context of functional data, which is where in statistics such problems are most prominent, and where the problem is sometimes called `self-modeling' or `shape invariant modeling' \citep{lawton1972self}. In this context, consistency results and/or rates of convergence are obtained in \citep{kneip1988convergence, kneip1992statistical, wang1999synchronizing}, while distributional limits are derived in \citep{hardle1990semiparametric, gamboa2007semi, bigot2009estimation, trigano2011semiparametric, kneip1995model}, and also in \citep{vimond2010efficient}, where questions of efficiency are considered in a semi-parametric model were only smoothness assumptions are made on the common `shape'. 
\cite{collier2012minimax, collier2015curve} consider a hypothesis testing problem in that setting.
We note that the signals are typically smoothed before the alignment is carried out, and that the Fourier transform plays an important role, and the noise is often assumed to be Gaussian or to have sub-Gaussian tails. This is in contrast to the present setting in which no smoothing is used and the Fourier transform plays no role, and we work with minimal assumptions on the noise distribution.
We mention a more recent line of work on this problem which focuses on more complex and noisy settings arising in applications such as cryo-electron microscopy \citep{perry2019sample, wang2013exact, perry2018message}.

\paragraph{Contribution}
%{\em Contribution.}
In the present paper we study a standard mathematical model for matching a template to a noisy signal by M-estimation. While the most popular method may still be based on maximizing the (Pearson) correlation, the estimators we study can be made much more robust to heavy-tailed noise or the presence of outlying observations. We draw on standard empirical process theory and decision theory, as expounded in \citep{van2000asymptotic}, to derive limit distributions and minimax convergence rates in a wide array of situations.

%\subsection{Notation}
%Define the vectors 
%\beq
%\by = (y_1, \dots, y_n), \quad 
%\beps = (\eps_1, \dots, \eps_n), \quad
%\bbf = (f(1/n), \dots, f(n/n)), \quad
%\bg = (g(1/n), \dots, g(n/n)).
%\eeq
%The indexing will be modulo $n$, meaning that for a vector $\bv = (v_1, \dots, v_n)$, $v_{i+n} \equiv v_i$ for all $i \in [n]$.
%For a vector $\bv = (v_1, \dots, v_n)$ and $k \in [n]$, define $\bv_k = (v_{1-k}, \dots, v_{n-k})$.  Also, for $\bu = (u_1, \dots, u_n)$ and $\bv = (v_1, \dots, v_n)$, both in $\bbR^n$, define the scalar product and corresponding norm
%\beq
%\<\bu, \bv\> = \frac1n \sum_{i \in [n]} u_i v_i, \quad
%|\bu|^2 = \<\bu, \bu\> = \frac1n \sum_{i \in [n]} u_i^2.
%\eeq
%Of interest to us is the estimation of 
%\beq
%k^* := \argmax_{k \in [n]}\ \<\bg_k, \bbf\>.
%\eeq
%
%We note that, as long as $f$ and $g$ are reasonably regular, as $n \to \infty$,
%\beq
%k^* = k^*_n \to t^* := \argmax_{t \in [0,1]}\ \<g_t, f\>,
%\eeq
%where for a periodic function $u : [0,1] \to \bbR$ and $t \in \bbR$, $u_t$ is the function defined as $u_t(s) = u(s - t)$ (meaning the translate of $u$ by $t$, understood modulo 1), and for two measurable functions $u, v : [0,1] \to \bbR$,
%\beq
%\<u, v\> = \int u(t) v(t) {\rm d} t, \quad
%|u|^2 = \<u, u\> = \int u(t)^2 {\rm d} t.
%\eeq

%\subsection{Notation}
%When indexing a vector $x = (x(1), \dots, x(n))$, we will take the indices modulo $n$, meaning $x(k) = x(k \mod n)$ for all $k \in \bbZ$.
%We write $a = b \pm c$ when $a \in [b - c, b + c]$.

\subsection{Model}
\label{sec:model}

%\paragraph{Deterministic design}
We assume a regression model with additive error
\beq\label{model}
Y_i = f(X_i - \theta^*) + Z_i, \quad i = 1, \dots, n,
\eeq
where $f:\bbR \to \bbR$ is a known function referred to as the {\em template}, and $\theta^* \in \bbR$ is the unknown {\em shift} of interest.
The design points $X_1, \dots, X_n$ are assumed to be iid with density $\lambda$. 
The noise or measurement error variables $Z_1, \dots, Z_n$ are assumed iid with density $\phi$, and independent of the design points.
Note that in this model $(X_1, Y_1), \dots, (X_n, Y_n)$ independent and identically distributed.

\begin{asp}
\label{asp:template}
We assume everywhere that $f$ is compactly supported and, for convenience, c\`adlag with none or finitely many discontinuities. In particular, $f$ is bounded.
And to make sure the shift parameter is identifiable, we also assume that, for any $\theta \ne \theta^*$, $f(\cdot-\theta) \ne f(\cdot-\theta^*)$ on a set of positive measure under $\lambda$, or equivalently, $\int (f(x-\theta^*) - f(x-\theta))^2 \lambda(x) \d x > 0$ whenever $\theta \ne \theta^*$.
\end{asp}

\begin{asp}
\label{asp:design}
We assume everywhere that $\lambda$ is compactly supported. 
%To simplify the statement of our results, we assume that its support contains the support of $f(\cdot - \theta^*)$.
\end{asp}

\begin{asp}
\label{asp:noise}
We assume everywhere that $\phi$ is even, so that the noise is symmetric about 0.
\end{asp}

The assumption that the template and the design density both have compact support is for convenience --- although it already applies in most of the settings encountered in practice. It allows us to effectively restrict the parameter space to a compact interval of the real line. Indeed, take $A$ large enough that $f$ and $\lambda$ are both supported on $[-A,A]$. Then the model \eqref{model} is effectively parameterized by $\theta \in [-2A,2A]$ as the model is the same, namely $Y_i = Z_i$, whenever $\theta$ is outside that interval.
The assumption that the noise is symmetric is not needed everywhere, but already covers interesting situations.

\subsection{Goal and methods}
Our goal is, in signal processing terminology, to match the template $f$ to the signal $Y$, which in statistical terms consists in the estimation of the shift $\theta^*$. 
We consider an {\em M-estimator} defined implicitly as the solution to the following optimization problem
\begin{equation}
\label{loss}
\hat\theta := \argmin_{\theta} \sum_{i=1}^n \loss(Y_i - f(X_i - \theta)),
\end{equation}
where $\loss$ is a loss function chosen by the analyst.
A popular choice is the squared error loss, $\loss(y) = y^2$, which defines the least squares estimator
\begin{equation}
\label{ls}
\hat\theta := \argmin_{\theta} \sum_{i=1}^n (Y_i - f(X_i - \theta))^2.
\end{equation}
This is the maximum likelihood estimator when the noise distribution is Gaussian.
Another popular choice is the absolute-value loss, $\loss(y) = |y|$, which defines the least absolute-value estimator
\begin{equation}
\label{la}
\hat\theta := \argmin_{\theta} \sum_{i=1}^n |Y_i - f(X_i - \theta)|.
\end{equation}
This is the maximum likelihood estimator when the noise distribution is Laplace.
Other popular losses include the Huber loss and the Tukey loss.
The Huber loss is of the form 
\begin{equation}
\label{Huber}
\loss(y) = \begin{cases}
\tfrac12 y^2 & \text{if } |y| \le c, \\
c |y| - \tfrac12 c^2 & \text{if } |y| > c.
\end{cases}
\end{equation}
Like the squared error and absolute-value losses, this loss is even and convex.
The Tukey loss is of the form
\begin{equation}
\label{Tukey}
\loss(y) = \begin{cases}
1 - (1 - (y/c)^2)^3 & \text{if } |y| \le c, \\
1 & \text{if } |y| > c.
\end{cases}
\end{equation}
The Tukey loss is also even, but not convex as it is in fact bounded.
In both cases, $c \ge 0$ is a parameter traditionally chosen to maximize the efficiency of the estimator relative to the maximum likelihood estimator under a particular noise distribution (often Gaussian).
%and the 0-1 loss is of the form
%\begin{equation}
%\label{0-1}
%\loss(y) = \begin{cases}
%0 & \text{if } |y| \le c, \\
%1 & \text{if } |y| > c.
%\end{cases}
%\end{equation}

\begin{asp}
\label{asp:loss}
We assume everywhere that $\loss$ is non-negative, even, non-decreasing on away from the origin, and because these are the losses used in practice, we also assume that $\loss$ is either Lipschitz or has a Lipschitz derivative. (We obviously assume that $\loss$ is not constant.)
\end{asp}

\begin{rem}
\label{rem:MLE}
If in addition $c := \int \exp(- L(z)) \d z < \infty$, then $\phi(z) := c^{-1} \exp(- L(z))$ satisfies \aspref{noise}, and $\hat\theta$ in \eqref{loss} is the maximum likelihood estimator when this is the noise distribution. This additional condition is for example fulfilled when the loss is even and convex.
\end{rem}  

\subsection{Content}
Our focus will be on a generic M-estimator and its asymptotic properties (as $n \to \infty$).
In \secref{consistency}, we establish the consistency of the M-estimator under mild assumptions. 
In \secref{smooth}, we consider the `smooth setting', which corresponds here to the case where the template is Lipschitz. We obtain distributional limits and convergence rates in $\sqrt{n}$, and discuss some optimality properties: local asymptotic minimaxity, relative efficiency, and finite-sample minimaxity.
See \figref{smooth} for an illustration.
In \secref{non-smooth}, we consider the `non-smooth setting', which in particular includes the case where the template is discontinuous. Similarly, we obtain the limit distribution and the convergence rate, which is in $n$ when the template is piecewise Lipschitz with at least one discontinuity.
See \figref{non-smooth} for an illustration.
In \secref{flexible} we consider more flexible models for which we derive similar results.
In \secref{extensions}, we discuss a number of variants and extensions, including the setting where the design is fixed. 
In \secref{numerics}, we present the result of some numerical experiments, which are only meant to probe our theory in finite samples.

\begin{figure}[ht!]
\centering
\subfigure[When the template is Lipschitz, and the other model components satisfy some mild assumptions, the model is `smooth'. The M-estimator is $\sqrt{n}$-consistent and asymptotically normal.]{
\label{fig:smooth}
\includegraphics[width=0.4\textwidth]{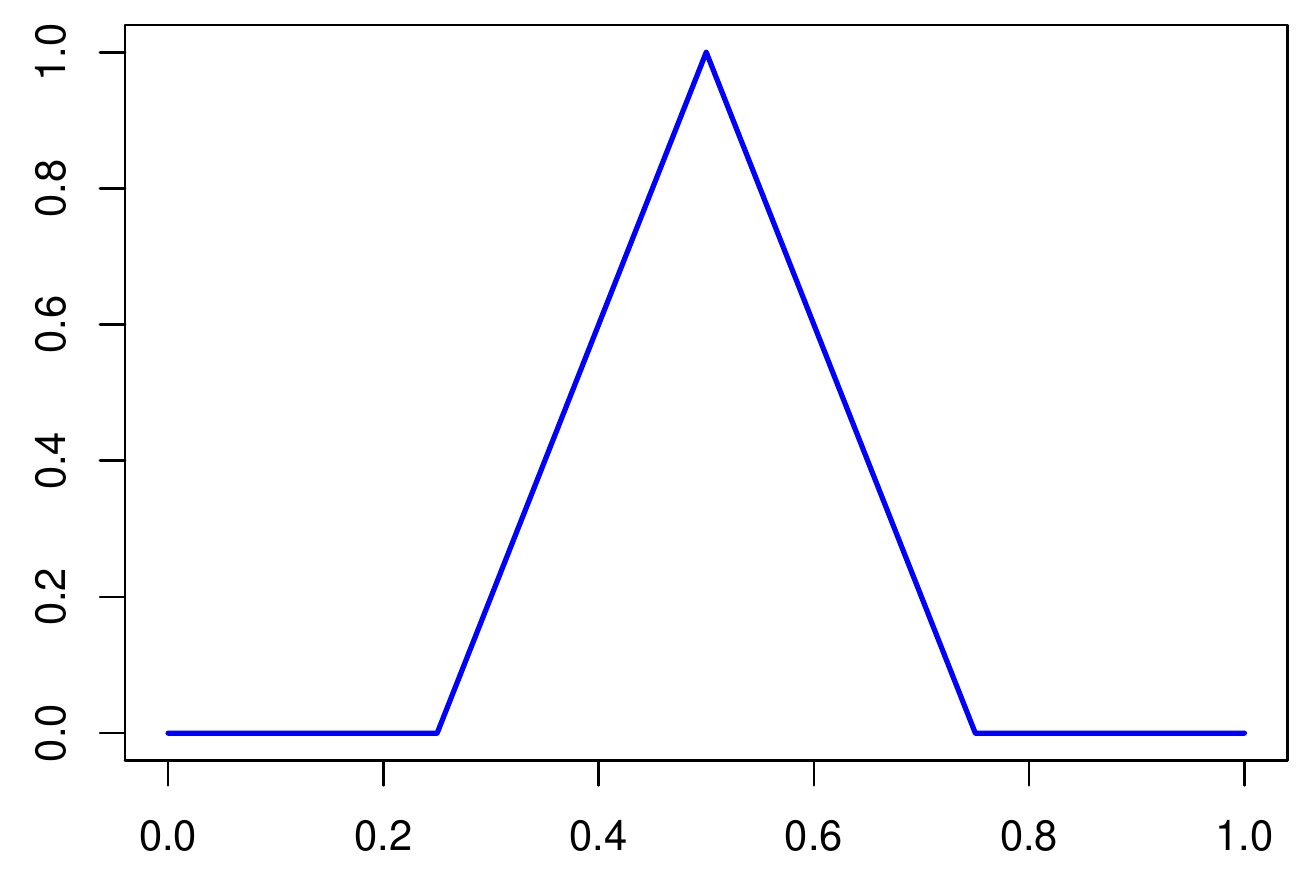}	}
\quad\quad\quad\quad
\subfigure[When the template is piecewise Lipschitz with at least one discontinuity, and the other model components satisfy some mild assumptions, the M-estimator is $n$-consistent and its asymptotic distribution, although not well-defined, is essentially the minimum of a marked Poisson process.]{
\label{fig:non-smooth}
\includegraphics[width=0.4\textwidth]{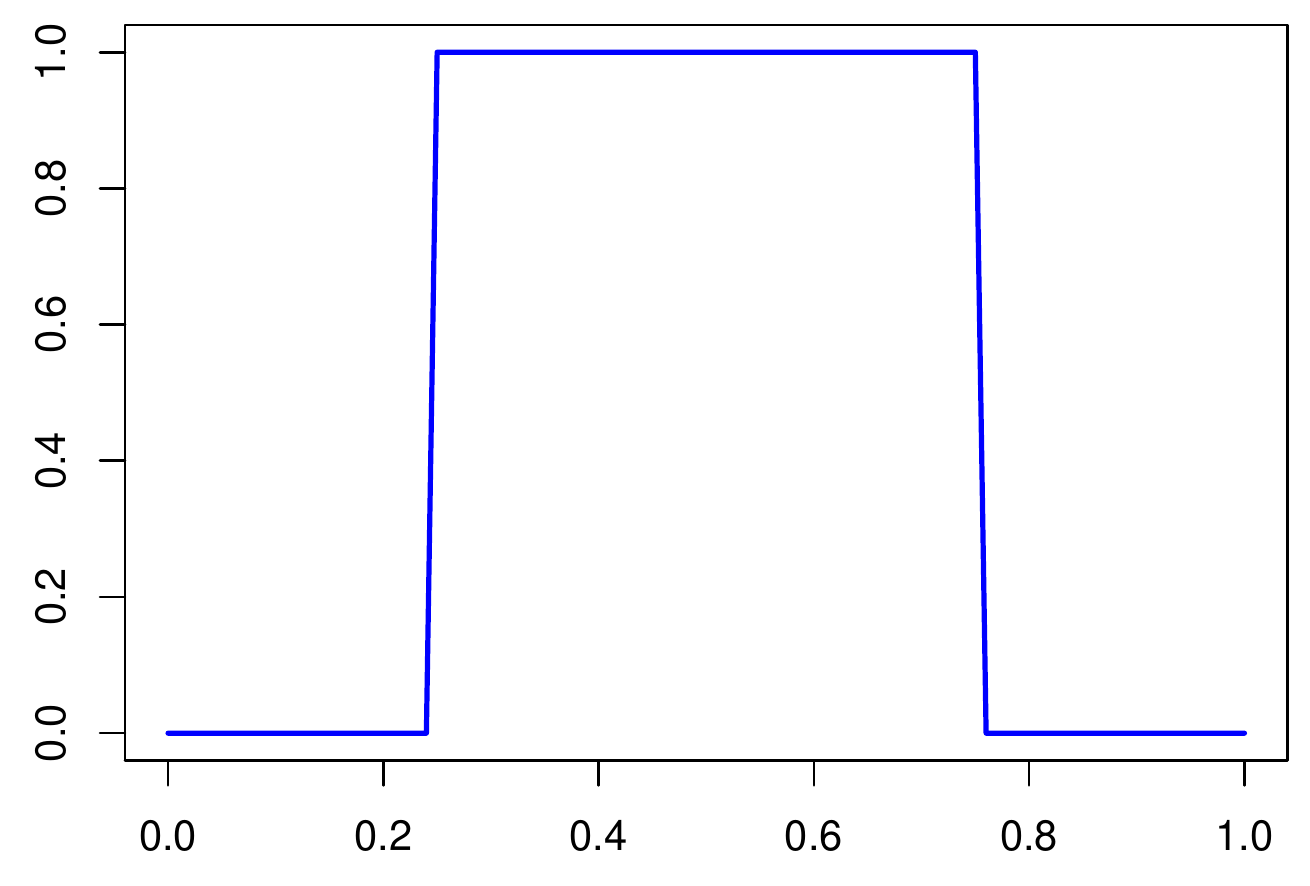}}
\caption{Two emblematic templates: a Lipschitz template on the left representing a `smooth' setting, and a piecewise Lipschitz template on the right representing a `non-smooth' setting.}
\end{figure}

\section{Consistency}
\label{sec:consistency}

Having chosen to work with a particular loss function $\loss$, define
\begin{equation}
\label{m}
m_\theta(x, y) := \loss(y - f(x - \theta)),
\end{equation}
so that the estimator in \eqref{loss} can be equivalently defined via
\begin{align}
\label{M-estimator}
\hat\theta_n := \argmin_\theta \widehat{M}_n(\theta), &&
\widehat{M}_n(\theta) := \frac1n \sum_{i=1}^n m_\theta(X_i, Y_i),
\end{align}
where we have added the subscript $n$ to emphasize that the estimator is being computed on a sample of size $n$.
So that we can carry out a large-sample analysis, we require that
\begin{equation}
\label{finite_loss}
\E[m_\theta(X, Y)] < \infty, \quad \text{for all $\theta$},
\end{equation}
where $(X,Y)$ is a generic observation and $\E$ is the expectation with respect to the distribution of $(X,Y)$ under $\theta^*$, the true value of the shift.
Under \aspref{template} and \aspref{loss}, the following can be seen to be sufficient
\begin{equation}
\label{Z0}
\E[\loss(Z)] < \infty,
\end{equation}
where $Z$ is a generic noise random variable.
For the squared error loss, the requirement is that the noise distribution have a finite second moment, while for the absolute-value and the Huber losses, the requirement is that the noise distribution have a finite first moment, and the requirement is automatically fulfilled when the loss is bounded.
When \eqref{finite_loss} holds, $\widehat{M}_n(\theta)$ has a well-defined expectation under $\theta^*$ given by
\begin{equation}
\label{M}
M(\theta) := \E[m_\theta(X, Y)].
\end{equation}
In particular, for any $\theta$, $\widehat{M}_n(\theta)$ converges to $M(\theta)$ in probability by the law of large numbers, and thus one anticipates that, under some additional conditions perhaps, $\hat\theta_n$ would converge to a minimizer of $M$.

\begin{asp}
\label{asp:M}
$M$ introduced in \eqref{M} is well-defined.
Moreover, $\theta^*$ is the unique minimum of $M$ and $\inf\{M(\theta) : |\theta - \theta^*| \ge \delta\} > M(\theta^*)$ for every $\delta > 0$.
\end{asp}

With \aspref{noise} and \aspref{loss} in place, \aspref{M} is satisfied for the squared error, absolute-value, and Huber losses, and also for the Tukey loss if in addition the noise distribution is unimodal. See \lemref{consistent}.
  
\begin{theorem}[Consistency]
\label{thm:consistent}
Under the basic assumptions, $\hat\theta_n$ is consistent for $\theta^*$.
\end{theorem}

\begin{proof}
The requirements for applying \cite[Th 5.7]{van2000asymptotic} are {\em (i)} $M$ takes its minimum at $\theta^*$ and is bounded away from its minimum at $\theta$ away from $\theta^*$ --- which is implied in \aspref{M}; and
{\em (ii)} $\widehat{M}_n$ converges to $M$ uniformly, meaning that   
\begin{equation}\label{GC}
\sup_\theta |\widehat{M}_n(\theta) - M(\theta)| \to 0, \quad n \to \infty,
\end{equation}
in probability --- which is due to the fact that the function class $\{m_\theta : \theta \in \bbR\}$ is Glivenko--Cantelli. This latter property is essentially established in \cite[Ex 19.8]{van2000asymptotic}, where the assumption of continuity can be relaxed to continuity almost everywhere, which holds here by the fact that $f$ and $\loss$ are at least piecewise continuous. \cite[Ex 19.8]{van2000asymptotic} works under the assumption that the parameter space is compact, and this is effectively the case here. 
\end{proof}

\section{Smooth setting}
\label{sec:smooth}

We start by analyzing the situation where $f$ is Lipschitz.
%\begin{asp}
%\label{asp:template_smooth}
%The template $f$ is Lipschitz. Its derivative, which exists almost everywhere, will be denoted by $f'$, and its Lipschitz constant $|f'|_\infty$. 
%\end{asp}
It turns out that, when this is the case, under mild assumptions on the design and noise distributions as well as the loss function, the situation is `standard' for a parametric model in the sense that the M-estimator is $\sqrt{n}$ - consistent and asymptotically normal. In addition, the model is `smooth' in the sense of being quadratic mean differentiable \cite[Sec 5.5]{van2000asymptotic}, which then implies that the maximum likelihood estimator --- which can coincide with the M-estimator as mentioned in \remref{MLE} --- is locally asymptotic minimax and efficient. 

\begin{rem}
We remind the reader that a Lipschitz function is differentiable almost everywhere with a bounded derivative, and that it is the integral of that derivative (i.e., it is absolutely continuous). For such a function $g$, we will denote by $g'$ its derivative and by $|g'|_\infty$ its Lipschitz constant (which bounds the derivative when it exists).
\end{rem}

\subsection{Asymptotic normality}
\label{sec:normal}

We first consider a smooth loss, encompassing the squared error loss and the Huber loss, among others.

\begin{theorem}%[Asymptotic normality: smooth loss]
\label{thm:normal}
Suppose the basic assumptions are in place. Assume, in addition, that $f$ is Lipschitz, that $\lambda$ is Lipschitz on an open set containing the support of $f(\cdot - \theta^*)$, and that $\loss$ has a Lipschitz first derivative.
Unless $\loss$ itself is Lipschitz, assume that the noise distribution has finite second moment.
%\begin{equation}
%\label{Z1Z2}
%\E[\loss'(|Z| + 2 |f|_\infty)^2] < \infty, \qquad
%\E[\loss''(|Z| + 2 |f|_\infty)] < \infty.
%\end{equation}
Then $\sqrt{n} (\hat\theta_n - \theta^*)$ is asymptotically normal with mean $0$ and variance $\tau^2$, with
\begin{equation}
\label{tau}
\tau^2 := \frac{C_{\phi, \loss}}{\E[f'(X-\theta^*)^2]}, \quad C_{\phi, \loss} := \frac{\E[\loss'(Z)^2]}{\E[\loss''(Z)]^2}.
\end{equation}
\end{theorem}

\begin{rem}
The assumption that $\lambda$ is Lipschitz is not needed, as can be seen by using the approach of \secref{non-smooth}. However, it makes for a particularly straightforward proof.
\end{rem}

\begin{proof}
Assume $\theta^* = 0$ without loss of generality.
In \cite[Th 5.23]{van2000asymptotic}, the first condition is that $\theta \mapsto m_\theta(x,y)$ be differentiable at $0$ for almost every $(x,y)$, which is clearly the case here, as it is differentiable with derivative
\begin{equation}
\label{dot_m}
\dot m_\theta(x,y) = f'(x-\theta) \loss'(y-f(x-\theta)),
\end{equation}

If $\loss$ itself is Lipschitz, we have
\begin{align}
|m_{\theta_1}(x,y) - m_{\theta_2}(x,y)|
&= |\loss(y-f(x-\theta_1)) - \loss(y-f(x-\theta_2))| \\
&\le |\loss'|_\infty\, |f(x-\theta_1) - f(x-\theta_2)| \\
&\le |\loss'|_\infty\, |f'|_\infty |\theta_1 - \theta_2| \\
&=: \overline{m}(x,y)\, |\theta_1 - \theta_2|.
\end{align}
Otherwise, we have
\begin{align}
|m_{\theta_1}(x,y) - m_{\theta_2}(x,y)|
&= |\loss(y-f(x-\theta_1)) - \loss(y-f(x-\theta_2))| \\
&\le |\loss''|_\infty\, (|y-f(x)| + 2|f|_\infty)\, |f'|_\infty\, |\theta_1 - \theta_2| \\
&=: \overline{m}(x,y)\, |\theta_1 - \theta_2|.
\end{align}
In the second inequality we used the fact that $|y - f(x-\theta)| \le |y-f(x)| + |f|_\infty$ for all $\theta$ and that $|\loss'(y)| \le |\loss''|_\infty |y|$ for all $y$ since $\loss'(0) = 0$.
In either case, $\E[\overline{m}(X, Y)^2] < \infty$, so that the second condition of \cite[Th 5.23]{van2000asymptotic} is satisfied.

The third condition in that theorem is that $M(\theta)$ admits a Taylor expansion of order two at $\theta^*$ with nonzero second order term. Our assumptions imply that we can differentiate once inside the expectation defining $M$ to obtain its first derivative.
To see this, note that $\theta \mapsto m_\theta(x,y)$ has first derivative given by \eqref{dot_m}, which is dominated by $|f'|_\infty |\loss'|_\infty$ if $\loss$ is Lipschitz, and by $|f'|_\infty |\loss''|_\infty (|y-f(x)| + 2|f|_\infty)$ otherwise, which is integrable in either case. Hence, $M$ is differentiable, with
\begin{align*}
M'(\theta) 
&= \E[\dot m_\theta(X,Y)] \\
&= \E[f'(X-\theta) \loss'(Y-f(X-\theta))] \\
&= \int\int f'(x-\theta) \loss'(f(x)-f(x-\theta)  +z) \phi(z) \lambda(x) \d z \d x \\
&= \int\int f'(x) \loss'(f(x+\theta)-f(x)  +z) \phi(z) \lambda(x+\theta) \d z \d x.
\end{align*}
Thus, by applying a simple change of variables we have transferred $\theta$ away from $f'$ at the cost of burdening $\lambda$. But our assumptions are exactly what we need to differentiate inside the integral, since the integrand is differentiable with derivative 
\begin{equation}
f'(x) \big\{f'(x+\theta) \loss''(f(x+\theta)-f(x) +z) \lambda(x+\theta) + \loss'(f(x+\theta) - f(x) + z) \lambda'(x+\theta)\big\} \phi(z),
\end{equation}
which is dominated by
\begin{equation}
|f'|_\infty \big\{|f'|_\infty |\loss''|_\infty |\lambda|_\infty + |\loss'|_\infty |\lambda'|_\infty\big\} \phi(z),
\end{equation} 
if $\loss$ is Lipschitz, and otherwise by 
\begin{equation}
|f'|_\infty \big\{|f'|_\infty |\loss''|_\infty |\lambda|_\infty + |\loss''|_\infty (|z| + 2 |f|_\infty) |\lambda'|_\infty\big\} \phi(z),
\end{equation}
and this is integrable in either case. 
(Recall that $\lambda$ is compactly supported.) 
Hence, $M$ is indeed twice differentiable, with
\begin{align*}
M''(0) 
&= \int \int f'(x) \big\{f'(x) \loss''(z) \lambda(x) + \loss'(z) \lambda'(x)\big\} \phi(z) \d z \d x \\
&= \int \int f'(x)^2 \loss''(z) \lambda(x) \phi(z) \d z \d x \\
&= \E[f'(X)^2] \E[\loss''(Z)],
\end{align*}
because $\int_{-\infty}^\infty \loss'(z) \phi(z) \d z = 0$ due to the fact that $\loss'$ is odd and $\phi$ is even.

The last condition in \cite[Th 5.23]{van2000asymptotic} is that $\hat\theta_n$ be consistent, which we have already established in \thmref{consistent}.

Therefore, \cite[Th 5.23]{van2000asymptotic} applies, and implies that $\sqrt{n} (\hat\theta_n - \theta^*)$ is asymptotically normal with mean $0$ and variance 
\begin{equation}
\frac{\E[\dot m_0(X,Y)^2]}{M''(0)^2},
\end{equation}
the latter reducing to \eqref{tau} after some simplifications.
\end{proof}

%\begin{rem}
%A similar result holds without any assumptions on the design density $\lambda$ if the template $f$ is twice differentiable. We leave the details to the interested reader. We opted for a result with minimal smoothness assumptions on $f$, at the cost of a mild smoothness requirement on $\lambda$.
%\end{rem}

\thmref{normal} applies to the squared error, the Huber, and the Tukey losses with no additional conditions other than the basic assumptions.
For squared error loss, the factor $C_{\phi, \loss}$ in the asymptotic variance is given by
\begin{equation}
\label{tau_L2}
C_{\rm squared} := \E[Z^2] = \text{noise variance}.
\end{equation}
%Note that $\E[Z^2]$ is the nose variance.
For the Huber loss \eqref{Huber},
\begin{equation}
\label{tau_huber}
C_{\rm huber} := \frac{\E[Z^2 \wedge c^2]}{\P(|Z| \le c)^2}.
\end{equation}

\thmref{normal} does not apply to the absolute-value loss. But very similar arguments can be used to obtain a normal limit distribution for that loss.

\begin{theorem}%[Asymptotic normality: absolute-value loss]
\label{thm:normal_L1}
In the context of \thmref{normal}, assume that $\loss$ is the absolute-value loss and that $\phi$ is continuous and strictly positive at the origin. Then the same conclusion holds with 
\begin{equation}
\label{tau_L1}
C_{\phi,\loss} = \frac1{4\phi(0)^2}.
\end{equation} 
%If $\loss$ is the 0-1 loss instead, then this is also true but with $C_{\phi,\loss} = XXX$.
\end{theorem}

\begin{proof}
Assume $\theta^* = 0$ without loss of generality.

In \cite[Th 5.23]{van2000asymptotic}, the first condition is that $\theta \mapsto m_\theta(x,y)$ be differentiable at $0$ for almost every $(x,y)$, which is clearly the case here, as it is differentiable with derivative
\begin{equation}
\label{dot_m_L1}
\dot m_\theta(x,y) = f'(x-\theta) \sign(y-f(x-\theta)),
\end{equation}

We also have
\begin{align}
|m_{\theta_1}(x,y) - m_{\theta_2}(x,y)|
&= \big||y-f(x-\theta_1)| - |y-f(x-\theta_2)|\big| \\
&\le |f(x-\theta_1) - f(x-\theta_2)| \\
&\le |f'|_\infty\, |\theta_1 - \theta_2|.
\end{align}
Hence, the second condition of \cite[Th 5.23]{van2000asymptotic} is satisfied.
 
Our assumptions imply that we can differentiate once inside the expectation defining $M$ to obtain its first derivative
\begin{align*}
M'(\theta) 
&= \E[\dot m_\theta(X,Y)] \\
&= \E[f'(X-\theta) \sign(Y-f(X-\theta))] \\
&= \int \int f'(x-\theta) \sign(f(x)-f(x-\theta)  +z) \phi(z) \lambda(x) \d z \d x \\
&= \int \int f'(x) \sign(f(x+\theta)-f(x)  +z) \phi(z) \lambda(x+\theta) \d z \d x \\
&= \int f'(x) \big(1 -2 \Phi(f(x)-f(x+\theta))\big) \lambda(x+\theta) \d x,
\end{align*}
where $\Phi(z) := \int_{-\infty}^z \phi(u) \d u$.
The integrand is differentiable with derivative 
\begin{equation}
f'(x) \big\{2 f'(x+\theta) \phi(f(x)-f(x+\theta)) \lambda(x+\theta) + \big(1 -2 \Phi(f(x)-f(x+\theta))\big) \lambda'(x+\theta)\big\}
\end{equation}
which is bounded and therefore dominated as $\theta$ varies.
The remaining arguments are as in the proof of \thmref{normal}.
\end{proof}

\begin{rem}
The expression for the asymptotic variance for the absolute-value loss given could have been anticipated based on the corresponding expression for the Huber loss given in \eqref{tau_huber}. Indeed, as $c \to 0$, the Huber loss \eqref{Huber} converges pointwise to the absolute-value loss, so that we could have speculated that the same would be true for the asymptotic variances.  It turns out that this prediction would have been correct, as it is indeed the case that
\begin{equation}
\E[Z^2 \wedge c^2]/\P(|Z| \le c)^2 \xrightarrow{c \to 0} 1/4\phi(0)^2
\end{equation}
when $\phi$ may be taken continuous at the origin. It is also the case that the Huber loss converges pointwise to the squared error loss when $c \to \infty$ instead, and that 
\begin{equation}
\E[Z^2 \wedge c^2]/\P(|Z| \le c)^2 \xrightarrow{c \to \infty} \E[Z^2].
\end{equation}
\end{rem}

\begin{rem}
\label{rem:location_model}
The constant $C_{\phi, \loss}$ in \thmref{normal} and \thmref{normal_L1} is the asymptotic variance in the classical location model where $Y_i = \theta^* + Z_i$. The only way in which the model we assume \eqref{model} is different asymptotically is in the denominator in \eqref{tau}.
\end{rem}

\subsection{Local asymptotic minimaxity}
\label{sec:LAM}

Under the same smoothness condition on the template as in \secref{normal}, namely, assuming that $f$ has a bounded derivative, and under some smoothness assumption on the noise density (and not the design density this time), the statistical model is smooth in the sense of being {\em quadratic mean differentiable (QMD)}. 
Indeed, under $\theta$, the joint distribution of $(X,Y)$ has density 
\begin{equation}
\label{joint_density}
p_\theta(x, y) := \lambda(x) \phi(y - f(x-\theta)).
\end{equation} 
The function $\theta \mapsto p_\theta(x,y)$ is differentiable when $f$ and $\phi$ are.
The {\em information} at $\theta$ is defined as 
\begin{align}
I_\theta
&:= \int \int \frac{\dot p_\theta(x, y)^2}{p_\theta(x,y)} \d y \d x \\
&= \int \int \frac{\lambda(x) f'(x-\theta)^2 \phi'(y - f(x-\theta))^2}{\phi(y - f(x-\theta))} \d y \d x \\
&= \int \lambda(x) f'(x-\theta)^2 \d x \times \int \frac{\phi'(y)^2}{\phi(y)} \d y,
\end{align}
and when it is finite and continuous, as it is the case here, the model is QMD \cite[Th 12.2.1]{lehmann2005testing}.
 
When a model is QMD then, under additional mild assumptions, the MLE behaves in a standard way: {\em If $\hat\theta_n$ denotes the MLE, then $\sqrt{n} (\hat\theta_n - \theta)$ is asymptotically normal with zero mean and variance $1/I_\theta$ under $\theta$.}
These additional assumptions are fulfilled, for example, when $(x,y) \mapsto \sup_\theta\, \dot p_\theta(x, y)^2/p_\theta(x,y)$ is square integrable \cite[Th 5.39]{van2000asymptotic}.
If the M-estimator is the MLE, meaning when $\phi(y) \propto \exp(-\loss(y))$, this is the case under the conditions of \thmref{normal}.

It turns out this behavior is best possible asymptotically as we describe next. 
In the present context, we say that an estimator $\hat\theta_n$ is {\em locally asymptotically minimax (LAM)} at some $\theta_0$ if 
\begin{equation}
\lim_{\delta \to 0} \liminf_{n \to \infty} \sup_{|\theta - \theta_0| < \delta} \E_\theta \big[\sqrt{n} |\hat\theta_n - \theta|\big]
\end{equation}
achieves the minimum possible value among all estimators.
Here $\E_\theta$ denotes the expectation with respect to $p_\theta$.
We say that an estimator is LAM if it is LAM at every $\theta_0$.
And, indeed, the MLE is LAM under the same conditions, meaning those of \cite[Th 5.39]{van2000asymptotic}, which again hold in the context of \thmref{normal}.
Hence, we may conclude the following.

\begin{theorem}%[Local asymptotic minimaxity]
\label{thm:LAM}
Under the conditions of \thmref{normal}, and assuming in addition that $\phi(y) \propto \exp(-\loss(y))$, the M-estimator is LAM. 
\end{theorem}

\subsection{Relative efficiency}
\label{sec:ARE}

Although the Huber estimator and the least absolute-value estimator can be motivated as maximum likelihood estimators in their own right, they are often considered as robust compromises to the otherwise preferred least squares estimator. The rationale behind this is in general flimsy, as it amounts to assuming that the noise distribution is Gaussian, as in that case the least squares estimator is the MLE and thus LAM. In signal processing applications, however, the normal assumption can be justified. In any case, this is the perspective we adopt. 

We assume therefore that the noise distribution is Gaussian, making the least squares estimator the gold standard for asymptotic comparisons.  In that context, if $\hat\theta_n$ is another estimator such that $\sqrt{n}(\hat\theta_n - \theta)$ is asymptotically normal with mean 0 and variance $\tau^2$ under $\theta$, then its relative efficiency with respect to the MLE is 
\begin{equation}
\frac{\tau^2}{\tau^2_{\rm mle}},
\end{equation}
and gives, asymptotically, how many more samples this estimator requires in order to achieve the same precision as the MLE in terms of the width of asymptotically-valid confidence intervals at an arbitrary level of confidence. 
In our context, $\tau_{\rm mle}$ is obtained from \eqref{tau} with \eqref{tau_L2}.

Using the expression \eqref{tau_huber}, the relative efficiency of the Huber estimator with parameter $c$ is 
\begin{equation}
\frac{\E[Z^2 \wedge c^2]}{\E[Z^2] \P(|Z| \le c)^2}.
\end{equation}
As anticipated, this is greater than 1 for all values of $c$ and converges to 1 as $c \to \infty$ (as the Huber loss converges to the squared error loss).
Using the expression \eqref{tau_L1}, the relative efficiency of the least absolute-value estimator is 
\begin{equation}
\frac{1}{4 \phi(0)^2 \E[Z^2]}.
\end{equation}
In line with \remref{location_model}, in both cases we recover the relative efficiency for the problem of estimating a normal mean.

\subsection{Finite-sample minimaxity}
\label{sec:minimax}

When the model is smooth, the M-estimator is locally asymptotically minimax when it is the MLE, as we saw in \secref{LAM}, and more generally achieves the optimal $\sqrt{n}$ rate of convergence only losing in the leading constant, as detailed in \secref{ARE}. What about in finite samples?

In the present context, we define the risk of an estimator $\hat\theta_n$ at $\theta$ as 
\begin{equation}
\risk(\hat\theta_n, \theta) := \E_\theta \big[|\hat\theta_n - \theta|\big].
\end{equation}
The {\em minimax risk} is then defined as
\begin{equation}
R_n^* := \inf_{\hat\theta_n} \sup_{\theta} \risk(\hat\theta_n, \theta),
\end{equation}
where the infimum is over all estimators taking in a sample of size $n$.

\begin{lem}%[Information lower bound]
\label{lem:information}
Suppose the basic assumptions are in place. Assume, in addition, that $f$ is Lipschitz, that $\lambda$ is Lipschitz on an open set containing the support of $f(\cdot - \theta^*)$, and that $t \mapsto B(t) := \int \phi(y) \log \phi(y-t) \d y$ has a derivative which is Lipschitz near the origin. Then the minimax risk satisfies $\sqrt{n} R_n^* \gtrsim 1$.
\end{lem}

\begin{proof}
Assume without loss of generality that $\theta^* = 0$.
Recall the notation used for the joint density of $(X,Y)$ set in \eqref{joint_density}. By \cite[Th 2.1 and Th 2.2]{tsybakov2009introduction}, it suffices to prove that ${\sf KL}(p_{\theta}, p_{0}) \lesssim 1/n$ when $\sqrt{n}|\theta| \lesssim 1$.
Here ${\sf KL}$ denotes the Kullback--Leibler divergence.
We have
\begin{align}
{\sf KL}(p_{\theta}, p_{0})
= A(\theta) - A(0),
\end{align}
where
\begin{align}
A(\theta) 
&:= \E[-\log p_{\theta}(X, Y)] \\
&= - \int B(f(x-\theta) - f(x)) \lambda(x) \d x.
\end{align}
Note that $|f(x-\theta) - f(x)| \le |f'|_\infty |\theta|$ and that $\theta$ is taken small. Hence, since $B$ is continuously differentiably in a neighborhood of the origin and $f$ is almost surely differentiable, $\theta \mapsto B(f(x-\theta) - f(x)) \lambda(x)$ is almost surely differentiable near the origin. Its derivative is 
\begin{equation}
f'(x-\theta) B'(f(x-\theta) - f(x)) \lambda(x),
\end{equation}
which is bounded since $f'$ and $\lambda$ are.
Also, $B'$ is continuous and therefore locally bounded.
Hence, by dominated convergence, $A$ is differentiable with
\begin{align}
A'(\theta)
&= \int f'(x-\theta) B'(f(x-\theta) - f(x)) \lambda(x) \d x \\
&= \int f'(x) B'(f(x) - f(x+\theta)) \lambda(x+\theta) \d x,
\end{align}
after a change of variable.
We also have that $\theta \mapsto f'(x) B'(f(x) - f(x+\theta)) \lambda(x+\theta)$ is almost surely differentiable near the origin, with derivative 
\begin{equation}
f'(x) \big\{- f'(x+\theta) B''(f(x) - f(x+\theta)) \lambda(x+\theta) + B'(f(x) - f(x+\theta)) \lambda'(x+\theta)\big\},
\end{equation}
which is bounded for similar reasons.
Hence, by dominated convergence, $A$ is twice differentiable with bounded second derivative.
Note that $A'(0) = B'(0) \int f'(x) \lambda(x) \d x$ with $B'(0) = 0$ since $B$ is differentiable and attains its maximum at $0$, the latter because $B(0) - B(t) = \int \phi(y) \log (\phi(y)/\phi(y-t)) \d y$ is the Kullback--Leibler divergence from $\phi(\cdot-t)$ to $\phi$. 
We thus obtain 
\begin{equation}
A(\theta) - A(0) \le \tfrac12 |A''|_\infty \theta^2.
\end{equation}
This in turn implies that ${\sf KL}(p_{\theta}, p_{0}) \le C \theta^2$ for some $C > 0$. And $C \theta^2 \le 1/n$ when $\sqrt{n} |\theta| \le 1/\sqrt{C}$.
\end{proof}

Hence, when the model is smooth, the minimax rate of convergence is also $\sqrt{n}$. And this rate is achieved by the M-estimator under essentially the same conditions.

\begin{cor}[Minimaxity]
Assume that the conditions of \thmref{normal} or \thmref{normal_L1} hold, as well as those of \lemref{information}. Then the M-estimator achieves the minimax convergence rate.
\end{cor}

\begin{proof}
It turns out that under the conditions of \thmref{normal} or \thmref{normal_L1}, the M-estimator satisfies \cite[Cor 5.53]{van2000asymptotic}\footnote{ The statement of \cite[Cor 5.53]{van2000asymptotic} says that $\sqrt{n} (\hat\theta_n - \theta)$ is bounded in probability under $p_\theta$, but this is done via proving that the conditions of Th 5.52 are fulfilled, and although Th 5.52 provides a bound in probability, a simple modification of the arguments underlying this result yield a bound in expectation under the same exact conditions. That bound in expectation gives \eqref{moment_bound} in the present context.} 
\begin{equation}
\label{moment_bound}
\E_{\theta}[\sqrt{n} |\hat\theta_n - \theta|] = O(1).
\end{equation}
(Note that this is {\em not} a consequence of the normal limit established in \thmref{normal} or \thmref{normal_L1}.)
In our case, we can bound the expectation in \eqref{moment_bound} uniformly over $\theta$. 
This follows trivially from the arguments provided in the proof of \cite[Cor 5.53]{van2000asymptotic} and further details are omitted.
Hence, $\sup_\theta \risk(\hat\theta_n, \theta) = O(1/\sqrt{n})$, and in particular, $\hat\theta_n$ achieves the minimax convergence rate established in \lemref{information}. 
\end{proof}

\section{Non-smooth setting}
\label{sec:non-smooth}

When $f$ is not Lipschitz, the situation can be nonstandard. We saw in \thmref{consistent} that the M-estimator remains consistent under the basic assumptions, but it turns out that it can have a consistency rate other than $\sqrt{n}$. 
Because nonstandard, the analysis is a little bit more involved, yet as it turns out all the tools we need are again available in \citep{van2000asymptotic}.
The model is no longer smooth and we do not go into questions of local asymptotic minimaxity or relative efficiency. However, it turns out that the M-estimator remains rate-minimax.

Although other situations may be of interest, for the sake of concreteness we will consider two cases (and some sub-cases).
In the first case, we take $f$ to be $\alpha$-Holder for some $0 < \alpha \le 1$, meaning 
\begin{equation}
\label{Holder}
H := \sup_{x_1, x_2} \frac{|f(x_1) - f(x_2)|}{|x_1 - x_2|^\alpha} < \infty.
\end{equation}
When this holds for $\alpha = 1$ the function is simply Lipschitz.
In the second case, we take $f$ to be piecewise $\alpha$-Holder. We will let $\cD$ denote the discontinuity set of $f$ which, remember, we are assuming is finite.

Everywhere, the basic assumptions are in place.

\subsection{Preliminaries}
We start with some preliminary results from which everything else follows. The following function will play an important role
\begin{equation}
\label{Delta}
\Delta(\theta_1, \theta_2) = \int (f(x-\theta_1) - f(x-\theta_2))^2 \lambda(x) \d x.
\end{equation}
We will also use the following notation
\begin{equation}
\label{m_tilde}
\tilde m_\theta(x, z) = \loss(z + f(x-\theta^*) - f(x-\theta)),
\end{equation}
which is useful because $m_\theta(x,y) = \tilde m_\theta(x,y-f(x-\theta^*))$.

The first result is an expansion of the function $M$ of \eqref{M}.

\begin{lem}
\label{lem:Delta}
If $\loss$ has a bounded and almost everywhere continuous second derivative, then
\begin{equation}
\label{M_expansion}
M(\theta) = M(\theta^*) + A(\theta) \Delta(\theta, \theta^*),
\end{equation}
where $A$ is bounded and satisfies $A(\theta) \to C_0 := \tfrac12 \int L'' \phi$ as $\theta \to \theta^*$.
If $\loss$ is the absolute-value loss, and $\phi$ is locally bounded and continuous at 0, then this continues to holds, except that $C_0 := \phi(0)$.
\end{lem}

\begin{proof}
Assume $\theta^* = 0$ without loss of generality and let $\Delta(\theta)$ be short for $\Delta(\theta, \theta^*)$.
We start with the first situation. 
Motivated by the fact that $f(x) - f(x-\theta)$ tends to be small when $\theta$ is small, we derive
\begin{align}
\label{m_order_2}
\tilde m_\theta(x, z)
&= \loss(z + f(x) - f(x-\theta)) \\
&= \loss(z) + \loss'(z) (f(x) - f(x-\theta)) + \tfrac12 \loss''(z + \zeta(z,x,\theta)) (f(x) - f(x-\theta))^2,
\end{align}
for some $|\zeta(z,x,\theta)| \le |f(x) - f(x-\theta)|$.
Integrating over $x$ and $z$, we obtain
\begin{align*}
M(\theta)
&= M(0) + \int \int_\bbR \tfrac12 \loss''(z + \zeta(z,x,\theta)) (f(x) - f(x-\theta))^2 \phi(z) \lambda(x) \d z \d x \\
&= M(0) + C_0 \Delta(\theta) 
\pm \int \int_\bbR \gamma(z, |f(x)-f(x-\theta)|) (f(x) - f(x-\theta))^2 \phi(z) \lambda(x) \d z \d x,
\end{align*}
where we have used the fact that $\int_\bbR \loss'(z) \phi(z) \d z = 0$, and where we used the notation
\begin{equation}
\label{gamma}
\gamma(z, a) := \sup_{|b| \le a} |\loss''(z+b) - \loss''(z)|.
\end{equation}
By our assumptions on $\loss''$, $\gamma(z,a) \to 0$ when $a \to 0$ for almost every $z$. In addition, $\gamma$ is uniformly bounded. Hence, by dominated convergence, the last integral is $o(\Delta(\theta))$ as $\theta \to 0$.

We now turn to the situation where $\loss$ is the absolute-value loss. We first note that
\begin{equation}
\int_\bbR |z+a| \phi(z) \d z
= 2 \Phi_2(z) + a,
\end{equation}
where
\begin{equation}
\Phi_2(z) := \int_{-\infty}^z \Phi(u) \d u, \quad
\Phi(z) := \int_{-\infty}^z \phi(u) \d u.
\end{equation}
Note that $\Phi_2' = \Phi$ and $\Phi_2'' = \phi$.
Using this, we derive
\begin{align*}
M(\theta)
&= M(0) 
+ \int \big\{2 \big[\Phi_2(f(x) - f(x-\theta)) - \Phi_2(0)\big] + f(x)-f(x-\theta)\big\} \d x \\
&= M(0) + C_0 \Delta(\theta)
+ \int \phi(\zeta(x,\theta)) (f(x) - f(x-\theta))^2 \lambda(x) \d x,
\end{align*}
for some $|\zeta(x,\theta)| \le |f(x) - f(x-\theta)|$.
Note that $|\zeta(x,\theta)| \le 2 |f|_\infty$, so that $\phi(\zeta(x,\theta))$ is bounded. Moreover, by the fact that $\phi$ is continuous at 0, and that $f$ is continuous almost everywhere, it holds that $\phi(\zeta(x,\theta)) \to 0$ as $\theta \to 0$ for almost all $x$. We can thus apply dominated convergence to find that the last integral is $o(\Delta(\theta))$ as $\theta \to 0$.
\end{proof}

Another notion that will be important is that of an envelope. We say that a function $\env$ is an envelope for a function class $\cM$ if any function in $\cM$ is, in absolute value, pointwise bounded by $\env$, or in formula, $|g(x)| \le V(x)$ for all $x$ and all $g \in \cM$.
Below,
\begin{equation}
\label{M_delta}
\cM_\delta := \big\{m_\theta - m_{\theta^*} : |\theta - \theta^*| \le \delta\big\},
\end{equation}
and
\begin{equation}
\label{D_delta}
\cD(\delta) := \{x = d + t: d \in \cD, |t| \le \delta\},
\end{equation}
where $\delta$ is thought of as being small.

\begin{lem}
\label{lem:envelope}
Assume either that $\loss$ has a Lipschitz derivative and that the noise distribution has finite second moment; or that $\loss$ itself is Lipschitz.
If $f$ is $\alpha$-Holder, then $\cM_\delta$ admits an envelope of the form $\env(x,y) := \overline{m}(z) \delta^\alpha$ where $z := y - f(x-\theta^*)$ and $\int \overline{m}^2 \phi < \infty$.
If $f$ is only piecewise $\alpha$-Holder, then the same is true but with $\env(x,y) := \overline{m}(z) (\delta^\alpha + \IND{x \in \cD(\delta)})$.
\end{lem}

\begin{proof}
Assume $\theta^* = 0$ without loss of generality, so that $\theta$ below satisfies $|\theta|\le \delta$.
For $x, y$, we let $z = y - f(x)$.

{\em First, assume that $f$ is $\alpha$-Holder} and let $H$ be defined as in \eqref{Holder}.
If the loss has a Lipschitz derivative, we have
\begin{align*}
|m_\theta(x,y) - m_0(x,y)|
&= |\loss(z+f(x)-f(x-\theta)) - \loss(z)| \\
&\le |\loss'(z + \zeta)| \times |f(x) - f(x-\theta)| \\
&\le |\loss''|_\infty (|z|+2|f|_\infty) \times H |\theta|^\alpha,
\end{align*}
for some $|\zeta| \le |f(x) - f(x-\theta)| \le 2|f|_\infty$.
We then conclude with $|\theta| \le \delta$.
If the loss is Lipshitz, we have
\begin{align*}
|m_\theta(x,y) - m_0(x,y)|
&\le |L'|_\infty |f(x) - f(x-\theta)| \\
&\le |L'|_\infty H |\theta|^\alpha,
\end{align*}
and we conclude in the same way.

{\em Next, assume that $f$ is piecewise $\alpha$-Holder.}
In this situation, let $H$ be defined as in \eqref{Holder}, except away from discontinuities.
Based on what we just did, it suffices to show that 
\begin{equation}
|f(x) - f(x-\theta)| \le C_1(\delta^\alpha + \IND{x \in \cD(\delta)}),
\end{equation}
when $|\theta| \le \delta$ for some constant $C_1$. 
%Since $f$ is bounded, it suffices to prove this for $\delta$ small enough, and we take $\delta$ to be smaller than the distance between two discontinuities of $f$. 
Indeed, either there are no discontinuity point between $x$ and $x-\theta$, in which case $|f(x) - f(x-\theta)| \le H |\theta|^\alpha$; or there is a discontinuity point, say $d$, between $x$ and $x-\theta$, so that $|x-d| \le |\theta| \le \delta$, implying that $|f(x) - f(x-\theta)| \le 2 |f|_\infty \IND{x \in \cD(\delta)}$.
\end{proof}

The next result is a complexity bound for $\cM_\delta$.
The complexity is in terms of bracketing numbers \cite[Sec 19.2]{van2000asymptotic}.
Two functions $g_1, g_2$ such that $g_1 \le g_2$ pointwise define a bracket made of all functions $g$ such that $g_1 \le g \le g_2$. It is said to be an $\eps$-bracket with respect to $L^2(\mu)$, for a positive measure $\mu$, if $\int (g_2 - g_1)^2 \d\mu \le \eps^2$. 
Given a class of functions $\cM$, its $\eps$-bracketing number with respect to $L^2(\mu)$ is the minimum number of $\eps$-brackets needed to cover $\cM$ (meaning to include any function in that class).
In our context, the measure is the underlying sample distribution, meaning $p_{\theta*}$ in the notation introduced in \eqref{joint_density}. We let $N_\delta(\eps)$ denote the $\eps$-bracketing number of $\cM_\delta$ with respect to $L^2(p_{\theta*})$. 

\begin{lem}
\label{lem:bracketing}
There is a constant $C>0$ such that the following holds. Assume the loss is as in \lemref{envelope}.
If $f$ is $\alpha$-Holder, then $N_\delta(\eps) \le C \delta \eps^{-1/\alpha}$. 
If $f$ is piecewise $\alpha$-Holder, and $\lambda$ is bounded, then $N_\delta(\eps) \le C \delta \eps^{-1/(\alpha \wedge 1/2)}$. 
\end{lem}

\begin{proof}
Assume $\theta^* = 0$ without loss of generality, so that any value of the parameter below is in $[-\delta, \delta]$.
Let $\theta_j = j \delta/k$ for $j = -k, \dots, k$.
For $x, y$, we let $z = y - f(x)$ below.
We rely on the proof of \lemref{envelope}.

First, assume that $f$ is $\alpha$-Holder.
We have
\begin{align*}
|m_\theta(x,y) - m_{\theta_0}(x,y)|
&\le \overline{m}(z) |\theta - \theta_0|^\alpha,
\end{align*}
where $\overline{m}(z)$ is square integrable.
For a given $\theta$, let $j$ be such that $|\theta - \theta_j| \le \delta/k$. Then
\begin{align*}
|m_\theta(x,y) - m_{\theta_j}(x,y)|
&\le \overline{m}(z) |\theta - \theta_j|^\alpha
\le \overline{m}(z) (\delta/k)^\alpha,
\end{align*}
implying that
\begin{equation}
A_j(x,y) \le m_\theta(x,y) \le B_j(x,y),
\end{equation}
where
\begin{align*}
A_j(x,y) 
&:= 
m_{\theta_j}(x,y) - \overline{m}(z) (\delta/k)^\alpha \\
B_j(x,y)
&:= m_{\theta_j}(x,y) + \overline{m}(z) (\delta/k)^\alpha.
\end{align*}
Note that $B_j-A_j$ has $L^2$ norm of order $(\delta/k)^\alpha$ since $\overline{m}$ is square integrable. Choosing $k \ge C_2 \delta/\eps^{1/\alpha}$ for a large enough $C_2$ makes it an $\eps$-bracket. And these $k$ brackets together cover $\cM_\delta$. 

Next, assume that $f$ is piecewise $\alpha$-Holder.
Assume for expediency that $f$ has a single discontinuity, and let $d$ denote the location of that discontinuity.
Following the corresponding arguments in \lemref{envelope}, we can prove that 
\begin{align*}
|m_\theta(x,y) - m_{\theta_0}(x,y)|
&\le \overline{m}(z) (|\theta - \theta_0|^\alpha + \IND{\theta_0 \le x-d \le \theta}),
\end{align*}
where $\overline{m}(z)$ is square integrable.
For a given $\theta$, let $j$ be such that $\theta_j \le \theta \le \theta_{j+1}$. Then
\begin{equation}
A_j(x,y) \le m_\theta(x,y) \le B_j(x,y),
\end{equation}
where
\begin{align*}
A_j(x,y) 
&:= m_{\theta_j}(x,y) - \overline{m}(z) ((\delta/k)^\alpha + \IND{\theta_j \le x-d \le \theta_{j+1}}) \\
B_j(x,y)
&:= m_{\theta_j}(x,y) + \overline{m}(z) ((\delta/k)^\alpha + \IND{\theta_j \le x-d \le \theta_{j+1}}).
\end{align*}
The difference satisfies
\begin{align*}
&\int (B_j(x,y) - A_j(x,y))^2 \phi(y-f(x)) \lambda(x) \d y \d x \\
&\le 4 \int \overline{m}(z)^2 \phi(z) \d z \times 2 \int ((\delta/k)^{2 \alpha} + \IND{\theta_j \le x-d \le \theta_{j+1}}) \lambda(x) \d x \\
&\lesssim (\delta/k)^{2 \alpha} + |\lambda|_\infty (\delta/k),
\end{align*}
and so has $L^2$ norm of order $(\delta/k)^\alpha + (\delta/k)^{1/2} \asymp (\delta/k)^{\alpha \wedge 1/2}$, uniformly in $j$. Therefore, choosing $k \ge C_3 \delta/\eps^{1/(\alpha \wedge 1/2)}$ for a large enough $C_3$ ensures that $B_j-A_j$ has $L^2$ norm bounded by $\eps$. And the $k$ brackets that these pairs of functions define as $j$ ranges through $\{-k, \dots, k\}$, together, cover the space $\cM_\delta$. 
\end{proof}

From here proceed in reverse order compared to \secref{smooth}: We first study the rate of convergence and then discuss the limit distribution.

\subsection{Rate of convergence}

The bracketing integral of a function class $\cM$ with respect to $L^2(\mu)$ is defined as the integral of the square root of the logarithm of the corresponding $\eps$-bracketing number as a function of $\eps$. In particular, we introduce the bracketing integral of $\cM_\delta$, denoted
\begin{equation}
J_\delta(t) := \int_0^t \sqrt{\log N_\delta(\eps)} \d \eps.
\end{equation}

A simple adaptation of the arguments underlying \cite[Th 5.52]{van2000asymptotic}, in combination with \cite[Cor 19.35]{van2000asymptotic}, gives the following result.

\begin{lem}
\label{lem:rate}
Suppose there are constants $a > b \ge 0$ and $C>0$ such that, for $\delta > 0$ small enough, 
\begin{equation}
\label{Delta_cond}
\inf_{|\theta - \theta^*| \ge \delta} M(\theta) - M(\theta^*) \ge \delta^a/C,
\end{equation}
and
\begin{equation}
\label{J_cond}
J_\delta(\infty) \le C \delta^b.
\end{equation}
Then the M-estimator satisfies
\begin{equation}
\label{moment_bound_nonsmooth}
\E\big[n^{\frac1{2(a-b)}} |\hat\theta_n - \theta^*|\big] = O(1).
\end{equation}
\end{lem}

With the preceding lemmas, we are able to establish an upper bound on the rate of convergence of the M-estimator.

\begin{theorem}
\label{thm:rate}
Suppose that either $\loss$ has a bounded and almost everywhere continuous second derivative; or is the absolute-value loss and $\phi$ is locally bounded and continuous at 0.
\begin{itemize}
\item Suppose $f$ is $\alpha$-Holder and 
\begin{equation}
\label{Delta_asymp}
\sup_{|\theta-\theta^*| \le \delta} \Delta(\theta, \theta^*) \asymp \delta^{2\alpha}.
\end{equation}
Then the M-estimator is $r_n$-consistent with $r_n = n^{1/2\alpha}$.
\item Suppose $f$ is discontinuous and piecewise $\alpha$-Holder with $\alpha \ge 1/2$, and $\lambda$ is bounded and continuous at the points of discontinuity of $f(\cdot - \theta^*)$, and strictly positive at one or more of these locations. 
Then the M-estimator is $r_n$-consistent with $r_n = n$.
\end{itemize}
\end{theorem}

\begin{proof}
Consider the first situation. 
We apply \lemref{rate}.
By \lemref{Delta} and \eqref{Delta_asymp}, the condition \eqref{Delta_cond} is satisfied with $a = 2\alpha$.
Hence, it suffices to show that \eqref{J_cond} holds with $b = \alpha$.
Indeed, by \lemref{envelope}, we have $J_\delta(\infty) = J_\delta(C \delta^\alpha)$ for $C$ large enough, and by \lemref{bracketing}, we have
\begin{equation}
\label{J_Holder}
J_\delta(C \delta^\alpha) 
\le \int_0^{C \delta^\alpha} \sqrt{\log (C \delta/\eps^{1/\alpha})} \d \eps
\asymp \delta^\alpha.
\end{equation}

We now turn to the second situation.
By of \lemref{envelope}, we have $J_\delta(\infty) = J_\delta(C \delta^{1/2})$ for $C$ large enough, and by \lemref{bracketing}, we have
\begin{equation}
\label{J_1/2}
J_\delta(C \delta^{1/2}) 
\le \int_0^{C \delta^{1/2}} \sqrt{\log (C \delta/\eps^{2})} \d \eps
\asymp \delta^{1/2},
\end{equation}
so that \eqref{J_cond} holds with $b = 1/2$.
It thus suffices to show that \eqref{Delta_cond} holds with $a = 1$.
Assume for simplicity that $f$ has only one discontinuity, and therefore of the form $f(x) = f_1(x) \IND{x < d} + f_2(x) \IND{x \ge d}$ with $f_1$ and $f_2$ being $\alpha$-Holder. For example, consider a situation where $\theta > 0 = \theta^*$. 
We have
\begin{equation}
\Delta(\theta, \theta^*)
= \int_{-\infty}^d + \int_d^{d+\theta} + \int_{d+\theta}^\infty \quad (f(x) - f(x-\theta))^2 \lambda(x) \d x.
\end{equation}
(In what follows, remember that $\lambda$ is compactly supported.)
In the 1st and 3rd integral, $|f(x) - f(x-\theta)| \le H |\theta|^\alpha$ since $f$ does not have a discontinuity over the corresponding ranges, $x\in [0,d)$ and $x \in (d+\theta, 1]$, respectively. Hence, these two integrals are of order $|\theta|^{2\alpha}$, which is at most of order $O(\theta)$ since $\alpha \ge 1/2$. For the 2nd or middle integral, we use the fact that $(f(x) - f(x-\theta))\lambda(x) = (f_2(x) - f_1(x-\theta))^2 \lambda(x) \to (f_2(d) - f_1(d))^2 \lambda(d)$ when $x \in [d, d+\theta]$ and $\theta \to 0$, so that the integral is $\sim |\theta| (f_2(d) - f_1(d))^2 \lambda(d)$ by dominated convergence.
More generally, if $\cD$ denotes the points of discontinuity of $f$, extending these arguments gives
\begin{equation}
\label{Delta_discont}
\Delta(\theta, \theta^*) \sim |\theta| D, \quad \text{as } \theta \to 0 = \theta^*, \quad \text{where } D := \sum_{d \in \cD} (f(d^+) - f(d^-))^2 \lambda(d)\ .
\end{equation}
From this we conclude.
\end{proof}

%\begin{exa}
%\label{exa:Holder_example}
%Take $\lambda$ to be the uniform distribution on some interval. Examples of templates that satisfy \eqref{Delta_asymp} include $f$ which coincides with $g(x) := \sum_{k \ge 0} c_k \cos(2\pi k x)$ with $c_k \asymp k^{-\alpha-1/2}$ on some bounded interval $I$ and is otherwise a compactly supported Lipschitz extension of $g$ outside that interval. 
%%Indeed, in that case we have $f(x) = \frac12 \sum_{k \in \bbZ} c_k \exp(i 2\pi k x) + \frac12 c_0$, and by Parseval--Plancherel's identify and some simplifications, we arrive at $\Delta(\theta, \theta^*) = 2 \sum_{k \ge 1} c_k^2 \sin(\pi k \theta)^2$. By the property assumed of $(c_k)$, we then get
%%\begin{align*}
%%\Delta(\theta, \theta^*)
%%&\asymp \sum_{k \ge 1} k^{-2\alpha-1} \sin(\pi k \theta)^2 \\
%%&= \pi^{2\alpha} \theta^{2\alpha} \times \pi\theta \sum_{k \ge 1} g(\pi \theta k), \quad g(x) := x^{-2\alpha-1} \sin(x)^2.
%%\end{align*} 
%%The second factor on the left-hand side is a Riemann sum, and although it is an infinite sum, it can be shown to converge to $\int_0^\infty x^{-2\alpha-1} \sin(x)^2 \d x$ as $\theta \to 0$. Therefore, we have established that \eqref{Delta_asymp} holds in this example.
%\end{exa}

Of course, we have only obtained an upper bound on the consistency rate. But as we will see in the next subsection, that rate is sharp. 

\begin{rem}
We note that \eqref{Delta_asymp} is not automatically satisfied even when the function $f$ is exactly $\alpha$-Holder (and not $(\alpha+\eta)$-Holder for any $\eta > 0$). Indeed, consider the case where $f$ coincides near the origin with $x \mapsto |x|^\alpha$ for some fixed $0 < \alpha < 1$, and is otherwise Lipschitz. Then such an $f$ is exactly $\alpha$-Holder, and yet, it can be shown that 
\begin{equation}
\sup_{|\theta-\theta^*| \le \delta} \Delta(\theta, \theta^*) \asymp 
\begin{cases}
\delta^2 & \text{if } \alpha > 1/2, \\
\delta^2 \log(1/\delta) & \text{if } \alpha = 1/2, \\
\delta^{2\alpha+1} & \text{if } \alpha < 1/2.
\end{cases}
\end{equation}
%And since the function is $\alpha$-Holder, \eqref{J_cond} holds with $b = \alpha$ exactly as in \eqref{J_Holder}. We may therefore apply \lemref{rate} to obtain a rate of $n^{1/(4-2\alpha)}$ when $\alpha > 1/2$ and a rate of $n^{1/(2+2\alpha)}$. 
\end{rem}

\subsection{Minimaxity}

We can easily adapt the arguments underlying the information bound stated in \lemref{information} to the settings that interest us in the present section where $f$ is not Lipschitz. We do so and obtain the following.

\begin{lem}
\label{lem:information_nonsmooth}
Suppose the basic assumptions are in place. Assume, in addition, that $\phi$ has finite first moment and that $B$ is as in \lemref{information}.  If $r_n$ is such that $\sup\{\Delta(\theta, \theta^*) : r_n |\theta-\theta^*| \le 1\} \lesssim 1/n$, the minimax rate satisfies $r_n R_n^* \gtrsim 1$.
\end{lem}

\begin{proof}
Assume without loss of generality that $\theta^* = 0$ and let $\Delta(\theta)$ be short for $\Delta(\theta, \theta^*)$.
In the notation introduced in the proof of \lemref{information}, it suffices to prove that $A(\theta) - A(0) \lesssim 1/n$ when $r_n |\theta| \le 1$.
We saw in the proof of that lemma that $B'(0) = 0$. This together with the fact that $B'$ is Lipschitz near the origin implies that there is a positive constant $C_1$ such that $|B(t) - B(0)| \le C_1 t^2$ when $|t|$ is small enough. 
We use that to derive
\begin{align}
A(\theta) - A(0)
&= \int \big\{B(0) - B(f(x-\theta)-f(x))\big\} \lambda(x) \d x \\
&\le C_1 \int (f(x-\theta)-f(x))^2 \lambda(x) \d x \\
&= C_1 \Delta(\theta). \label{A_theta_A_0}
\end{align}
From this we conclude.
\end{proof}

\begin{cor}
In the context of \lemref{information_nonsmooth}, if $f$ is $\alpha$-Holder, then $n^{1/2\alpha} R_n^* \gtrsim 1$; if instead $f$ is discontinuous and piecewise $\alpha$-Holder with $\alpha \ge 1/2$, then $n R_n^* \gtrsim 1$.
\end{cor}

\begin{proof}
As we saw in the proof of \thmref{rate}, if $f$ is $\alpha$-Holder, $\Delta(\theta, \theta^*) \lesssim |\theta-\theta^*|^{2\alpha}$; while if instead $f$ is discontinuous and piecewise $\alpha$-Holder with $\alpha \ge 1/2$, $\Delta(\theta,\theta^*) \lesssim|\theta-\theta^*|$. 
We then apply \lemref{information_nonsmooth} to conclude.  
\end{proof}

\subsection{Limit distribution}

To establish a limit distribution, we follow the standard strategy which starts by showing that a properly normalized version of the empirical process $\widehat{M}_n$ converges to a Gaussian process, and is followed by an application of the argmax continuous mapping theorem.

The following is a direct consequence of \cite[Th 19.28]{van2000asymptotic}.

\begin{lem}
\label{lem:gaussian_process}
Consider a set of functions on some Euclidean space and let $\mu$ denote a Borel probability measure.  Consider a class of functions $\cW_{n,T} := \{w_{n,t} : n \ge 1, t \in [-T,T]\}$ satisfying the following:
\begin{enumerate}[label=(\roman*)]
\item The class has a square integrable envelope.
\item The limit $g(t) := \lim_n \sqrt{n} \int w_{n,t} \d\mu$ is well-defined for all $t$.
\item There is a function $\omega_1$ with $\lim_{t\to 0+} \omega_1(t) = 0$ and a sequence $\eta_n \to 0$ such that 
\begin{equation}
\label{argmax_2}
v_n(s,t) := \int (w_{n,s} - w_{n,t})^2 \d \mu \le \omega_1(|s-t|) + \eta_n; 
\end{equation}
\item The limit $v(s,t) := \lim_n v_n(s,t)$ is well-defined for all $s,t$.
\item The class has bracketing integral $J_{n,T}(\delta) \le \omega_2(\delta)$ for some function $\omega_2$ with $\lim_{t\to 0+} \omega_2(t) = 0$.
\end{enumerate}
Then $W_{n,t} := \sum_{i=1}^n w_{n,t}$ converges weakly to $g(t) + G_t$ on $[-T, T]$, where $G_t$ denotes the centered Gaussian process with variance function $v$. Note that $v(s,t) \le \omega_1(|s-t|)$.
\end{lem}

The following is a special case of \cite[Cor 5.58]{van2000asymptotic}.
\begin{lem}
\label{lem:argmax}
Suppose that a sequence of processes $W_{n}$ defined on the real line converges weakly in the uniform topology on every bounded interval to a process $W$ with continuous sample paths each having a unique minimum point $h$ (almost surely). If $h_n$ minimizes $W_n$, and $(h_n)$ is uniformly tight, then $h_n$ converges weakly to $h$.
\end{lem}

We apply these two lemmas to obtain a limit distribution for the M-estimator when the template is $\alpha$-Holder. For expediency, we only consider smoother losses.

\begin{theorem}
\label{thm:limit_holder}
Suppose that $\loss$ has a bounded and almost everywhere continuous second derivative. Assume that $f$ is $\alpha$-Holder and such that
\begin{equation}
r^{2\alpha} \Delta(\theta^* + s/r, \theta^* + t/r) \to \Delta_0(s,t), \quad r \to \infty,
\end{equation}
where $\Delta_0$ is a continuous function such that $\Delta_0(s,t) \ne 0$ when $s \ne t$.  Then $n^{1/2\alpha} (\hat\theta_n - \theta^*)$ converges weakly to $\argmin_t \{g(t)+G_t\}$ where $g(t) := C_0 \Delta_0(t,0)$ with $C_0 := \tfrac12 \int L'' \phi$ and $G_t$ is the centered Gaussian process on the real line with variance function $C_1 \Delta_0(s,t)$ with $C_1 := \int (\loss')^2 \phi$.
\end{theorem}

\begin{proof}
As usual, assume that $\theta^* = 0$ without loss of generality, and let $z = y - f(x)$.

We first prove that the process $\sqrt{n} (\widehat{M}_n(t/r_n) - \widehat{M}_n(0))$ --- where $r_n = n^{1/2\alpha}$ is the rate of convergence established in \thmref{rate} --- converges weakly to $g(t) + G_t$. For this it is enough to prove that it converges on every interval of the form $[-T,T]$, and we do so by applying \lemref{gaussian_process} with $w_{n,t} := \sqrt{n} (m_{t/r_n} - m_{0})$. Note that $\mu = p_{0}$, the distribution of $(X,Y)$ under $\theta = 0$.
\lemref{envelope} and a rescaling argument gives {\em (i)}. But to be sure, using \lemref{Delta}, we have
\begin{align}
|w_{n,t}(x,y)|
&\le \sqrt{n} \big\{|\loss'(z)| |f(x) - f(x-t/r_n)| + \tfrac12 |\loss''|_\infty (f(x) - f(x-t/r_n))^2\big\} \\
&\le \sqrt{n} \big\{|\loss'(z)| H (T/r_n)^\alpha + \tfrac12 |\loss''|_\infty (H (T/r_n)^\alpha)^2\big\} \\
&\le C (\loss'(|z|) + 1) =: \overline{w}(x,y).
\end{align}
We then conclude with the fact that $\overline{w}$ is square integrable by the usual assumptions.

For {\em (ii)}, using \lemref{Delta}, we have
\begin{align}
\sqrt{n} \int w_{n,t} \d \mu
&= \sqrt{n} \times \sqrt{n} (M(t/r_n) - M(0)) \\
&\sim n C_0 \Delta(t/r_n, 0), \quad n \to \infty, \\
&= n C_0 r_n^{-2\alpha} r_n^{2\alpha} \Delta(t/r_n, 0) \\
&\to C_0 \Delta_0(t, 0) = g(t), \quad n \to \infty.
\end{align}

For {\em (iii)}, using \eqref{m_order_2}, we have for $s, t \in [-T,T]$,
\begin{align}
\int (w_{n,s} - w_{n,t})^2 \d \mu 
&= n \int (m_{s/r_n} - m_{t/r_n})^2 \d p_0 \\
&\le n \Big[2 C_1 \Delta(s/r_n, t/r_n) + |\loss''|_\infty^2 \{\Delta(s/r_n, 0)^2 + \Delta(t/r_n, 0)^2\}\Big] \\
&\le n \Big[2 C_1 (H |s/r_n - t/r_n|^\alpha)^2 + |\loss''|_\infty^2 \{(H |s/r_n|^\alpha)^4 + (H |t/r_n|^\alpha)^4\}\Big] \\
&\le n C \big(|s-t|^{2\alpha}/r_n^{2\alpha} + (T/r_n)^{4\alpha}\big) \\
&= C \big(|s-t|^{2\alpha} + T^{4\alpha}/n\big),
\end{align}
using the fact that $r_n^{2\alpha} = n$.
We conclude that \eqref{argmax_2} holds with $\omega_1(s,t) := C |s-t|^{2\alpha}$ and $\eta_n := C T^{4\alpha}/n$.

For {\em (iv)}, we refine these arguments using dominated convergence, to get for any $s, t \in \bbR$,
\begin{align}
v_n(s,t)
&= n \int (m_{s/r_n} - m_{t/r_n})^2 \d p_0 \\
&\sim n C_1 \Delta(s/r_n, t/r_n), \quad n \to \infty, \\
&= n C_1 r_n^{-2\alpha} r_n^{2\alpha} \Delta(s/r_n, t/r_n) \\
&\to C_1 \Delta_0(s, t) = v(s,t), \quad n \to \infty.
\end{align}

For {\em (v)}, we use \lemref{bracketing} and a rescaling argument to bound the bracketing number of this function class by $C (T/r_n) (\eps/\sqrt{n})^{-1/\alpha} = C T \eps^{-1/\alpha}$, so that its bracketing integral at $\delta$ is bounded by $\omega_2(\delta) := \int_0^\delta \sqrt{\log(CT/\eps^{-1/\alpha})} \d\eps$.

The fact that $\sqrt{n} (\widehat{M}_n(t/r_n) - \widehat{M}_n(0))$ converges weakly to $g(t) + G_t$ is useful to us because $\hat\theta_n = \hat t_n/r_n$ where $\hat t_n = \argmax_t \sqrt{n} (\widehat{M}_n(t/r_n) - \widehat{M}_n(0))$.
To conclude, we only need to show that \lemref{argmax} applies.
On the one hand, the process $g(t) + G_t$ has continuous sample paths. This is because $g$ is continuous --- since $g(t) \propto \Delta_0(t,0)$ and $\Delta_0$ is assumed continuous --- and $G_t$ is Gaussian with variance function $v$ satisfying $v(s,t) \le C |s-t|^{2\alpha}$ and this is enough for $G_t$ to have continuous sample paths according to \cite[Th 7.1]{dudley1967sizes}.
On the other hand, the process has a unique minimum point since $v(s,t) \ne 0$ when $s \ne t$ --- since $v \propto \Delta_0$ and $\Delta_0$ satisfies that property --- which is sufficient by \cite[Lem 2.6]{kim1990cube}.
And we have shown in \thmref{rate} that $\hat t_n$ is uniformly tight.
\end{proof}

%\begin{exa}
%Similar to \exaref{Holder_example}, take $\lambda \equiv 1$ and $\theta^* = 0$ for concreteness, and consider a template function $f$ of the form $f(x) = \sum_{k \ge 0} c_k \cos(2\pi k x)$, this time with $c_k \sim k^{-\alpha-1/2}$ as $k \to \infty$. It can be shown that 
%\begin{align*}
%\Delta(\theta_1, \theta_2) 
%&= \frac12 \sum_{k \ge 1} c_k^2 \big[1 - \cos(2\pi\theta_1 k) - \cos(2\pi\theta_2 k) + \cos(2\pi(\theta_1-\theta_2) k)\big].
%\end{align*}
%Assuming that $c_k = k^{-\alpha-1/2}$ for simplicity, we have
%\begin{align*}
%r^{2\alpha} \Delta(s/r, t/r) 
%&= \pi^{2\alpha} s^{2\alpha} \times \pi (s/r) \sum_{k \ge 1} g(\pi (s/r) k) \\
%&\quad + \pi^{2\alpha} t^{2\alpha} \times \pi (t/r) \sum_{k \ge 1} g(\pi (t/r) k) \\
%&\quad\quad - \pi^{2\alpha} |s-t|^{2\alpha} \times \pi (|s-t|/r) \sum_{k \ge 1} g(\pi (|s-t|/r) k).
%\end{align*}
%where $g(x) := x^{-2\alpha-1} \sin(x)^2$ as before.
%Via a Riemann sum, we can establish \eqref{Delta_argmax} with
%\begin{equation}
%\Delta_0(s,t) := C \big(s^{2\alpha} + t^{2\alpha} - |s-t|^{2\alpha} \big), \quad C := \pi^{2\alpha} \int_0^\infty x^{-2\alpha-1} \sin(x)^2 \d x\ . 
%\end{equation}
%The arguments extend to the case where $c_k \sim k^{-\alpha-1/2}$.
%\end{exa}

\thmref{limit_holder} also applies when $\alpha = 1$, meaning when the template is Lipschitz, and in fact implies \thmref{normal} in that case. 
Thus, we have established that, when the template $f$ is Holder, the M-estimator converges weakly to the minimizer of a Gaussian process which does not depend on the noise distribution except through the constants $C_0$ and $C_1$.
When the template $f$ is discontinuous, the situation is qualitatively different: the limit process does not have a unique minimum point and the M-estimator does not have a limit distribution, and the limit process is far from `universal' but rather depends heavily on the noise distribution.

Instead of the more commonly-used argmax theorem (\lemref{argmax}), which takes place in the uniform topology, we will use the following mild variant of \cite[Th 3]{ferger2004continuous}, which takes place in the Skorohod topology.

\begin{lem}
\label{lem:argmax_extension}
Suppose that a sequence of processes $W_{n}$ defined on the real line converges weakly in the Skorohod topology on every bounded interval to a process $W$ with sample paths each achieving their minimum somewhere in $[\underline{h}, \overline{h}]$, and only there, with $\underline{h}$ and $\overline{h}$ being continuous random variables.
Assume also that $W_n$ and $W$ have right and left limits at every point, and that $\P(W \text{ is continuous at } x) = 1$ for all $x$. 
If $h_n$ minimizes $W_n$, and $(h_n)$ is uniformly tight, then, for every $x$,
\begin{equation*}
\P(\underline{h} \le x) 
\le \liminf_n \P(h_n \le x)
\le \limsup_n \P(h_n \le x)
\le \P(\overline{h} \le x).
\end{equation*}
\end{lem}

It turns out that the limit process is here a marked Poisson process, and with the help of this variant of the argmax continuous mapping theorem, we obtain the following.

\begin{theorem}
\label{thm:limit_discont}
Suppose that either $\loss$ has a bounded and almost everywhere continuous second derivative. Assume that $f$ is discontinuous and piecewise $\alpha$-Holder with $\alpha > 1/2$, and $\lambda$ is bounded and continuous at the points of discontinuity of $f(\cdot - \theta^*)$, and strictly positive at one or more of these locations. 
Then $n (\hat\theta_n - \theta^*)$ dominates $\underline{t}$ and is dominated by $\overline{t}$ asymptotically, where $[\underline{t}, \overline{t}]$ is the closure of the minimum point set of $W = \sum_{d \in \cD} W_d$, where $\cD$ denotes the discontinuity set of $f(\cdot - \theta^*)$ and the $W_d$'s are independent with $W_d$ being the double-sided marked Poisson process on $\bbR$ with intensity $\lambda(d)$ and mark distribution that of $\loss(Z+\delta_d) - \loss(Z)$, where $\delta_d := f(d^+ - \theta^*)-f(d^- - \theta^*)$. 
\end{theorem}

\begin{rem}
The values that $W$ takes at its discontinuity points are irrelevant, and in particular it does not need to be taken c\`adlag as is the norm. In the proof we set things up so that it is, but this is only for convenience.
\end{rem}

\begin{proof}
As usual, assume that $\theta^* = 0$ without loss of generality.
When the meaning of $(x,y)$ is clear from context, we use $z$ as a shorthand for $y-f(x)$.
We assume for convenience that $f$ is c\`aglad, the reverse of c\`adlag, meaning that at every point it is continuous from the left and has a limit from the right. We do is in order for the Poisson processes that follow to be c\`adlag.

We first prove that the process $W_n(t) := n(\widehat{M}_n(t/n) - \widehat{M}_n(0))$ converges weakly to $W(t)$. Note that $n$ is the rate of convergence established in \thmref{rate}. Indeed, take $x$, and also $t > 0$ smaller than the separation between any two discontinuity points of $f$. Two cases are possible.
Either there is $d \in \cD$ such that $x-t \le d < x$, in which case we use the fact that 
\begin{align*}
f(x) - f(x-t) 
&= f(x) - f(d^+) + \delta_d + f(d^-) - f(x-t) \\
&= \pm H (x-d)^\alpha + \delta_d \pm H (d-(x-t))^\alpha \\
&= \delta_d \pm 2 H t^\alpha,
\end{align*}
implying that
\begin{align}
\loss(z + f(x) - f(x-t)) - \loss(z)
&= \loss(z + \delta_d) - \loss(z) \pm \loss'(z + \delta_d) 2 H t^\alpha \pm \tfrac12 |\loss''|_\infty (2 H t^\alpha)^2 \\
&= B_d(z) \pm \overline{m}(z) t^\alpha, \quad B_d(z) := \loss(z + \delta_d) - \loss(z),
\end{align}
where $\overline{m} \ge 0$ and $\int \overline{m}^2 \phi < \infty$. 
Otherwise, 
\begin{align*}
f(x) - f(x-t) 
= \pm H (x-(x-t))^\alpha
= \pm H t^\alpha,
\end{align*}
implying that
\begin{align}
\loss(z + f(x) - f(x-t)) - \loss(z)
&= \loss'(z) (f(x) - f(x-t)) \pm \tfrac12 |\loss''|_\infty (H t^\alpha)^2 \\
&= \loss'(z) (f(x) - f(x-t)) \pm \overline{m}(z) t^{2\alpha},
\end{align}
for a possibly different non-negative, square integrable function $\overline{m}$.  We simply take the pointwise maximum of the two.
Hence, 
\begin{align*}
m_t(x,y) - m_0(x,y) 
&= \loss(z + f(x) - f(x-t)) - \loss(z) \\
&= \sum_{d \in \cD} \Big\{(B_d(z) \pm \overline{m}(z) t^\alpha) \IND{x \in (d, d+t]} \\
&\quad\quad\quad\quad + (\loss'(z) (f(x) - f(x-t)) \pm \overline{m}(z) t^{2\alpha}) \IND{x \notin (d, d+t]} \Big\},
\end{align*}
and therefore,
\begin{align}
n(\widehat{M}_n(t/n) - \widehat{M}_n(0))
&= \sum_{d \in \cD} W_{n,d}(t) \pm (t/n)^\alpha \sum_{i=1}^n \overline{m}(Z_i) \IND{d < X_i  \le d+t/n} \\
&\quad\quad\quad\quad + \sum_{d \in \cD} A_{n,d}(t) \pm (t/n)^{2\alpha} \sum_{i=1}^n \overline{m}(Z_i),
\end{align}
where
\begin{equation}
W_{n,d}(t) := \sum_{i=1}^n \IND{d < X_i  \le d+t/n} B_d(Z_i)
\end{equation}
and
\begin{equation}
A_{n,d}(t) := \sum_{i=1}^n \loss'(Z_i) (f(X_i) - f(X_i-t/n)) \IND{X_i \notin (d, d+t/n]}.
\end{equation}
A Poisson approximation gives that $W_{n,d}$ converges as a process to a marked Poisson process with intensity $\lambda(d)$ on $\bbR_+$ and mark distribution that of $B_d(Z)$, and these processes are independent of each other in the large-$n$ limit. 
We elaborate in \lemref{poisson}.
So it suffices to show that the other three terms above are $o_P(1)$.
The second term has absolute first moment equal to
\begin{align*}
&(t/n)^\alpha n \E[\overline{m}(Z)] \P(d \le X < d+t/n) \\
&\sim (t/n)^\alpha n \E[\overline{m}(Z)] \lambda(d) (t/n) \\
&\asymp t^{\alpha+1}/n^\alpha \to 0.
\end{align*}
For the third term, each $A_{n,d}$ has mean 0 by the fact that $\E[\loss'(Z)] = 0$ and $X \indep Z$, and it has second moment equal to
\begin{align*}
&n \E[\loss'(Z)^2] \E\big[(f(X) - f(X-t/n))^2 \IND{X \notin (d, d+t/n]}\big] \\
&\le n \E[\loss'(Z)^2] (H (t/n)^\alpha)^2 \\
&\asymp n/n^{2\alpha} \to 0,
\end{align*}
since $\alpha > 1/2$.
The fourth term has absolute first moment equal to
\begin{align*}
(t/n)^{2\alpha} n \E[\overline{m}(Z)]
\asymp n/n^{2\alpha} \to 0,
\end{align*}
since $\alpha > 1/2$.
In all cases, we may thus apply Markov's inequality to get that these terms all converge to 0 in probability.

We focused on $t > 0$, but the same arguments apply to $t < 0$. Note that, when considering $t < 0$, $B^-_d(z) := \loss(z-\delta_d)-\loss(z)$ plays the role of $B_d(z)$, but $B_d^-(Z) \sim B_d(Z)$, using the fact that $\loss$ is even and that $Z$ is symmetric about 0, so that the mark distribution remains that of $B_d(Z)$.

We now apply \lemref{argmax_extension}.
By construction, $W_n$ and $W$ both are piecewise continuous and thus have right and left limits at every point.
And we proved that $n \hat\theta_n$ is uniformly tight in \thmref{rate}.
So we just have to show that $W$ satisfies the other properties required in the lemma. 
We note that $W$ is itsef a marked Poisson process with intensity $\sum_{d\in\cD} \lambda(d)$ and mark distribution the mixture of the $B_d(Z)$ distributions with mixture weights proportional to the $\lambda(d)$'s. 
In particular, as is well-known, for every (fixed) $x$, $W$ is continuous at $x$ with probability 1.
Further, the mark distribution has strictly positive mean, this being the case because $\E[B_d(Z)] > 0$ under \aspref{M}. It follows from this that $W(t) \to +\infty$ as $t \to \pm \infty$ (almost surely), and given that $W$ is piecewise constant with different values on each interval --- since the distribution of $B_d(Z)$ is continuous under our assumptions --- its minimum point set is a (bounded) interval defined by two discontinuity points, $\underline{t}$ and $\overline{t}$, which have continuous distributions.
\end{proof}

\begin{rem}
In change-point settings where the design is fixed, the minimizer of the limit process is often unique, in which case this is the limit of any empirical minimizer. This is the case, for example, in \citep{yao1989least, hinkley1970inference, dumbgen1991asymptotic}.
In \cite[Sec 14.5.1]{kosorok2008introduction}, a function of the form $a \{x \le t\} + b \{x > t\}$ is fitted to the data (or in the language that we have been using, is `matched' to the noisy signal). The design is random as it is here, and therefore the empirical minimizer is not unique in terms of $t$. To circumvent this issue, the smallest minimizer (in $t$) is used as the location estimator, which is then shown to converge to the smallest minimizer of the limit process, which in terms of $t$ is also a compound Poisson process. See also \citep{seijo2011change, lan2009change}. 
This approach does not seem applicable in our context. Indeed, the situation here is different in that $f$ is not necessary piecewise constant, and in particular it is very possible that the empirical minimizer is unique, rendering the selection of the smallest minimizer superfluous and leaving the issue of multiple minima in the limit untouched. We thus opted for the approach proposed by \cite{ferger2004continuous}.
\end{rem}

\section{More flexible templates}
\label{sec:flexible}
Some situations may call for finding the best match in a family of templates. Assuming a parametric model, the model becomes
\begin{equation}
\label{flexible_model}
Y_i = f_{\theta^*}(X_i) + Z_i,
\end{equation}
where the family $\{f_\theta : \theta \in \Theta\}$ is known. The smooth setting can be considered in a very general framework and we do so in \secref{flexible_smooth}. The non-smooth setting is, as usual, more delicate, so we content ourselves with considering a location-scale extension of our basic model \eqref{model} in \secref{flexible_non-smooth} which seems popular in practice.

\subsection{Smooth setting}
\label{sec:flexible_smooth}

We consider the model \eqref{flexible_model} in a rather general setting where $f_\theta: \bbR^p \to \bbR$ and $\Theta$ is a bounded subset of a Euclidean space.
The same toolset can then be used to derive the asymptotic behavior of the M-estimator, defined as in \eqref{M-estimator} with 
\begin{equation}
\label{flexible_m}
m_\theta(x, y) := \loss(y - f_\theta(x)).
\end{equation}
%It turns out that the smooth setting is almost completely parallel in terms of results and technique. The non-smooth setting is, on the other hand, not as simple.
Throughout, the same basic assumptions are in place with appropriate modifications: In particular, in \aspref{template} we now assume that
\begin{equation}
\label{identifiable}
\text{For any $\theta \ne \theta^*$, $f_\theta(\cdot) \ne f_{\theta^*}(\cdot)$ on a set of positive measure under $\lambda$.}
\end{equation}
Here we consider the smooth case, which again corresponds to a template $f$ which is Lipschitz. 

\paragraph{Consistency}
In view of \aspref{M}, consistency ensues as soon as the function class $\{m_\theta : \theta \in \Theta\}$ is Glivenko--Cantelli, which is the case here since  
\begin{equation}
\label{Lip_theta}
|f_{\theta_1}(x) - f_{\theta_2}(x)| \le H \|\theta_1-\theta_2\|
\end{equation} 
for all $\theta_1, \theta_2 \in \Theta$. See \cite[Ex 19.7]{van2000asymptotic}.

\paragraph{Rate of convergence}
To obtain a rate of convergence, we rely on \lemref{rate} as before. We note that \lemref{Delta} remains valid, now with $\Delta(\theta, \theta^*) = \int (f_\theta(x) - f_{\theta^*}(x))^2 \lambda(x) \d x$, and by dominated convergence, 
\begin{equation}
\Delta(\theta, \theta^*) 
\sim (\theta-\theta^*)^\top Q (\theta-\theta^*), \quad \theta \to \theta^*, \quad \text{where } Q := \E\big[\dot f_{\theta^*}(X) \dot f_{\theta^*}(X)^\top\big],
\end{equation} 
so that $\Delta(\theta, \theta^*) \asymp \|\theta-\theta^*\|^2$ as long as $Q$ is full rank, which we assume henceforth.
It is also the case that $\cM_\delta$ has an envelope with square integral bounded by $C \delta^2$ and $\eps$-bracketing number $N_\delta(\eps) \le C \delta \eps^{-d}$.
The former is an immediate consequence of \eqref{Lip_theta}, while the latter is based on \cite[Ex 19.7]{van2000asymptotic} and a scaling argument.
We are thus in a position to apply \lemref{rate}, with $a=2$ and $b=1$, to get that the M-estimator is $\sqrt{n}$-consistent in expectation, meaning that
\begin{equation}
\E\big[\sqrt{n} \|\hat\theta_n - \theta^*\|\big] = O(1).
\end{equation}
And \lemref{information_nonsmooth} applies (essentially verbatim) to establish this as the minimax rate of convergence under the same condition on $\phi$.

\paragraph{Limit distribution}
We can also obtain a limit distribution following the arguments underlying \thmref{limit_holder}. With all these arguments in plain view, it is a simple endeavor to check that everything proceeds in the same way, based on the fact that
\begin{equation}
n \Delta(\theta^*+s/\sqrt{n}, \theta^*+t/\sqrt{n}) 
\xrightarrow{n \to \infty} \Delta_0(s,t) := (s-t)^\top Q (s-t).
\end{equation}
Following that path, we arrive at the conclusion that
\begin{equation}
\sqrt{n}(\hat\theta_n - \theta^*) \xRightarrow{n \to \infty} \argmin_t \{g(t) + G_t\},
\end{equation} 
where $g(t) := C_0 \Delta_0(t,0) = C_0 t^\top Q t$ and $G_t$ is the centered Gaussian process with variance function $v(s,t) := C_1 \Delta_0(s,t) = C_1 (s-t)^\top Q (s-t)$. But this is the classical setting. Indeed, $G_t$ can be represented as $U^\top Q^{1/2} t$ where $U$ is a standard normal random vector. This is true because $G_t$ and $Q^{1/2} U^\top t$, as processes, are both Gaussian with same mean and same variance function. 
Hence, 
\begin{equation}
\argmin_t \{g(t) + G_t\} 
\equiv \argmin_t \{C_0 t^\top Q t + C_1^{1/2} Q^{1/2} U^\top t\}
= -\frac{C_1^{1/2}}{2 C_0} Q^{-1/2} U,
\end{equation}
which is centered normal with covariance
\begin{equation}
\label{asymp_cov}
(\E[\loss'(Z)^2]/\E[\loss''(Z)]^2) \E[\dot f_{\theta^*}(X) \dot f_{\theta^*}(X)^\top]^{-1}.
\end{equation}
We have thus established this as the limit distribution of $\sqrt{n}(\hat\theta_n - \theta^*)$.

The model is also quadratic mean differentiable \cite[Th 12.2.2]{lehmann2005testing}, with the information at $\theta$ being given by
\begin{equation}
I_{\theta} := \int \dot f_{\theta}(x) \dot f_{\theta}(x)^\top \lambda(x) \d x \times \int \frac{\phi'(y)^2}{\phi(y)} \d y\ .
\end{equation}
The analog of \thmref{LAM} applies in that case, meaning that the M-estimator is locally asymptotically minimax when it coincides with the MLE.

\begin{exa}
A special case which might be of interest in the context of matching a template to a signal is when $f_\theta(x) = S_\theta (f(T_\theta(x)))$, where for every $\theta$, $S_\theta : \bbR \to \bbR$ and $T_\theta: \bbR^p \to \bbR^p$ are both diffeomorphisms, and such that $\theta \mapsto S_\theta(x)$ and $\theta \mapsto T_\theta(x)$ are differentiable for almost every $x$.
As before, $f$ is a known function. In this special case, and assuming without loss of generality that $S_{\theta^*}$ and $T_{\theta^*}$ are the identity functions for their respective spaces, we have
\begin{equation}
\dot f_{\theta^*}(x) 
= \dot S_{\theta^*}(f(x)) + \dot T_{\theta^*}(x)^\top \nabla f(x).
\end{equation}
For instance, in the affine family of templates given by $a f(B^{-1}(\cdot-b))$, the parameter is $\theta = (a, b, B)$ where $a > 0$, $b \in \bbR^p$ and $B \in \bbR^{p \times p}$ invertible, and the transformations are given by $S_\theta(y) = a y$, $T_\theta(x) = B^{-1}(x-b)$. 
Assuming without loss of generality that $\theta^* = (1, 0, I)$, $S_{\theta^*}$ and $T_{\theta^*}$ are the identity functions.
Then $\dot f_{\theta^*}(x) = (f(x), - \nabla f(x), - \nabla f(x) x^\top)$. When the design is on the real line, that is when $b$ and $B$ are real numbers, this becomes $\dot f_{\theta^*}(x) = (f(x), - f'(x), - f'(x) x)$, in which case the asymptotic covariance matrix of the M-estimator is, according to \eqref{asymp_cov}, proportional to the inverse of the integral with respect to $\lambda$ of the following matrix 
\begin{equation}
\dot f_{\theta^*}(x) \dot f_{\theta^*}(x)^\top
= \begin{pmatrix}
f(x)^2 & - f(x) f'(x) & - x f(x) f'(x) \\
- f(x) f'(x) & f'(x)^2 & x f'(x)^2 \\
- x f(x) f'(x) & x f'(x)^2 & x^2 f'(x)^2
\end{pmatrix}\ .
\end{equation}
\end{exa}

\subsection{Non-smooth setting}
\label{sec:flexible_non-smooth}

Unlike the smooth setting, non-smooth situations typically need a more custom treatment. Instead of attempting to do that, which if possible at all is beyond the scope of the present paper, we consider a relatively simple extension of our basic model \eqref{model} which already exhibits some interesting features and which, simultaneously, happens to be popular in practice.
Specifically we consider \eqref{flexible_model} where, for $\theta = (\beta, \xi, \nu) \in \Theta := \bbR \times \bbR \times \bbR_+^*$, $f_\theta(x) = \beta f((x-\xi)/\nu)$. As before, $f$ is fixed.
We are thus considering the following amplitude-location-scale model
\begin{equation}
\label{location_scale}
Y_i = \beta^* f((X_i - \xi^*)/\nu^*) + Z_i,
\end{equation}
where $f$ is known to the analyst and the goal is to estimate the three parameters $\beta^*, \xi^*, \nu^*$.
(Sometimes only one of these parameters is of interest, so that the others can be considered nuisance parameters. But this perspective does not fundamentally change the analysis.)
The M-estimator is then defined as in \eqref{M-estimator} with 
\begin{equation}
m_\theta(x, y) := \loss(y - \beta f((X_i - \xi)/\nu)),
\end{equation}
with the minimization now being over $\Theta$. 
We of course assume \eqref{identifiable}, and in addition, that $f_{\theta^*}$ is not constant on the support of $\lambda$. So that the setting is non-smooth, we assume below that $f$ is piecewise Lipschitz. For concreteness, we focus on a loss function which has a bounded second derivative.
Otherwise, all the other basic assumptions are in place.

It turns out that the Euclidean norm, which was used in \secref{smooth} and \secref{non-smooth} as the absolute value, and in \secref{flexible_smooth}, is here {\em not} a good reference semimetric. Instead, we use the following
\begin{equation}
{\sf d}(\theta_1, \theta_2) := (\beta_1-\beta_2)^2 + |\xi_1-\xi_2| + |\nu_1-\nu_2|.
\end{equation}
The square-root of that would also do, but the main feature of this semimetric is that the amplitude is treated differently than the location and the scale, and this will be crucial when deriving rates of convergence (as we do below) since the rate of convergence for the estimation of the amplitude is in $\sqrt{n}$ while those for the location and scale are in $n$.

\paragraph{Consistency}
$\Theta$ as defined above is not compact, and some odd things can happen. To avoid complications, we assume that the support of $\lambda$ contains in its interior the support of $f_{\theta^*}$. Besides being a somewhat natural restriction, with this assumption we are able to reduce the analysis to a compact subset $\Theta_0 \subset \Theta$. We do so in \lemref{compact}, where we make mild assumptions on the noise distribution and loss function, which are essentially the same underlying \aspref{M}, as detailed in \lemref{consistent}.
Then consistency is, again, a consequence of the fact that $\{m_\theta : \theta \in \Theta_0\}$ is Glivenko--Cantelli --- in fact, it is Donsker, as we argue below.

\paragraph{Rate of convergence}
To obtain a rate of convergence, we rely on \lemref{rate} as before. We note that \lemref{Delta} remains valid, now with $\Delta(\theta, \theta^*) = \int (f_\theta(x) - f_{\theta^*}(x))^2 \lambda(x) \d x$. Without loss of generality, assume that $\theta^* = (1,0,1)$ henceforth.
Developing the square inside the integral, we get
\begin{align*}
\Delta(\theta, \theta^*)
&= (1-\beta)^2 \int f(x)^2 \lambda(x) \d x \\
&\quad + \beta^2 \int (f(x) - f(x-\xi))^2 \lambda(x) \d x \\
&\quad + \beta^2 \int (f(x-\xi) - f((x-\xi)/\nu))^2 \lambda(x) \d x \\
&\quad\quad +2(1-\beta)\beta \int f(x)(f(x) - f(x-\xi)) \lambda(x) \d x \\
&\quad\quad +2(1-\beta)\beta \int f(x)(f(x-\xi) - f((x-\xi)/\nu) \lambda(x) \d x \\
&\quad\quad +2\beta^2 \int (f(x) - f(x-\xi))(f(x-\xi) - f((x-\xi)/\nu) \lambda(x) \d x.
\end{align*}
The first three terms on the RHS are the main terms, as we will see, with the other terms being negligible.
The 1st term is $= C (1-\beta)^2$. For the following, consider $(\beta,\xi,\nu)$ approaching $(\beta^*, \xi^*, \nu^*) = (1,0,1)$.
As we saw in \eqref{Delta_discont}, the 2nd term is $\sim |\xi| D$.
Similar calculations establish that the 3rd term is $\sim |\nu - 1| D$, and also that the 4th term is $=O((1-\beta) \xi)$, that the 5th term is $=O((1-\beta)(1-\nu))$, and that the 6th term is $=O(\xi(1-\nu))$.
We thus conclude that
\begin{align}
\Delta(\theta, \theta^*) \asymp {\sf d}(\theta, \theta^*), \quad \theta \to \theta^*.
\end{align}
In particular, \eqref{Delta_cond} holds with $a=1$.

We now look at the class $\cM_\delta := \{m_\theta - m_{\theta^*} : {\sf d}(\theta, \theta^*) \le \delta\big\}$.
Using similar arguments as those underlying \lemref{envelope}, we derive
\begin{align*}
|m_\theta(x,y) - m_{\theta^*}(x,y)|
\le \overline{m}(z) |f(x) - \beta f((x-\xi)/\nu)|,
\end{align*}
where $z:=y-f(x)$ and $\overline{m}$ is a square integrable function with respect to $\phi$.
We then have
\begin{align*}
|f(x) - \beta f((x-\xi)/\nu)|
&\le |f(x) -  \beta f(x)| 
+ |\beta f(x) - \beta f(x-\xi)| 
+ |\beta f(x-\xi) - \beta f((x-\xi)/\nu)| \\
&\le |1-\beta| |f(x)| 
+ |\beta| |f(x) - f(x-\xi)| 
+ |\beta| |f(x-\xi) - f((x-\xi)/\nu)| \\
&\le |1-\beta| |f|_\infty 
+ |\beta| C (|\xi| + \IND{x \in \cD(\delta)})
+ |\beta| C (|\nu-1| + \IND{x \in \cD(\delta)}),
\end{align*}
where for the last two terms on the RHS we reasoned as in the proof of \lemref{envelope}.
When ${\sf d}(\theta, \theta^*) \le \delta$, we have $|\beta - 1| \le \delta^{1/2}$, while $|\xi| \le \delta$ and $|\nu-1| \le \delta$, so that
\begin{align*}
|m_\theta(x,y) - m_{\theta^*}(x,y)|
\le \overline{m}(z) \cdot C(\delta^{1/2} + \IND{x \in \cD(\delta)}).
\end{align*}
The function on the RHS is therefore an envelope for the class $\cM_\delta$.
In terms of bracketing numbers, reasoning as above, we find that for any $\theta, \theta_0$ within $\delta$ of $\theta^*$,
\begin{align*}
|m_\theta(x,y) - m_{\theta_0}(x,y)|
\le \overline{m}(z) \cdot \big(|\beta - \beta_0| + \IND{\dist(x, \xi_0 + \nu_0 \cD) \le C(|\xi-\xi_0| \vee |\nu-\nu_0|)}\big),
\end{align*}
where again $\overline{m}$ is square integrable.
For $\theta = (\beta, \xi,\nu)$, 
let $\beta_0 \in \{1 + (j/k) \delta^{1/2}: j = -k ,\dots, k\}$ such that $|\beta-\beta_0| \le \delta^{1/2}/k$; 
let $\xi_0 \in \{(j/l) \delta: j = -l ,\dots, l\}$ such that $|\xi-\xi_0| \le \delta/l$; and
let $\nu_0 \in \{1 + (j/l) \delta: j = -l,\dots, l\}$ such that $|\nu-\nu_0| \le \delta/l$. Then
\begin{align*}
m_\theta(x,y) = m_{\theta_0}(x,y)
\pm \overline{m}(z) \cdot \big(\delta^{1/2}/k + \IND{\dist(x, \xi_0 + \nu_0 \cD) \le C \delta/l}\big),
\end{align*}
so that $m_\theta$ is in the bracket defined by the two functions (differing by $\pm$) on the RHS. The square integral of the difference between these two functions is $\asymp \delta/k^2 + \delta/l$. So this is bounded by $\eps^2$, we choose $k \asymp \delta^{1/2}/\eps$ and $l \asymp \delta/\eps^2$. We conclude that we can cover $\cM_\delta$ with $k \times l \times l = O(\delta^{1/2}/\eps)^5$ $\eps$-brackets, so that $\cM_\delta$ has $\eps$-bracketing number $N_\delta(\eps) \le C (\delta^{1/2}/\eps)^5$. And from this we deduce, as in the proof of \thmref{rate}, that \eqref{J_cond} holds with $b = 1/2$.

We are thus in a position to apply \lemref{rate} and get that the M-estimator $\hat\theta_n$ is $n$-consistent, which here implies that
\begin{align*}
\E\big[\sqrt{n} |\hat\beta_n - \beta^*|\big] = O(1), &&
\E\big[n |\hat\xi_n - \xi^*|\big] = O(1), &&
\E\big[n |\hat\nu_n - \nu^*|\big] = O(1).
\end{align*}
In particular, $\hat\beta_n$ is $\sqrt{n}$-consistent while $\hat\xi_n$ and $\hat\nu_n$ are $n$-consistent.
And these rates are minimax, as \lemref{information_nonsmooth} applies essentially verbatim, all resting on the behavior of $\Delta(\theta, \theta^*)$ as $\theta \to \theta^*$.

\paragraph{Limit distribution}
We can also obtain a limit distribution following the arguments underlying \thmref{limit_discont}. 
We assume for convenience that $f$ is c\`aglad.
In what follows we assume that $\xi$ is sufficiently close to 0 and $\nu$ to 1 that intervals of the form $[d \wedge (\xi+\nu d), d \vee (\xi+\nu d)]$ where $d \in \cD$ do not intersect. For $d \in \cD$, define $\delta_d := f(d^+) - f(d^-)$, which is the `jump' that $f$ makes at $d$.
In view of the rates of convergence that we obtained just above, we consider $\beta = 1 + r/\sqrt{n}$, $\xi = s/n$, and $\nu = 1 + t/n$, where $r, s, t$ will remain fixed while $n \to \infty$.
Using the fact that $f$ is piecewise Lipschitz, we have
\begin{itemize}
\item
If $d < x \le \xi+\nu d$, for some $d \in \cD$, then 
\begin{align*}
m_\theta(x,y) - m_{\theta^*}(x,y)
&= \loss(z + \delta_d \pm C/\sqrt{n}) - \loss(z) \\
&= B^+_d(z) \pm \overline{m}(z)/\sqrt{n}, \quad B^+_d(z):= \loss(z + \delta_d) - \loss(z),
\end{align*}
where $\overline{m}$ is generic for a square integrable function. To arrive there we used a Taylor development of $\loss$ around $z+\delta_d$.
\item
Similarly, if $\xi+\nu d < x \le d$, for some $d \in \cD$, then
\begin{align*}
m_\theta(x,y) - m_{\theta^*}(x,y)
%&= \loss(z - \delta_d \pm C/\sqrt{n}) - \loss(z) \\
&= B^-_d(z) \pm \overline{m}(z)/\sqrt{n}, \quad B^-_d(z):= \loss(z - \delta_d) - \loss(z).
\end{align*}
\item
If there is no discontinuity point between $x$ and $(x-\xi)/\nu$, then using the fact that 
$$\loss(z+t) - \loss(z) = \loss'(z) t + \frac12 \loss''(z) t^2 \pm \frac12 \gamma(z, |t|) t^2,$$ 
with $\gamma$ defined in \eqref{gamma}, and the fact that 
\begin{equation}
f(x) - \beta f((x-\xi)/\nu) = (1-\beta) f(x) + R(x; \beta, \xi, \nu),
\end{equation}
with
\begin{equation}
|R(x; \beta, \xi, \nu)| \le C(|\xi|+|\nu-1|) \le C/n,
\end{equation}
we have
\begin{align*}
m_\theta(x,y) - m_{\theta^*}(x,y)
&= \loss'(z) (1-\beta) f(x) + \loss'(z) R(x; \beta, \xi, \nu) + \tfrac12 \loss''(z) (1-\beta)^2 f(x)^2 \\
&\quad \pm \tfrac12 \gamma(z, C/n) C/n \pm C/n^2.
\end{align*}
\end{itemize}
Hence, denoting $I_d(\xi,\nu') := (d \wedge (d+\xi+d (\nu'-1)), d \vee (d+\xi+d (\nu'-1))]$, we have
\begin{align*}
&n \big(\widehat M_n(\theta) - \widehat M_n(\theta^*)\big) \\
&= \sum_i \sum_{d \in \cD} \big(B^+_d(Z_i) \IND{d < X_i \le d+s/n+d t/n} + B^-_d(Z_i) \IND{d+s/n+d t/n < X_i \le d}\big) \\
&\quad + \sum_i \sum_{d \in \cD} (\overline{m}(Z_i)/\sqrt{n}) \IND{X_i \in I_d(s/n,t/n)} \\
&\quad + \sum_i \loss'(Z_i) (1-\beta) f(X_i) \IND{X_i \notin \cup_d I_d(s/n,t/n)} \\
&\quad + \sum_i \tfrac12 \loss''(Z_i) (1-\beta)^2 f(X_i)^2 \IND{X_i \notin \cup_d I_d(s/n,t/n)} \\
&\quad + \sum_i \loss'(Z_i) R(X_i; \beta, s/n, t/n) \IND{X_i \notin \cup_d I_d(s/n,t/n)} \\
&\quad \pm \sum_i \tfrac12 \gamma(Z_i, C/n) (C/n) \IND{X_i \notin \cup_d I_d(s/n,t/n)} \\
&\quad \pm \sum_i (C/n^2)\IND{X_i \notin \cup_d I_d(s/n,t/n)}. 
\end{align*}
As $n\to\infty$, the 1st sum on the RHS converges to $\sum_{d\in\cD} N_d(s + t d)$, where $\{N_d : d \in \cD\}$ are independent two-sided marked Poisson processes on the real line, with $N_d$ having intensity $\lambda(d)$ and mark distribution that of $B_d^+(Z)$. (Note that $B_d^+(Z) \sim B_d^-(Z)$.) This process can be equivalently expressed as $\sum_{d\in\cD} N_d(s) + \sum_{d\in\cD} \tilde N_d(t)$, where $N_d$ is as before while $\tilde N_d$ has intensity $\lambda(d) d$ and same mark distribution, and all these processes are independent of each other.
The 2nd sum has expectation 
$$n \sum_{d\in\cD} (C/\sqrt{n}) \P(X \in I_d(s/n,t/n)) \asymp 1/\sqrt{n}, \quad \text{since } \P(X \in I_d(s/n,t/n)) = O(1/n),$$
and so it converges to 0 in probability by Markov's inequality.
Because $\P(X \notin \cup_d I_d(s/n,t/n)) \to 1$, by the Lindeberg central limit theorem, the 3rd sum converges in distribution to the normal distribution with zero mean and variance $r^2 \E[\loss'(Z)^2] \E[f(X)^2]$.
The 4th sum converges to $\tfrac12 r^2 \E[\loss''(Z)]  \E[f(X)^2]$ in probability, essentially by the law of large numbers. (More rigorously, via Chebyshev's inequality, for example.)
The 5th sum has zero mean and variance bounded by $n \E[\loss'(Z)^2] \E[R(X_i; \beta, s/n, t/n)^2] = O(1/n)$, and so converges to 0 in probability by Markov's inequality.
The 6th sum is non negative and has expectation bounded by $C \E[\gamma(Z, C/n)] = o(1)$, by dominated convergence, so that it converges to 0 in probability by Markov's inequality.
The 7th sum is simply $= O(1/n)$.

The minimization over $\beta$, or equivalently over $r$, corresponds in the limit to a classical setting, since the Gaussian process, as a (random) function of $r$, can be expressed as \smash{${\tilde C_1}^{1/2} r U$}, where $\tilde C_1 := \E[\loss'(Z)^2] \E[f(X)^2]$ and $U$ is standard normal, and the drift term is equal to $\tilde C_0 r^2$, $\tilde C_0 := \tfrac12 \E[\loss''(Z)]  \E[f(X)^2]$. In particular, we have established that

\begin{quote}
{\em $\sqrt{n} (\hat\beta_n - \beta^*)$ converges weakly to the normal distribution with zero mean and variance} 
$$(\E[\loss'(Z)^2]/\E[\loss''(Z)]^2) \E[f(X)^2]^{-1}.$$
\end{quote}

The minimization over $\xi$, or equivalently over $s$, corresponds in the limit to the minimum a compound Poisson process. In particular, as in \thmref{limit_discont}, we can establish that

\begin{quote}
{\em $n (\hat\xi_n - \xi^*)$ is asymptotically, and in distribution, between the two most extreme minimizers of $\sum_d N_d$.}
\end{quote}

And, similarly, 
\begin{quote}
{\em $n (\hat\nu_n - \nu^*)$ is asymptotically, and in distribution, between the two most extreme minimizers of $\sum_d \tilde N_d$.}
\end{quote}

Since the minimization over $\xi$ and $\nu$ rely, in the limit, on independent processes, $\hat\xi_n$ and $\hat\nu_n$ are independent in the asymptote. Also, these processes are asymptotically independent of the Gaussian process driving the minimization over $\beta$, as these Poisson processes only rely on $O_P(1)$ data points. Hence,

\begin{quote} 
{\em $\hat\beta_n$, $\hat\xi_n$ and $\hat\nu_n$ are asymptotically mutually independent.}
\end{quote}

\section{Variants and extensions}
\label{sec:extensions}

\subsection{Fixed design}
\label{sec:fixed}
A fixed design is most common in signal processing, and also in change point analysis.
Most, if not all, of our results can be proved with very similar tools in that context. So that there is a close correspondence with the basic setting of \eqref{model}, in the context of a fixed design we assume that 
\begin{equation}
Y_i = f(x_i - \theta^*) + Z_i, \quad i = 1, \dots, n,
\end{equation}
where $x_i = \Lambda^{-1}(i/(n+1))$, $\Lambda$ being the distribution with density $\lambda$ and $\lambda$ being as before.
That the same results apply is due to the fact that the empirical process theory behind our results extends quite naturally (and easily) to the independent-but-not-necessarily-iid setting, which is exactly the extension needed since $(x_1, Y_1), \dots, (x_n, Y_n)$ are no longer iid but are still assumed independent. The additional complexity is the fact that now the underlying distribution depends on $n$: In terms of $(x_1, Z_1), \dots, (x_n, Z_n)$, it is the product of $\delta_{x_i} \otimes \phi$ over $i = 1, \dots, n$. But this presents no particular difficulty in our case.

These results from empirical process theory enable us to understand the large-$n$ behavior of $\widehat{M}_n(\theta) - M_n(\theta)$, where
\begin{align*}
M_n(\theta) 
:= \E[\widehat{M}_n(\theta)]
= \frac1n \sum_{i=1}^n m_\theta(x_i), && 
m_\theta(x) := \E[\loss(Z + f(x-\theta^*) - f(x-\theta))].
\end{align*}
But what is the behavior of $M_n$ itself?
As is clear, $M_n$ is a Riemann sum, and when $f$ is piecewise Lipschitz, for example, and $\loss$ is as before, one can easily prove that
\begin{equation}
\sup_\theta |M_n(\theta) - M(\theta)| \le C/n,
\end{equation}
where $C$ depends $f$ and $\loss$. The limit function $M$ is defined exactly as before in \eqref{M}, and we have used the fact that the supremum is over a compact set when $f$ and $\lambda$ are compactly supported, which we assume them to be.

\subsection{Periodic template}
In the signal processing literature it is not uncommon to consider the periodic setting where the design points are effectively in a torus rather than the real line. This is equivalent to considering a template that is periodic. The torus can be taken to be the unit interval with algebra there done modulo 1. In that case, the design density $\lambda$ is simply a density with support the unit interval, and the assumption that the template $f$ is compact is waved. So that the shift is identifiable, we require that $f$ is exactly 1-periodic meaning that $f(\cdot -\theta) \ne f(\cdot-\theta^*)$ unless $\theta = \theta^*$ (remember, modulo 1).

As can be easily verified, all our results apply in that setting with only minor modifications, if any at all.

\begin{rem}
In signal processing it is common to match a template to a signal by maximizing its convolution or Pearson correlation with the signal. Considering a regular design (also very common) where $x_i = i/n$, this amounts to defining the following estimator
\begin{equation}
\label{correlation}
\hat t := \argmax_{t \in [n]} \sum_{i=1}^n f\left(\tfrac{i-t}n\right) Y_i\ .
\end{equation}
The parameter $\theta$ corresponds to $t/n$ and is here constrained to be on the grid, as is typically the case in signal processing.
This estimator, as defined, is in fact the least squares estimator since 
\begin{equation}
\sum_{i=1}^n \left(Y_i - f\left(\tfrac{i-t}n\right)\right)^2 = \sum_{i=1}^n Y_i^2 + \sum_{i=1}^n f\left(\tfrac{i-t}n\right)^2 - 2 \sum_{i=1}^n f\left(\tfrac{i-t}n\right) Y_i\ ,
\end{equation}
and the first two sums on right-hand side do not depend on $t$. (For the second one, this is because $f$ is 1-periodic.)
\end{rem}

\subsection{Agnostic setting}
\label{sec:agnostic}
In practice, the model \eqref{model} could be completely wrong. Suppose, however, that the following holds
\begin{equation}
Y_i = g(X_i) + Z_i,
\end{equation}
with otherwise the same assumptions on the design and noise.
(This type of situation is sometimes referred to as a `mispecified model' or `improper learning'.)
In that case, the function $\tilde m_\theta(x, z)$ of \eqref{m_tilde}, which plays a crucial role in our derivations, takes the form
\begin{equation}
\tilde m_\theta(x, z) 
= \loss(z + g(x) - f(x - \theta)).
\end{equation}
It then becomes quite clear that most, if not all of our results extend to this case, if it is true that 
\begin{equation}
\theta^* := \argmin_\theta \E[\tilde m_\theta(X,Z)]
\end{equation}
is uniquely defined. Almost no conditions on $g$ are required other than, say, boundedness.

A case in point is where $g$, a function on the unit interval, is known to have a single discontinuity, say at $d$. Note that $d$ is unknown. Then one may want to use a stump, $f(x) = a \IND{x > 0}$, to locate the discontinuity. If one uses the squared error loss for example, it happens that $\theta^*$ is exactly the location of the discontinuity if it is the case that $g(x) < a/2$ when $x < d$ and $g(x) > a/2$ when $x > d$. And using an additional scale parameter as in \eqref{location_scale} frees one from having to guess a good value for $a$. The point here is that a simple, parameteric model can be used to locate a feature of interest (a discontinuity in this example) in an otherwise `nonparametric' setting.
%A similar strategy is deployed in \citep{korostelev1988minimax}.

\subsection{Semi-parametric models}
\label{sec:semiparametric}
The models we discussed in this paper are all parametric. This was intentional, to keep the exposition focused. However, empirical process theory has developed an arsenal of tools for semi-parametric models \cite[Ch 25]{van2000asymptotic, bickel1998efficient}. An emblematic example in the context of matching a template would be a class of piecewise Lipschitz functions: the parametric component of the model would be the location of the knots defining the intervals where the template is Lipschitz, while the nonparametric component would be the Lipschitz functions defining the template on these intervals.  \cite{korostelev1988minimax} considered this problem in the continuous white noise model and showed that it is possible to locate the location of a discontinuity to what corresponds to $O(1/n)$ accuracy in our context by simply looking for unusually large slopes. This is similar to the approach we allude to at the end of \secref{agnostic}. In any case, the same rate of convergence can be obtained for the M-estimator using tools similar to the ones we used in this paper; see \cite[Sec 5.8.1]{van2000asymptotic}.

\subsection{Alignment/registration}
We mentioned in the \hyperref[sec:intro]{Introduction} the close relationship with the problem of signal alignment (aka registration). A typical setting, that most resembles our basic shift model \eqref{model}, is where
\begin{equation}
Y_{ij} = f(X_{ij} - \theta_j) + Z_{ij}, \quad i \in [n], \quad j \in [s],
\end{equation}
so that instead of a single sample of size $n$ from \eqref{model}, we are provided with $s$ such samples, all shifted differently. Invariably, the function $f$ is unknown, and not much needs to be assumed about $f$ to estimate the shifts (i.e., align the signals) with $\sqrt{n}$ precision. Indeed, with some kernel smoothing, it is possible to design an estimator that is $\sqrt{n}$-consistent, as shown in \citep{hardle1990semiparametric, gamboa2007semi, vimond2010efficient, trigano2011semiparametric}.
Although this convergence rate can be achieved with hardly any assumption of $f$, our work here, which we placed in the context of the change point analysis literature, would indicate that this rate is suboptimal when $f$ has a discontinuity. 
%If anything, the procedure of \cite{korostelev1988minimax} applied to each of the samples achieves the $n$ rate.

\subsection{Concentration bounds}
The rate of convergence and the limit distribution, in each case, was established under very mild assumptions on the noise distribution. In particular, in terms of tail decay, we only assumed that it was sufficient for the expected loss to be finite at the true value of the parameter, i.e., $\E[\loss(Z)] < \infty$. As we discussed earlier, for the losses considered here (and almost everywhere else in the literature), this implies that $\E[\loss'(Z + c)^2] < \infty$ for every $c>0$, which is all that was needed.

When some tail decay is assumed, then concentration bounds can be derived. The simplest and most straightforward example of that is when the loss function is bounded, as is the case for the Tukey loss. Then, under no additional assumption on the noise distribution, the M-estimator enjoys sub-Gaussian concentration.
This can be seen from combining Talagrand's inequality for empirical processes \citep{talagrand1996new} and \cite[Th 3.2.5]{van1996weak} --- the latter being a more general version of \cite[Th 5.52]{van2000asymptotic}, which is the result at the foundation of \lemref{rate}.

\section{Numerical experiments}
\label{sec:numerics}

We performed some basic experiments to probe our theory. We present the result of these experiments below, subdivided into `smooth' and `non-smooth' settings. The design distribution is the uniform distribution on the unit interval. We consider three noise distributions: Gaussian, Student $t$-distribution with 3 degrees of freedom, and Cauchy. And we consider four losses: squared error, absolute-value, Huber, and Tukey.
We assume throughout that $\theta^* = 0$.

\subsection{Smooth setting}
We consider the following two filters:
\beq
\text{Template $A$: } \;\;f(x)=
\begin{cases}
4x-1& 0.25\leq x<0.5,\\
3-4x& 0.5 \leq x<0.75,\\
0& \text{otherwise}
\end{cases}
\eeq
and 
\beq
\text{Template $B$: }\;\;f(x) = \max\{0, (1-(4x-2)^2)^3\}.
\eeq\\
Template $A$ is Lipschitz, while Template $B$ is even smoother. Our motivation for considering Template $B$ is to verify that more smoothness does not change things much (as predicted by our theory). See \figref{Fig-1} for an illustration. 

\begin{figure}[htbp]
\centering 
\subfigure[Template $A$]{
\label{fig:Fig-1-1}
\includegraphics[width=0.43\textwidth]{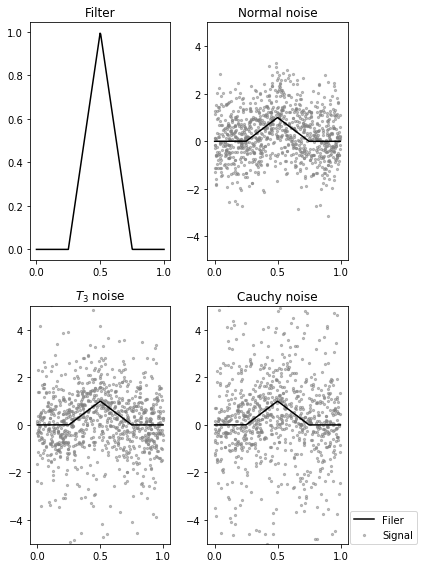}}
\subfigure[Template $B$]{
\label{fig:Fig-1-2}
\includegraphics[width=0.43\textwidth]{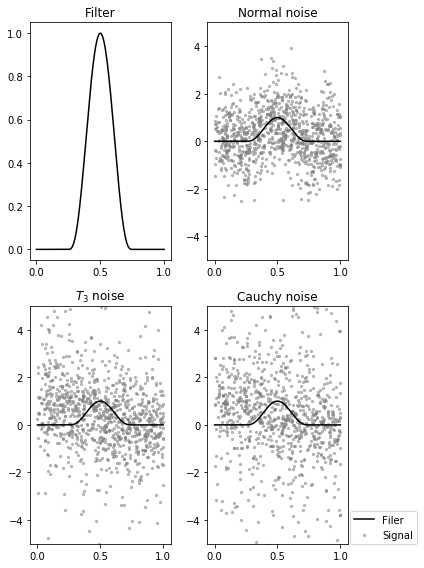}}
\caption{Templates and noisy signals. Although the sample size is $n = 10000$, for the sake of clarity, we only include $1000$ points and limit the range of the y-axis to $[-5,5]$.}
\label{fig:Fig-1}
\end{figure}

We used a sample of size $n = 10000$ and repeated each scenario (combination of template, noise distribution, and loss function) $200$ times.  We show the mean of $|\sqrt{n}(\hat{\theta_n}-\theta^*)|$ in Table~\ref{table1}. Box plots of estimation error $|\hat{\theta_n}-\theta^*|$ are shown in \figref{Fig-5} and \figref{Fig-6}. The distribution of $\sqrt{n}(\hat{\theta_n}-\theta^*)$ is plotted in \figref{Fig-2} and \figref{Fig-3} as an histogram overlaid with the Gaussian distribution predicted by our asymptotic calculations. Out of curiosity, we also looked at the setting where the noise is Cauchy and yet we use squared error as loss in \figref{Fig-4}.
The result of these experiments are by and large congruent with our theory. In particular, there is no noticeable difference between the two templates. 

\begin{table}[htbp]
\centering
\begin{tabular}{ cccccc } 
\toprule  
\multirow{2}{*}{Template} & \multirow{2}{*}{Noise} & \multicolumn{4}{c}{Loss Function}\\
&& Squared error  &Absolute-value &Huber & Tukey\\
\midrule
\multirow{3}{*}{Template $A$} & Normal & 0.2791  & 0.3705 &0.2620 & 0.3301\\ 
& $T_3$  & 0.5168 & 0.3496 &0.3634 & 0.5053\\ 
& Cauchy  & 28.2535 & 0.4355 & 0.5203 & 0.4113\\ 
\midrule
\multirow{3}{*}{Template $B$} & Normal & 0.2511  & 0.3326 & 0.2453 & 0.2918\\ 
& $T_3$  & 0.4524 & 0.3236 &0.3293 & 0.3286\\ 
& Cauchy  & 48.2211 & 0.3958 & 0.3982 & 0.3462\\ 
\bottomrule
\end{tabular}
\caption{Mean of $|\sqrt{n}(\hat{\theta_n}-\theta^*)|$ with $n = 10000$ over $200$ repeats.}
\label{table1}
\end{table}

\begin{figure}[htbp]
\centering 
\includegraphics[width=12cm]{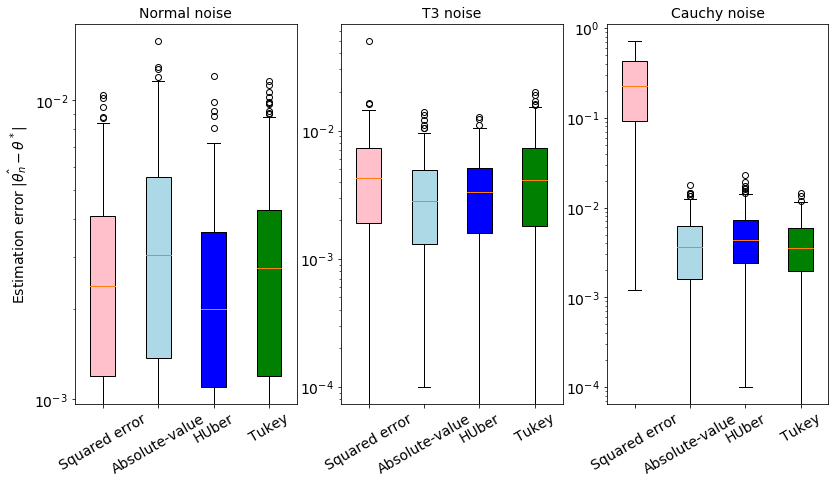}
\caption{Box plot of estimation error $|\hat{\theta_n}-\theta^*|$ for Template $A$}
\label{fig:Fig-5}
\end{figure}

\begin{figure}[htbp]
\centering 
\includegraphics[width=12cm]{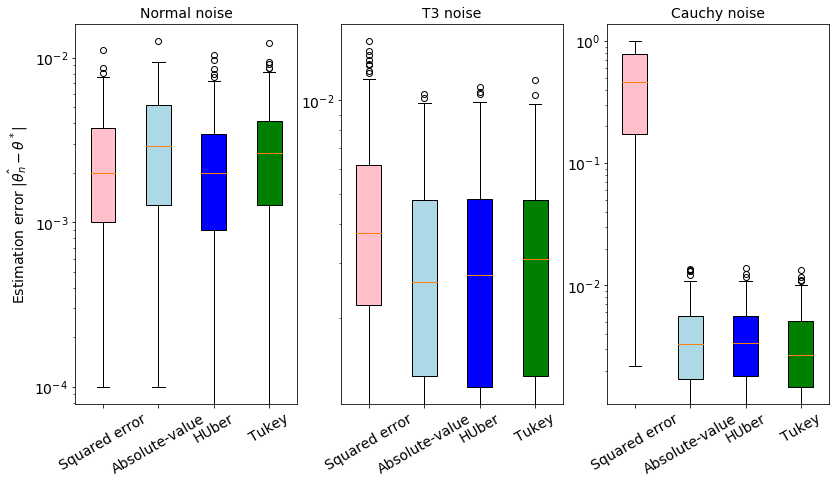}
\caption{Box plot of estimation error $|\hat{\theta_n}-\theta^*|$ for Template $B$}
\label{fig:Fig-6}
\end{figure}

\begin{figure}[htbp]
\centering 
\subfigure[Normal noise]{
\label{fig:Fig-2-1}
\includegraphics[width=0.85\textwidth]{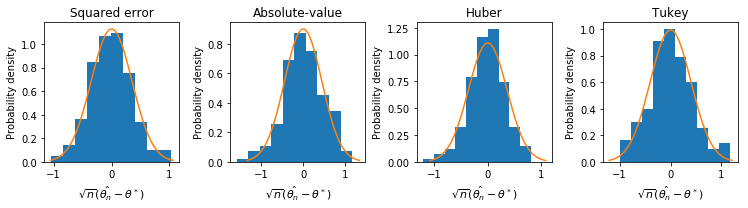}}
\subfigure[$T_3$ noise]{
\label{fig:Fig-2-2}
\includegraphics[width=0.85\textwidth]{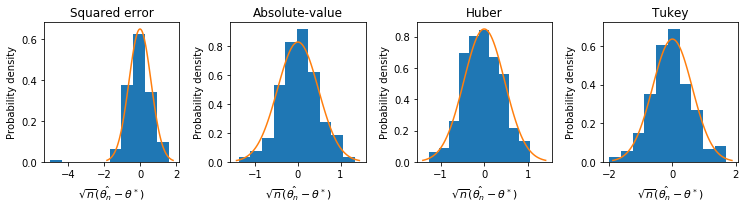}}
\subfigure[Cauchy noise]{
\label{fig:Fig-2-3}
\includegraphics[width=0.85\textwidth]{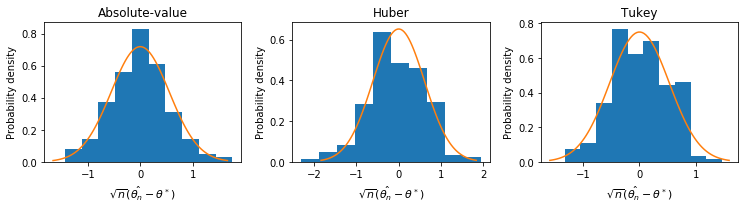}}
\caption{Distribution under Template $A$. The histogram presents the distribution of $\sqrt{n}(\hat{\theta_n}-\theta^*)$. The orange bell-shaped curve is the density of normal distribution predicted by the theory.}
\label{fig:Fig-2}
\end{figure}

\begin{figure}[htbp]
\centering 
\subfigure[Normal noise]{
\label{fig:Fig-3-1}
\includegraphics[width=0.85\textwidth]{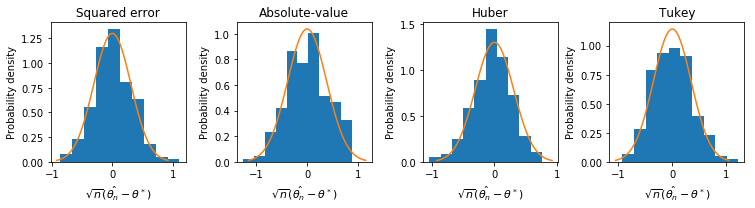}}
\subfigure[$T_3$ noise]{
\label{fig:Fig-3-2}
\includegraphics[width=0.85\textwidth]{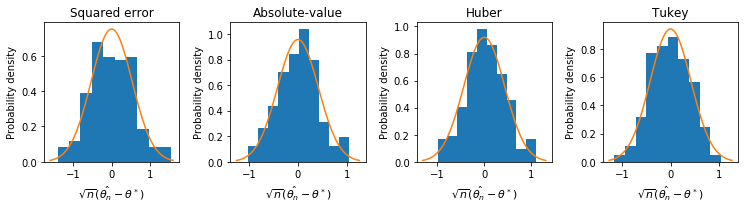}}
\subfigure[Cauchy noise]{
\label{fig:Fig-3-3}
\includegraphics[width=0.85\textwidth]{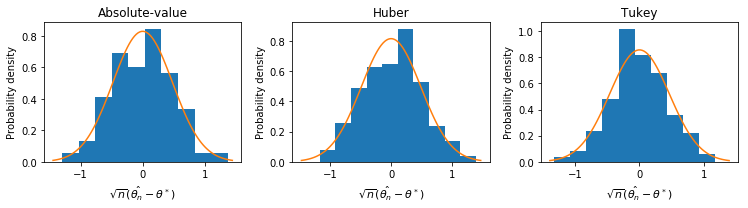}}
\caption{Distribution under Template $B$. The histogram presents the distribution of $\sqrt{n}(\hat{\theta_n}-\theta^*)$. The orange bell-shaped curve is the density of normal distribution predicted by the theory.}
\label{fig:Fig-3}
\end{figure}

\begin{figure}[htbp]
\centering 
\includegraphics[width=14cm]{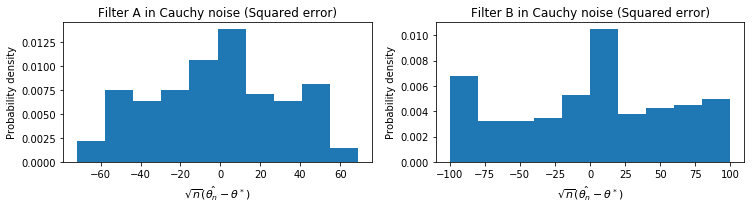}
\caption{Distribution of $\sqrt{n}(\hat{\theta_n}-\theta^*)$, for Templates $A$ and $B$, with squared error loss under the Cauchy distribution as noise distribution. Note that our theory is silent on this setting.}
\label{fig:Fig-4}
\end{figure}

We also ran experiments with varying sample size $n$. We focused on Template $A$ with absolute-value loss and $T_3$ noise, to investigate the accuracy of the asymptotic distribution as $n$ increases. As sample size we used $n \in \{100, 500, 1000, 5000, 10000\}$. To have a finer sense of the accuracy, we used $1000$ repeats. We show the mean of $|\sqrt{n}(\hat{\theta_n}-\theta^*)|$ in Table~\ref{table2}, the box plot of $|\hat{\theta_n}-\theta^*|$ in \figref{Fig-8}, and the histogram of $\sqrt{n}(\hat{\theta_n}-\theta^*)$ in \figref{Fig-7}.

\begin{table}[htbp]
\centering
\begin{tabular}{ cccccc } 
\toprule  
n  & 100 & 500 & 1000 & 5000 & 10000\\
\midrule
Mean of $|\sqrt{n}(\hat{\theta_n}-\theta^*)|$ & 0.5704 & 0.4286 & 0.4278 & 0.3766 & 0.3889\\
\bottomrule
\end{tabular}
\caption{Mean of $|\sqrt{n}(\hat{\theta_n}-\theta^*)|$ for Template $A$ with absolute-value loss and $T_3$ noise based on $1000$ repeats.}
\label{table2}
\end{table}

\begin{figure}[htbp]
\centering 
\includegraphics[width=10cm]{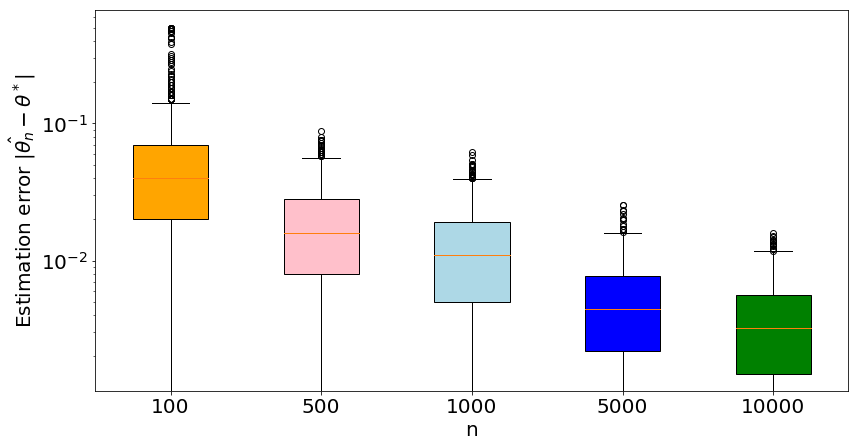}
\caption{Box plot of the estimation error $|\hat{\theta_n}-\theta^*|$ for Template $A$ with absolute-value loss and $T_3$ noise based on $1000$ repeats.}
\label{fig:Fig-8}
\end{figure}

\begin{figure}[htbp]
\centering 
\includegraphics[width=15cm]{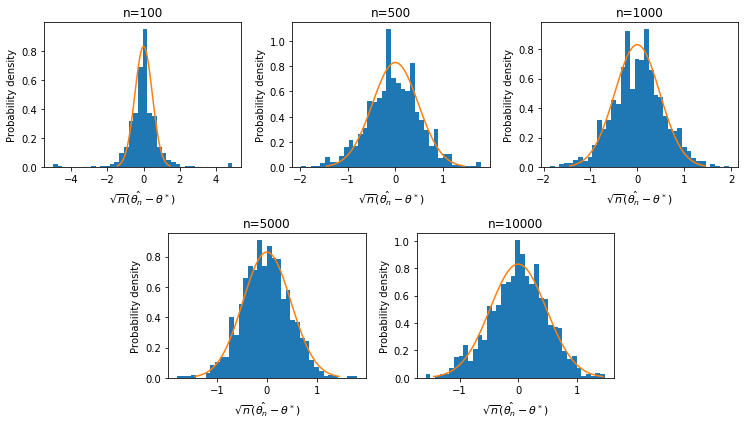}
\caption{Distribution of $\sqrt{n}(\hat{\theta_n}-\theta^*)$ for Template $A$ with absolute-value loss and $T_3$ noise based on $1000$ repeats.}
\label{fig:Fig-7}
\end{figure}

In \figref{Fig-7}, as $n \rightarrow \infty$, the distributions of $\sqrt{n}(\hat{\theta_n}-\theta^*)$ approaches to the normal distribution shown in the theory. \figref{Fig-8} indicates   estimation error $|\hat{\theta_n}-\theta^*|$ decreases as $n$ increases.

\newpage
\subsection{Non-smooth setting}
We consider the following three filters:
\beq
\text{Template $C$: } \;\;f(x)=
\begin{cases}
1& 0.25\leq x<0.75,\\
0& \text{otherwise,}
\end{cases}
\eeq

\beq
\text{Template $D$: } \;\;f(x) = 
\begin{cases}
1& 0.2\leq x<0.4 \;\;\text{or}\;\; 0.6 \leq x <0.8,\\
0& \text{otherwise}
\end{cases}
\eeq
and
\beq
\text{Template $E$: } \;\;f(x) = 
\begin{cases}
4x-1& 0.25 \leq x <0.5,\\
0& \text{otherwise.}
\end{cases}
\eeq\\
Template $C$ is a piecewise constant function with two discontinuities. Template $D$ is another piecewise constant function with more discontinuities. Template $E$ is a half-triangle with one discontinuity. See \figref{Fig-9} for an illustration. Our theory predict that what matters is the number and size of the discontinuities.

\begin{figure}[htbp]
\centering 
\subfigure[Template $C$]{
\label{fig:Fig-9-1}
\includegraphics[width=0.43\textwidth]{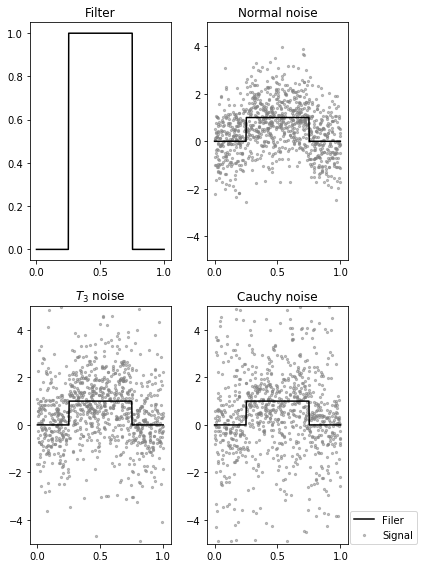}}
\subfigure[Template $D$]{
\label{fig:Fig-9-2}
\includegraphics[width=0.43\textwidth]{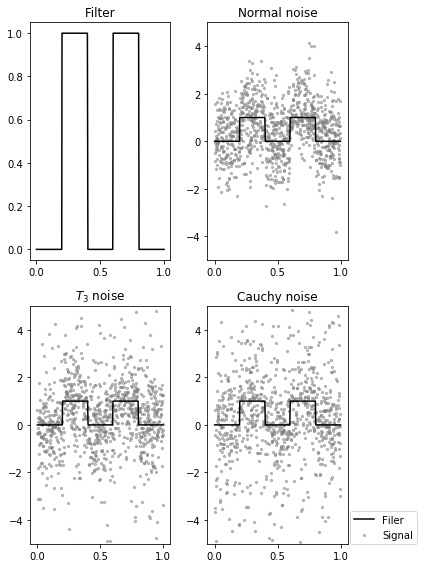}}
\subfigure[Template $E$]{
\label{fig:Fig-9-3}
\includegraphics[width=0.43\textwidth]{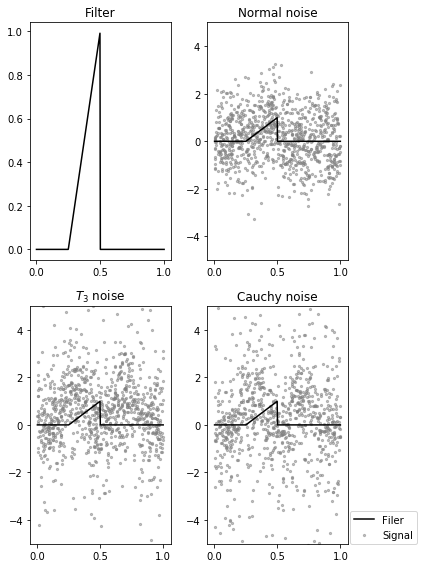}}
\caption{Templates and noisy signals. Although the sample size is $n = 10000$, for the sake of clarity, we only include $1000$ points and limit the range of the y-axis to $[-5,5]$.}
\label{fig:Fig-9}
\end{figure}

We show the mean of $|n(\hat{\theta_n}-\theta^*)|$ in Table~\ref{table3}, the estimation error $|\hat{\theta_n}-\theta^*|$ is shown as a box plot in Figures~\ref{fig:Fig-14}, \ref{fig:Fig-15}, and \ref{fig:Fig-16}, and the distribution of $n(\hat{\theta_n}-\theta^*)$ is plotted in Figures~\ref{fig:Fig-10}, \ref{fig:Fig-11}, and \ref{fig:Fig-12}.

\begin{table}[htbp]
\centering
\begin{tabular}{ cccccc } 
\toprule  
\multirow{2}{*}{Template} & \multirow{2}{*}{Noise} & \multicolumn{4}{c}{Loss Function}\\
&& Squared error  &Absolute-value &Huber & Tukey\\
\midrule
\multirow{3}{*}{Template $C$} & Normal & 1.876  & 1.849 & 1.868 & 2.120\\ 
& $T_3$  & 3.889 & 2.739 & 3.550 & 2.761\\ 
& Cauchy  & 2326.050 & 4.362 & 3.310 & 4.204\\ 
\midrule
\multirow{3}{*}{Template $D$} & Normal & 0.868  & 1.080 & 0.897 & 0.877\\ 
& $T_3$  & 1.802 & 1.220 &1.344 & 1.813\\ 
& Cauchy  & 2766.000 & 2.278 &2.078 & 1.862\\ 
\midrule
\multirow{3}{*}{Template $E$} & Normal & 3.498  & 3.798 & 3.414 & 4.008\\ 
& $T_3$ & 7.307 & 5.366 & 5.720 & 5.734\\ 
& Cauchy  & 4798.370 & 9.102 &8.839 & 7.759\\ 
\bottomrule
\end{tabular}
\caption{Mean of $|n(\hat{\theta_n}-\theta^*)|$ based on $200$ repeats.}
\label{table3}
\end{table}

\begin{figure}[htbp]
\centering 
\includegraphics[width=13cm]{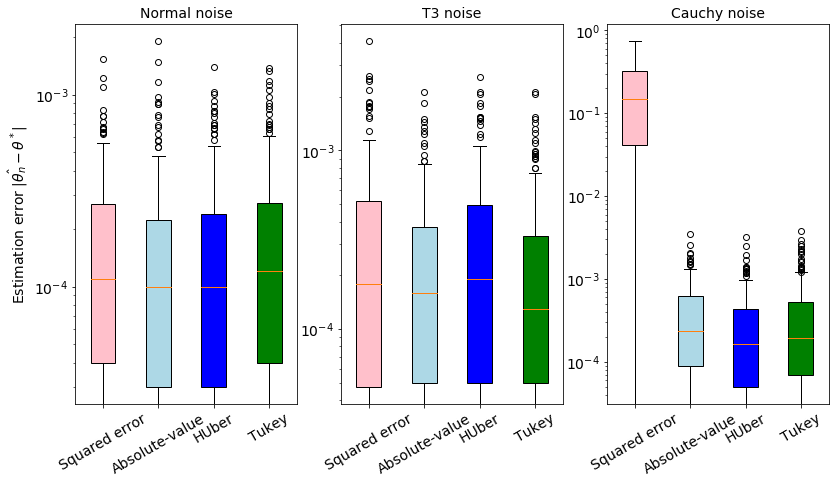}
\caption{Box plot of estimation error $|\hat{\theta_n}-\theta^*|$ for Template $C$ based on $200$ repeats.}
\label{fig:Fig-14}
\end{figure}

\begin{figure}[htbp]
\centering 
\includegraphics[width=13cm]{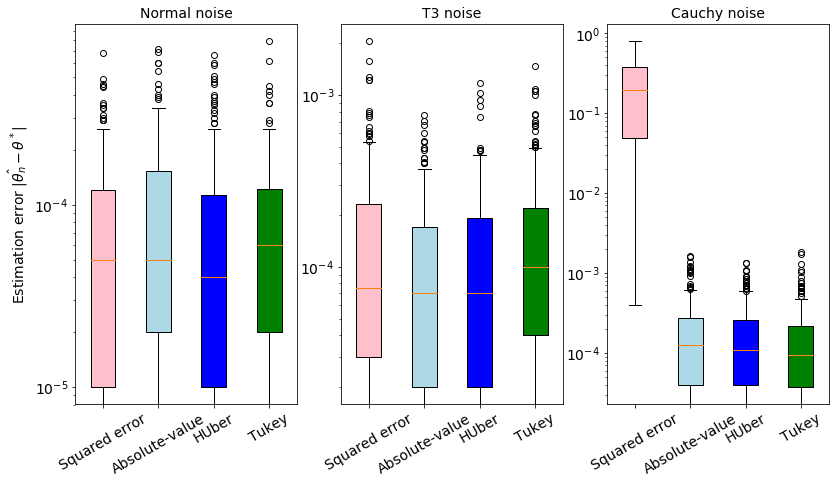}
\caption{Box plot of estimation error $|\hat{\theta_n}-\theta^*|$ for Template $D$ based on $200$ repeats.}
\label{fig:Fig-15}
\end{figure}

\begin{figure}[htbp]
\centering 
\includegraphics[width=13cm]{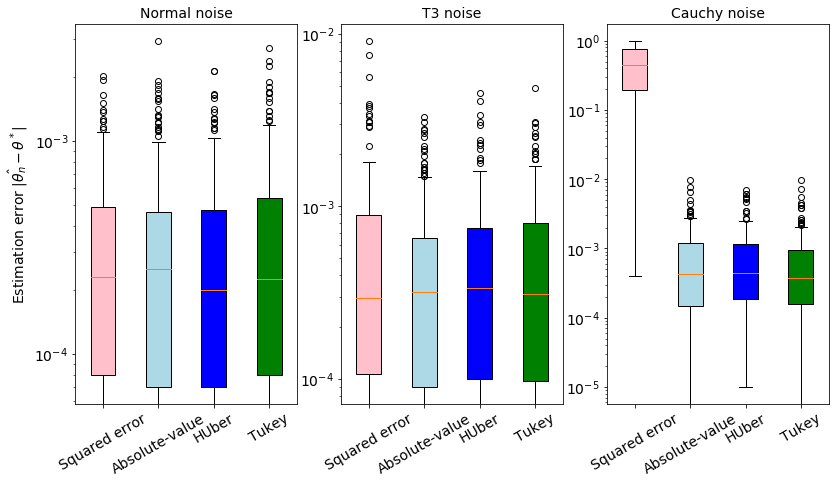}
\caption{Box plot of estimation error $|\hat{\theta_n}-\theta^*|$ for Template $E$ based on $200$ repeats.}
\label{fig:Fig-16}
\end{figure}

\begin{figure}[htbp]
\centering 
\subfigure[Normal noise]{
\label{fig:Fig-10-1}
\includegraphics[width=0.85\textwidth]{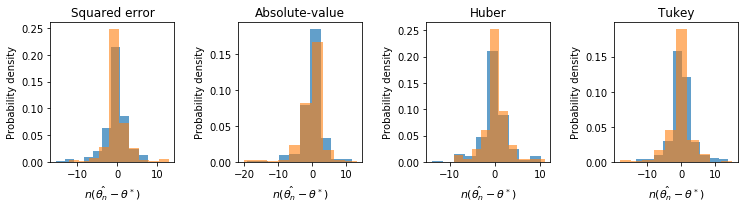}}
\subfigure[$T_3$ noise]{
\label{fig:Fig-10-2}
\includegraphics[width=0.85\textwidth]{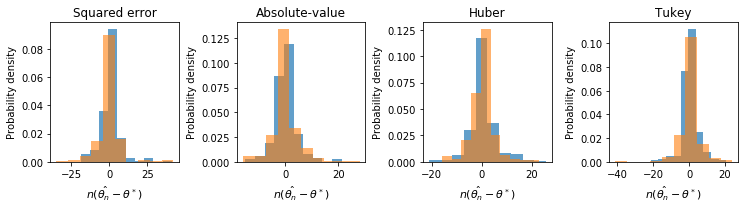}}
\subfigure[Cauchy noise]{
\label{fig:Fig-10-3}
\includegraphics[width=0.85\textwidth]{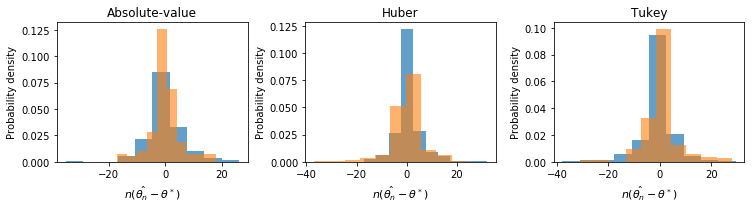}}
\caption{The blue histogram represents the distribution of $n(\hat{\theta_n}-\theta^*)$ for Template $C$. The orange histograms shows the simulated density of the midpoint of the minimizer interval of the marked Poisson process predicted by the theory. Based on $200$ repeats.}
\label{fig:Fig-10}
\end{figure}

\begin{figure}[htbp]
\centering 
\subfigure[Normal noise]{
\label{fig:Fig-11-1}
\includegraphics[width=0.85\textwidth]{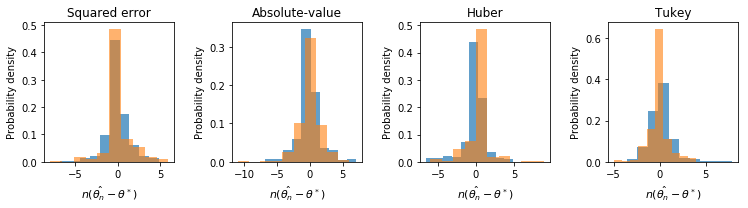}}
\subfigure[$T_3$ noise]{
\label{fig:Fig-11-2}
\includegraphics[width=0.85\textwidth]{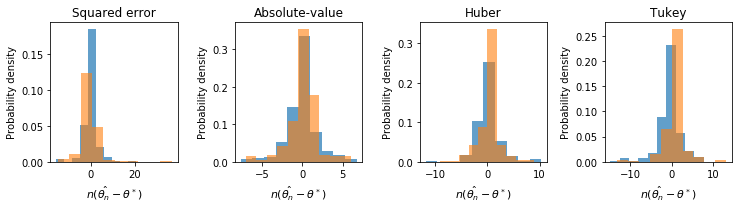}}
\subfigure[Cauchy noise]{
\label{fig:Fig-11-3}
\includegraphics[width=0.85\textwidth]{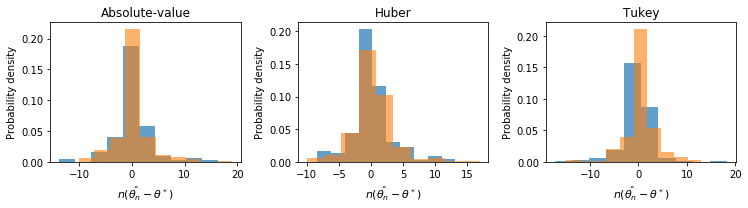}}
\caption{The blue histogram represents the distribution of $n(\hat{\theta_n}-\theta^*)$ for Template $D$. The orange histograms shows the simulated density of the midpoint of the minimizer interval of the marked Poisson process predicted by the theory. Based on $200$ repeats.}
\label{fig:Fig-11}
\end{figure}

\begin{figure}[htbp]
\centering 
\subfigure[Normal noise]{
\label{fig:Fig-12-1}
\includegraphics[width=0.85\textwidth]{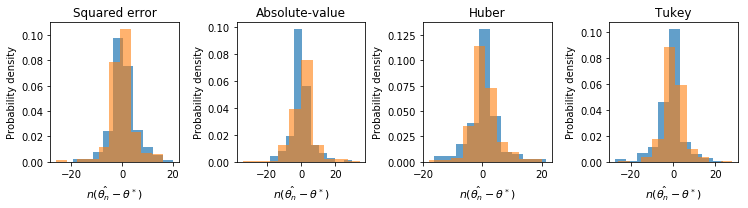}}
\subfigure[$T_3$ noise]{
\label{fig:Fig-12-2}
\includegraphics[width=0.85\textwidth]{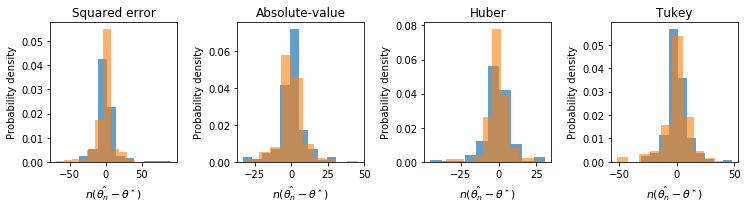}}
\subfigure[Cauchy noise]{
\label{fig:Fig-12-3}
\includegraphics[width=0.85\textwidth]{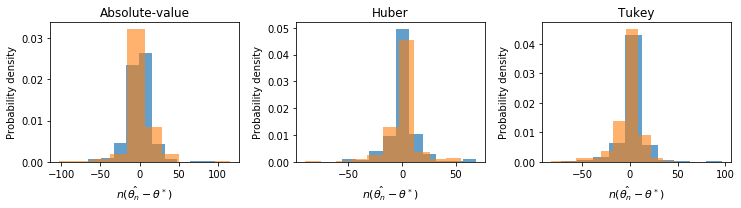}}
\caption{The blue histogram represents the distribution of $n(\hat{\theta_n}-\theta^*)$ for Template $E$. The orange histograms shows the simulated density of the midpoint of the minimizer interval of the marked Poisson process predicted by the theory. Based on $200$ repeats.}
\label{fig:Fig-12}
\end{figure}

%\begin{figure}[htbp]
%\centering 
%\includegraphics[width=14.5cm]{figures/2-hist-abnormal.png}
%\caption{Abnormal distribution of $n(\hat{\theta_n}-\theta^*)$ in non-smooth setting. The three cases are not covered by the theory in this paper. As anticipated, their performances are not good.}
%\label{fig:Fig-13}
%\end{figure}

We also ran some experiments with varying $n$ as before. We focused on Template E with the absolute-value loss under the $T_3$ noise distribution. The mean of $|n(\hat{\theta_n}-\theta^*)|$ is reported in Table~\ref{table4}, a box plot of the estimation error $|\hat{\theta_n}-\theta^*|$ is given in \figref{Fig-18}, and the distribution of $n(\hat{\theta_n}-\theta^*)$ is plotted in \figref{Fig-17}.

\begin{table}[htbp]
\centering
\begin{tabular}{ cccccc } 
\toprule  
n & 100 & 500 & 1000 & 5000 & 10000\\
\midrule
Mean of $|n(\hat{\theta_n}-\theta^*)|$ & 10.1872 & 5.0808 & 5.7420 & 5.7146 & 5.0682\\
\bottomrule
\end{tabular}
\caption{Mean of $|n(\hat{\theta_n}-\theta^*)|$ for Template $E$ with absolute-value loss under $T_3$ noise based on $1000$ repeats.}
\label{table4}
\end{table}

\begin{figure}[htbp]
\centering 
\includegraphics[width=10cm]{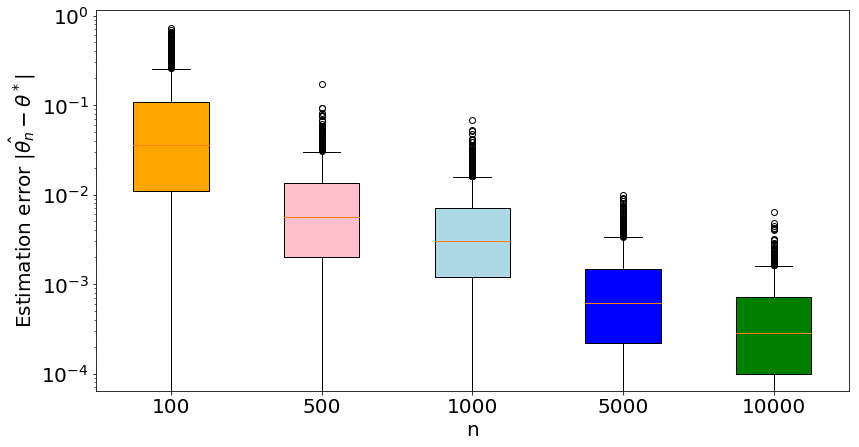}
\caption{Box plot of estimation error $|\hat{\theta_n}-\theta^*|$ for Template $E$ with absolute-value loss under $T_3$ noise based on $1000$ repeats.}
\label{fig:Fig-18}
\end{figure}

\begin{figure}[htbp]
\centering 
\includegraphics[width=15cm]{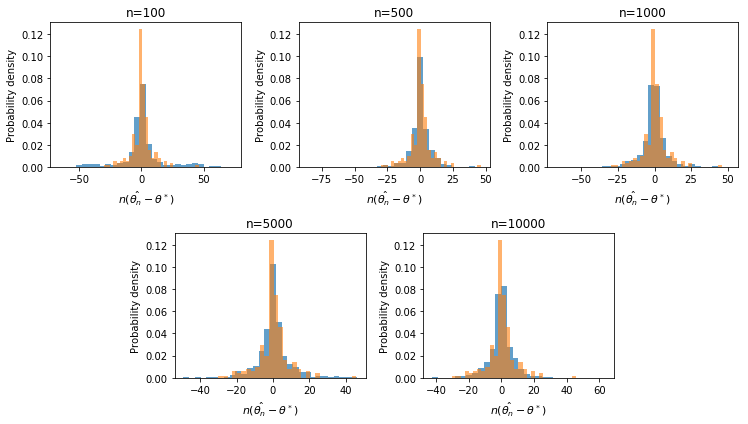}
\caption{The blue histograms present the distribution of $n(\hat{\theta_n}-\theta^*)$ for Template $E$ with absolute-value loss and $T_3$ noise. The orange histograms show the simulated density of marked Poisson process predicted by the theory. Based on $1000$ repeats.}
\label{fig:Fig-17}
\end{figure}

\subsection*{Acknowledgments}
We are particularly grateful to Lutz D\"umbgen for providing some illuminating intuition for the change point setting where the template is discontinuous; to Jason Schweinsberg for clarifying the notion of weak convergence towards a Poisson process; and to Nicolas Verzelen for important pointers to the literature and for his general support throughout the writing of the paper.
We are also thankful to Ioan Bejenaru, Richard Nickl, and Wenxin Zhou for helpful discussions and pointers.   
%This work was partially supported by XXX.

\bibliographystyle{chicago}
\bibliography{ref}

\appendix
\section{Miscellanea}

\subsection{Auxiliary results}

\begin{lem}
\label{lem:consistent}
Suppose Assumptions \ref{asp:template}, \ref{asp:noise}, and \ref{asp:loss} are in place. Then, \aspref{M} is fulfilled if {\em (i)} the loss is strictly convex, or {\em (ii)} the noise distribution is unimodal and the loss is either Lipschitz or has a Lipschitz derivative and $\phi$ has finite first moment. \aspref{M} is also fulfilled for the Huber loss and the absolute-value loss if the noise distribution has a unique median.
%; and for the 0-1 loss if the noise distribution is unimodal.
\end{lem}

\begin{proof}
Assume $\theta^* = 0$ without loss of generality.
For $x$ and $y$ given, we let $z = y - f(x)$.

The function $\theta \mapsto m_\theta(x,y)$ is piecewise continuous under our assumptions. Moreover, it is dominated since, by the properties of $\loss$,
\begin{align*}
m_\theta(x, y)
&= \loss(y - f(x-\theta)) \\
&= \loss(y - f(x) + f(x) - f(x-\theta)) \\
&\le \loss(|y - f(x)| + |f(x) - f(x-\theta)|) \\
&\le \overline{m}(x, y) := \loss(|y - f(x)| + 2 |f|_\infty),
\end{align*}
and $\E[\overline{m}(X,Y)] < \infty$ by \eqref{Z0} and the fact that the loss is slowly increasing.
Hence, $M$ is continuous by dominated convergence.
In addition, by the fact that $f$ is compactly supported, $m_\theta(x,y) \to \loss(y)$ when $\theta \to \pm \infty$, and so we also have $M(\theta) \to M(\infty) := \E[\loss(Y)]$ as $\theta \to \pm\infty$.
We have
\begin{align*}
M(\theta) = \E[\loss(Z + f(X) - f(X-\theta))], &&
M(\infty) = \E[\loss(Z + f(X))].
\end{align*}
The identifiability condition in \aspref{template} implies that $\P(f(X) - f(X-\theta) \ne 0) > 0$ when $\theta \ne 0$ and $\P(f(X) \ne 0) > 0$, so by conditioning on $X$, taking into account that $X$ and $Z$ are independent, it suffices to prove that $g(a) := \E[\loss(Z+a)]$ achieves its minimum uniquely at $a = 0$. Since $g$ is even, it suffices to prove that $g(a) > g(0)$ for any $a > 0$.

{\em Assume $\loss$ is convex.}
Then
\begin{align*}
g(a)
= \E[\loss(Z+a)]
&= \E[\loss(Z-a)] \\
&= \E[\tfrac12 \loss(Z+a) + \tfrac12 \loss(Z-a)] \\ 
&\ge \E[\loss(Z)]
= g(0),
\end{align*}
where in the second equality we used the fact that the distribution of $Z$ is symmetric about 0 and that $\loss$ is even, and in the inequality we used the fact that $\loss$ is convex.
That inequality is strict when the loss is strictly convex. This covers the squared error loss, for example.

For the Huber loss \eqref{Huber}, the inequality is strict unless $Z-a \ge c$ or $Z+a \le -c$ with probability 1, meaning that $\P(Z \in [-a-c,a+c]) = 0$; for the absolute-value loss, the inequality is strict unless $\sign(Z-a) = \sign(Z+a)$ with probability 1, meaning that $\P(Z \in [-a,a]) = 0$. In either case, this is only possible if the noise distribution does not put any mass in $[-a,a]$, which would then imply that any point in this entire interval is a median of the noise distribution.

{\em Assume that the loss is Lipschitz.}
In that case, by dominated convergence, $g$ is differentiable with $g'(a) = \E[\loss'(Z+a)]$, and all we have to prove is that $g'(a) \ge 0$ when the noise distribution is unimodal. (Recall that we work with $a > 0$.) The uniqueness would then come from the fact that $g$ is not constant.
We have
\begin{align*}
g'(a)
&= \int_{-\infty}^\infty \loss'(z+a) \phi(z) \d z \\
&= \int_0^\infty \loss'(z) \big\{\phi(z-a) - \phi(z+a)\big\} \d z.
\end{align*}
We claim that the integrand is non-negative. Indeed, $\loss'(z) \ge 0$ for $z \ge 0$ since $\loss$ is non-decreasing away from the origin per \aspref{loss}. Also, using the fact that $\phi$ is non-increasing on $[0,\infty)$, we reason as follows: if $z \ge a$, then $z+a \ge z-a \ge 0$, implying $\phi(z+a) \le \phi(z-a)$; and if $z \le a$, then $a+z \ge a-z \ge 0$, implying $\phi(a+z) \le \phi(a-z)$, in turn implying $\phi(z+a) \le \phi(z-a)$ since $\phi$ is even. Our claim is thus established.
This covers the Tukey loss, for example.

{\em Assume that the loss has a Lipschitz derivative.} 
In that case we can reduce this to the previous case by observing that $g$ above remains differentiable as long as $\phi$ has a finite first moment. Indeed, 
\begin{equation}
\E[|\loss'(Z+a)|] \le \E[|\loss''|_\infty |Z+a|] \le |\loss''|_\infty (\E[|Z|]+a) < \infty,
\end{equation}
for all $a \ge 0$.
%Finally, we consider the 0-1 loss \eqref{0-1} loss separately. 
%For that loss, 
%\begin{equation}
%g(a)
%= \int_{-\infty}^{-a-c} \phi(z) \d z + \int_{-a+c}^\infty \phi(z) \d z,
%\end{equation}
%and differentiating with respect to $a$ gives
%\begin{equation}
%g'(a) 
%= -\phi(-a-c) + \phi(-a+c)
%= -\phi(a+c) + \phi(a-c),
%\end{equation}
%using the fact that $\phi$ is even.
%If $a \ge c$, then $a+c \ge a-c \ge 0$, implying that $\phi(a+c) \le \phi(a-c)$; and if $a \le c$, then $c+a \ge c-a \ge 0$, implying that $\phi(c+a) \le \phi(c-a)$, in turn implying $\phi(a+c) \le \phi(a-c)$ since $\phi$ is even.
\end{proof}

\begin{lem}
\label{lem:poisson}
Suppose that, on the real line, $X_1, \dots, X_n$ are iid from some density $\lambda$, and independently, that $A_1, \dots, A_n$ are iid with distribution $\Psi$. Fix a finite subset $d_1 < \cdots < d_p$ where $\lambda$ can be taken to be continuous. Define $W_{n,j}(t) = \sum_i \{d_j < X_i \le d_j + t/n\} A_i$. Then $W_{n,j}$ converges weakly to the marked Poisson process with intensity $\lambda(d_j)$ on the positive real line and mark distribution $\Psi$. Moreover, $W_{n,1}, \dots, W_{n,p}$ are asymptotically independent.  
\end{lem}

\begin{proof}
It is suffices to establish the statement on intervals. Without loss of generality, we consider the unit interval, so that we only consider $t \in [0,1]$. Everywhere, $n$ is large enough that $d_{j+1} - d_j > 1/n$ for all $j$. This guarantees that, for any $t \in [0,1]$, the intervals $[d_j, d_j+t/n]$ are disjoint.
Let $\{A_{j,i}: i \ge 1, j = 1, \dots, p\}$ be iid from $\Psi$. Then it's easy to see that $W_{n,1}, \dots, W_{n,p}$, jointly, have the same distribution as $\tilde W_{n,1}, \dots, \tilde W_{n,p}$, where 
\begin{equation}
\tilde W_{n,j}(t) := \sum_{i=1}^{N_{n,j}(t)} A_{j,i}, \quad
N_{n,j}(t) := \#\{i : d_j < X_i \le d_j + t/n\}.
\end{equation}
It is thus sufficient to show that $N_{n,j}$ converges weakly to the Poisson process with intensity $\lambda(d_j)$ on the positive real line and that $N_{n,1}, \dots, N_{n,p}$ are asymptotically independent.
By \cite[Th 12.6]{billingsley1999convergence}, it suffices to look at the finite-dimensional distributions.
Therefore, fix $0 \le t_1 < \cdots < t_k \le 1$ and $\{n_{j,s}: j = 1, \dots, p; s = 1, \dots, k\}$ non-negative integers. Based on the `balls-in-bins' dependency structure of these counts, we have
\begin{align}
&\P\big(N_{n,j}(t_s) \le n_{j,s}, \forall j \in [p], s \in [k]\big) \\
&= \P\big(N_{n,1}(t_s) \le n_{1,s}, \forall s \in [k]\big) \\ 
&\quad \times \P\big(N_{n,2}(t_s) \le n_{2,s}, \forall s \in [k]\mid N_{n,1}(t_k) \le n_{1,k}\big) \\
&\quad\quad \times \cdots \\
&\quad\quad\quad \times \P\big(N_{n,p}(t_s) \le n_{p,s}, \forall s \in [k] \mid N_{n,1}(t_k) \le n_{1,k}, \dots, N_{n,p-1}(t_k) \le n_{p-1,k}\big),
\end{align}
with
\begin{align*}
\P\big(N_{n,j}(t_s) \le n_{j,s}, \forall s \in [k] \mid N_{n,1}(t_k) \le n_{1,k}, \dots, N_{n,j-1}(t_k) \le n_{j-1,k}\big) \\
\begin{cases}
\ge \P\big(N_{n,j}(t_s) \le n_{j,s}, \forall s \in [k]\big) \\
\le \P\big(N_{n - r_j,j}(t_s) \le n_{j,s}, \forall s \in [k]\big), \quad r_j := n_{1,k}+\cdots+n_{j-1,k}.
\end{cases}
\end{align*}
Note that $r_1, \dots, r_p$ are all fixed, and it's easy to convince oneself that it suffices at this point to show that, for each $j$, 
\begin{equation}
\P\big(N_{n,j}(t_s) \le n_{j,s}, \forall s \in [k]\big)
\xrightarrow{n \to \infty} \P\big(N_j(t_s) \le n_{j,s}, \forall s \in [k]\big),
\end{equation} 
were $N_j$ is a Poisson process with intensity $\lambda(d_j)$ on the positive real line, in other words, that $N_{n,j}$ converges weakly to $N_j$.
We do so by looking at the probability mass function instead of the cumulative distribution function. With $t_1 < \cdots < t_k$ and $\{n_{j,s}\}$ generic, we have
\begin{align*}
&\P\big(N_{n,j}(t_1) = n_{j,1}, \dots, N_{n,j}(t_k) = n_{j,k}\big) \\
&=\P\big(N_{n,j}(t_1) = n_{j,1}, N_{n,j}(t_2) - N_{n,j}(t_1) = n_{j,2}-n_{j,1}, \dots, N_{n,j}(t_k) - N_{n,j}(t_{k-1}) = n_{j,k}-n_{j,k-1}\big) \\
&\sim \P\big(N_{n,j}(t_1) = n_{j,1}\big) 
\times \P\big(N_{n,j}(t_2) - N_{n,j}(t_1) = n_{j,2}-n_{j,1}\big) \\
&\quad\quad\quad\quad \times \cdots \times \P\big(N_{n,j}(t_k) - N_{n,j}(t_{k-1}) = n_{j,k}-n_{j,k-1}\big), \quad \text{as } n\to\infty,
\end{align*}
where the approximation follows the same arguments used above to prove the asymptotic independence of the various count processes. We then conclude with the fact that, for any $m\ge0$ integer,
\begin{align*}
\P\big(N_{n,j}(t_s) - N_{n,j}(t_{s-1}) = m\big)
&= \#\{i : d_j+t_{s-1}/n < X_i \le d_j+t_s/n\} \\
&\xrightarrow{n\to\infty} \lambda(d_j) (t_s-t_{s-1})
= \P\big(N_j(t_s) - N_j(t_{s-1}) = m\big),
\end{align*}
using the fact that $\lambda$ is continuous at $d_j$.
%This follows from standard arguments, based on the fact that $N_{n,j}$ is the empirical count process based on the $X$'s that fall in the interval $[d_j, d_j+1/n]$ (remember that we are focused on $t \in [0,1]$). By a change in location and scale, from $x \mapsto n(x-d_j)$, the resulting sample is asymptotically iid uniform with random sample 
\end{proof}

\begin{lem}
\label{lem:compact}
Consider the setting of \secref{flexible_non-smooth}.
Assume in addition that $\E[\loss(Z)] < \loss(\infty)$ and that $\E[\loss(Z)] < \E[\loss(Z+a)]$ for all $a \ne 0$.
Then there is a compact subset $\Theta_0 \subset \Theta$ of the form $[-B,B] \times [-B,B] \times [1/B, B]$ for some $B \ge 1$, such that, with probability tending to 1, any minimizer $\hat\theta_n$ of $\widehat M_n$ is in $\Theta_0$.
\end{lem}

\begin{proof}
Let $a = \inf\{x : f(x) > 0\}$ and $b = \sup\{x: f(x) > 0\}$, and assume that $\lambda$ is supported on $[-c, c]$. 
Since we know that $f_{\theta^*}$ has support contained within the support of $\lambda$, we can restrict attention to those $\theta$ such that $f_\theta$ has support contained within the support of $\lambda$, which is equivalent to $\xi$ and $\nu$ satisfying $\xi + a \nu \ge -c$ and $\xi + b \nu \le c$. Note that this implies that $-c |b+a|/(b-a) \le \xi \le c |b+a|/(b-a)$ and $\nu \le 2c/(b-a)$.

We now show that there is $\nu_0 > 0$ such that, if $\nu \le \nu_0$, then $\widehat M_n(\beta, \xi, \nu)$ is, asymptotically, bounded away (and above) from $\widehat M_n(\theta^*)$, regardless of $\beta$ and $\xi$.
Henceforth, we assume without loss of generality that $\theta^* = (1,0,1)$.

On the one hand, by the law of large numbers, in probability as $n \to \infty$, 
\begin{equation}
\widehat M_n(\theta^*) \to M(\theta^*) = \E[\loss(Z)].
\end{equation}
On the other hand,
\begin{align}
\inf_{\beta, \xi, \nu \le \nu_0} \widehat M_n(\beta, \xi, \nu)
&= \inf_{\beta, \xi, \nu \le \nu_0} \frac1n \sum_i \loss\big(Z_i + f(X_i) - \beta f((X_i-\xi)/\nu)\big) \\
&\ge \inf_{\beta, \xi, \nu \le \nu_0} \frac1n \sum_i \loss\big(Z_i + f(X_i)) \IND{X_i \notin \xi \pm \nu K} \\
&= \inf_{|I| \le 2 \nu_0 K} \frac1n \sum_i \loss\big(Z_i + f(X_i)) \IND{X_i \notin I},
\end{align}
where $I$ above is a closed interval of length $|I|$.
The inequality relies on the fact that $\beta f((x-\xi)/\nu) = 0$ when $x \notin \xi \pm \nu K$ and the fact that the loss function is non-negative.
Now,
\begin{align}
\frac1n \sum_i \loss\big(Z_i + f(X_i)) \IND{X_i \notin I}
= \frac1n \sum_i g_I(X_i, Z_i), \quad
g_I(x,z) := \loss\big(z + f(x)) \IND{x \notin I},
\end{align}
and the class of functions $\{g_I : |I| \le 2 \nu_0 K\}$ is clearly Glivenko--Cantelli. Hence,
\begin{equation}
\inf_{|I| \le 2 \nu_0 K} \frac1n \sum_i \loss\big(Z_i + f(X_i)) \IND{X_i \notin I}
\xrightarrow{n \to \infty} \inf_{|I| \le 2 \nu_0 K} \E[\loss\big(Z + f(X)) \IND{X \notin I}].
\end{equation}
In turn, by dominated convergence, as $\nu_0 \to 0$, the last expression converges to $\E[\loss(Z + f(X))]$, which is strictly larger than $\E[\loss(Z)]$ by (the equivalent of) \aspref{M}. In particular, if $\nu_0$ is so small that the last infimum is $> \E[\loss(Z)]$, we guarantee that, with probability tending to 1, $\hat\nu_n \ge \nu_0$.

So far, we showed that we can restrict $(\xi,\nu)$ in order for $f_\theta$ to have support inside that of $\lambda$, and that we may add a lower bound on $\nu$ (away from $0$). It remains to show that, under these conditions, we may also add a bound on $|\beta|$. 
Fix $\eps > 0$ and define $S_\eps = \{x: |f(x)| > \eps\}$.
Then, with the infimum over $\xi$ and $\nu$ restricted as described above, for $\beta_0 > 0$ and $z_0 > 0$, we have
\begin{align*}
&\inf_{|\beta| \ge \beta_0, \xi, \nu} \widehat M_n(\beta, \xi, \nu) \\
&\ge \inf_{|\beta| \ge \beta_0, \xi, \nu}\quad \frac1n \sum_i \inf_{|a_i| \ge \eps} \loss\big(Z_i + f(X_i) - \beta a_i\big) \IND{X_i \in \xi + \nu S_\eps} + \frac1n \sum_i \loss\big(Z_i + f(X_i)\big) \IND{X_i \notin \xi + \nu S_0} \\
&\ge \inf_{\xi, \nu}\quad \frac1n \sum_i \loss((\beta_0\eps -Z_i -f(X_i))_+) \IND{X_i \in \xi + \nu S_\eps} + \frac1n \sum_i \loss\big(Z_i + f(X_i)\big) \IND{X_i \notin \xi + \nu S_0} \\
&\xrightarrow{n \to \infty} \inf_{\xi, \nu}\quad \E[\loss((\beta_0\eps -Z -f(X))_+) \IND{X \in \xi+\nu S_\eps}] + \E[\loss(Z+f(X)) \IND{X \notin \xi+\nu S_0}] \\
&\xrightarrow{\beta_0 \to \infty} \inf_{\xi, \nu}\quad  \loss(\infty) \P(X \in \xi+\nu S_\eps) + \E[\loss(Z+f(X)) \IND{X \notin \xi+\nu S_0}] \\
&\xrightarrow{\eps \to 0} \inf_{\xi, \nu}\quad \loss(\infty) \P(X \in \xi+\nu S_0) + \E[\loss(Z+f(X)) \IND{X \notin \xi+\nu S_0}]. 
\end{align*}
The first limit uses a uniform law of large numbers as we did above, only here it is based on the fact that the collection of sets of the form $\xi+\nu S_\eps$, as $\xi$ and $\nu$ vary as they do here, has finite VC dimension, and similarly for the collection of sets of the form $\xi+\nu S_0$. We compare the last expression with $M(\theta^*)$ to get
\begin{align}
&\loss(\infty) \P(X \in \xi+\nu S_0) + \E[\loss(Z+f(X)) \IND{X \notin \xi+\nu S_0}] - M(\theta^*) \\
&= (\loss(\infty)-\E[\loss(Z)]) \P(X \in \xi+\nu S_0)  + \E[\varphi(f(X)) \IND{X \notin \xi+\nu S_0}] =: Q(\xi, \nu),
\end{align}
where $\varphi(a) := \E[\loss(Z+a)]-\E[\loss(Z)]$.
By assumption, $\loss(\infty)-\E[\loss(Z)] > 0$ and $\varphi(a) > 0$ when $a \ne 0$, so that $Q(\xi, \nu) = 0$ if and only if $\P(X \in \xi+\nu S_0) = 0$ and $\P(X \notin \xi+\nu S_0, X \in S_0) = 0$, which together imply that $\P(X \in S_0) = 0$, which would contradict the fact that $f \nequiv 0$ on the support of $\lambda$. Hence, $Q(\xi,\nu) > 0$, and since $Q$ is continuous by dominated convergence, and $(\xi,\nu)$ is in a compact set, we have that $\inf_{\xi, \nu} Q(\xi,\nu) > 0$.
All in all, we are able to conclude that there is $\beta_0 > 0$ such that, with probability tending to 1, $|\hat\beta_n| \le \beta_0$. And this is the only thing that was left to prove to establish the lemma.
\end{proof}

\end{document}

%% file: notation.tex
\def\env{V}

\def\risk{{\sf risk}}

\def\loss{{\sf L}}

\def\d{{\rm d}}

\newtheorem{theorem}{Theorem}[section]

\newtheorem{lem}[theorem]{Lemma}
\newtheorem{cor}[theorem]{Corollary}

\theoremstyle{remark}

\newtheorem{rem}[theorem]{Remark}
\newtheorem{asp}{Assumption}
\newtheorem{exa}[theorem]{Example}

\def\beq{\begin{equation}} % \setcounter{equation}{1}}
\def\eeq{\end{equation}}
\def\beqn{\begin{eqnarray*}}
\def\eeqn{\end{eqnarray*}}
\def\Bitem{\begin{itemize}\setlength{\itemsep}{.2in}}
\def\bitem{\begin{itemize}\setlength{\itemsep}{.05in}}
\def\eitem{\end{itemize}}
\def\Benum{\begin{enumerate}\setlength{\itemsep}{.2in}}
\def\benum{\begin{enumerate}\setlength{\itemsep}{.05in}}
\def\eenum{\end{enumerate}}
\def\bmult{\begin{multline*}}
\def\emult{\end{multline*}}
\def\bcenter{\begin{center}}
\def\ecenter{\end{center}}
\def\bframe{\begin{frame}}
\def\eframe{\end{frame}}

\newcommand{\thmref}[1]{Theorem~\ref{thm:#1}}

\newcommand{\lemref}[1]{Lemma~\ref{lem:#1}}
\newcommand{\secref}[1]{Section~\ref{sec:#1}}
\newcommand{\figref}[1]{Figure~\ref{fig:#1}}

\newcommand{\aspref}[1]{Asumption~\ref{asp:#1}}
\newcommand{\remref}[1]{Remark~\ref{rem:#1}}

\DeclareMathOperator*{\argmax}{arg\, max}
\DeclareMathOperator*{\argmin}{arg\, min}

\DeclareMathOperator{\dist}{dist}

\def\cD{\mathcal{D}}

\def\cM{\mathcal{M}}

\def\cW{\mathcal{W}}

\def\bbI{\mathbb{I}}

\def\bbR{\mathbb{R}}

\newcommand{\E}{\operatorname{\mathbb{E}}}
\renewcommand{\P}{\operatorname{\mathbb{P}}}

\newcommand\indep{\protect\mathpalette{\protect\independenT}{\perp}}
\def\independenT#1#2{\mathrel{\rlap{$#1#2$}\mkern2mu{#1#2}}}

\def\eps{\varepsilon}

\DeclareMathOperator{\sign}{sign}

\def\1{\mathbbm{1}}
\newcommand{\IND}[1]{\bbI{\{ #1 \}}}

%% file: paper-v12.bbl
\begin{thebibliography}{}

\bibitem[\protect\citeauthoryear{Arias-Castro, Cand\`es, and
  Durand}{Arias-Castro et~al.}{2011}]{cluster}
Arias-Castro, E., E.~J. Cand\`es, and A.~Durand (2011).
\newblock Detection of an anomalous cluster in a network.
\newblock {\em The Annals of Statistics\/}~{\em 39\/}(1), 278--304.

\bibitem[\protect\citeauthoryear{Arias-Castro, Castro, T{\'a}nczos, and
  Wang}{Arias-Castro et~al.}{2018}]{arias2018distribution}
Arias-Castro, E., R.~M. Castro, E.~T{\'a}nczos, and M.~Wang (2018).
\newblock Distribution-free detection of structured anomalies: Permutation and
  rank-based scans.
\newblock {\em Journal of the American Statistical Association\/}~{\em
  113\/}(522), 789--801.

\bibitem[\protect\citeauthoryear{Arias-Castro, Donoho, and Huo}{Arias-Castro
  et~al.}{2005}]{arias2005near}
Arias-Castro, E., D.~L. Donoho, and X.~Huo (2005).
\newblock Near-optimal detection of geometric objects by fast multiscale
  methods.
\newblock {\em IEEE Transactions on Information Theory\/}~{\em 51\/}(7),
  2402--2425.

\bibitem[\protect\citeauthoryear{Bai}{Bai}{1997}]{bai1997estimation}
Bai, J. (1997).
\newblock Estimation of a change point in multiple regression models.
\newblock {\em Review of Economics and Statistics\/}~{\em 79\/}(4), 551--563.

\bibitem[\protect\citeauthoryear{Basseville and Nikiforov}{Basseville and
  Nikiforov}{1993}]{basseville1993detection}
Basseville, M. and I.~V. Nikiforov (1993).
\newblock {\em Detection of abrupt changes: theory and application}, Volume
  104.
\newblock Englewood Cliffs, New Jersey: Prentice Hall.

\bibitem[\protect\citeauthoryear{Belongie, Malik, and Puzicha}{Belongie
  et~al.}{2002}]{belongie2002shape}
Belongie, S., J.~Malik, and J.~Puzicha (2002).
\newblock Shape matching and object recognition using shape contexts.
\newblock {\em IEEE Transactions on Pattern Analysis and Machine
  Intelligence\/}~{\em 24\/}(4), 509--522.

\bibitem[\protect\citeauthoryear{Bickel, Klaassen, Ritov, and Wellner}{Bickel
  et~al.}{1998}]{bickel1998efficient}
Bickel, P., C.~Klaassen, Y.~Ritov, and J.~Wellner (1998).
\newblock {\em Efficient and adaptive estimation for semiparametric models}.
\newblock Springer.

\bibitem[\protect\citeauthoryear{Bigot, Gamboa, and Vimond}{Bigot
  et~al.}{2009}]{bigot2009estimation}
Bigot, J., F.~Gamboa, and M.~Vimond (2009).
\newblock Estimation of translation, rotation, and scaling between noisy images
  using the {F}ourier--{M}ellin transform.
\newblock {\em SIAM Journal on Imaging Sciences\/}~{\em 2\/}(2), 614--645.

\bibitem[\protect\citeauthoryear{Billingsley}{Billingsley}{1999}]{billingsley1999convergence}
Billingsley, P. (1999).
\newblock {\em Convergence of probability measures}.
\newblock John wiley \& Sons.

\bibitem[\protect\citeauthoryear{Boutsikas and Koutras}{Boutsikas and
  Koutras}{2006}]{boutsikas}
Boutsikas, M.~V. and M.~V. Koutras (2006).
\newblock On the asymptotic distribution of the discrete scan statistic.
\newblock {\em Journal of Applied Probability\/}~{\em 43\/}(4), 1137--1154.

\bibitem[\protect\citeauthoryear{Brodsky and Darkhovsky}{Brodsky and
  Darkhovsky}{2013}]{brodsky2013nonparametric}
Brodsky, E. and B.~S. Darkhovsky (2013).
\newblock {\em Nonparametric methods in change point problems}, Volume 243.
\newblock Springer Science$+$Business Media.

\bibitem[\protect\citeauthoryear{Brunelli}{Brunelli}{2009}]{brunelli2009template}
Brunelli, R. (2009).
\newblock {\em Template matching techniques in computer vision: theory and
  practice}.
\newblock John Wiley \& Sons.

\bibitem[\protect\citeauthoryear{Chen and Gupta}{Chen and
  Gupta}{2011}]{chen2011parametric}
Chen, J. and A.~K. Gupta (2011).
\newblock {\em Parametric statistical change point analysis: with applications
  to genetics, medicine, and finance}.
\newblock Springer Science$+$Business Media.

\bibitem[\protect\citeauthoryear{Collier and Dalalyan}{Collier and
  Dalalyan}{2012}]{collier2012minimax}
Collier, O. and A.~S. Dalalyan (2012).
\newblock Minimax hypothesis testing for curve registration.
\newblock {\em Electronic Journal of Statistics\/}~{\em 6}, 1129--1154.

\bibitem[\protect\citeauthoryear{Collier and Dalalyan}{Collier and
  Dalalyan}{2015}]{collier2015curve}
Collier, O. and A.~S. Dalalyan (2015).
\newblock Curve registration by nonparametric goodness-of-fit testing.
\newblock {\em Journal of Statistical Planning and Inference\/}~{\em 162},
  20--42.

\bibitem[\protect\citeauthoryear{Cs{\"o}rg{\"o} and Horv{\'a}th}{Cs{\"o}rg{\"o}
  and Horv{\'a}th}{1997}]{csorgo1997limit}
Cs{\"o}rg{\"o}, M. and L.~Horv{\'a}th (1997).
\newblock {\em Limit theorems in change-point analysis}.
\newblock John Wiley \& Sons Inc.

\bibitem[\protect\citeauthoryear{Desolneux, Moisan, and Morel}{Desolneux
  et~al.}{2003}]{morel}
Desolneux, A., L.~Moisan, and J.-M. Morel (2003).
\newblock Maximal meaningful events and applications to image analysis.
\newblock {\em The Annals of Statistics\/}~{\em 31\/}(6), 1822--1851.

\bibitem[\protect\citeauthoryear{D{\"o}ring}{D{\"o}ring}{2011}]{doring2011convergence}
D{\"o}ring, M. (2011).
\newblock Convergence in distribution of multiple change point estimators.
\newblock {\em Journal of Statistical Planning and Inference\/}~{\em 141\/}(7),
  2238--2248.

\bibitem[\protect\citeauthoryear{Dudley}{Dudley}{1967}]{dudley1967sizes}
Dudley, R.~M. (1967).
\newblock The sizes of compact subsets of {H}ilbert space and continuity of
  {G}aussian processes.
\newblock {\em Journal of Functional Analysis\/}~{\em 1\/}(3), 290--330.

\bibitem[\protect\citeauthoryear{D\"umbgen}{D\"umbgen}{1991}]{dumbgen1991asymptotic}
D\"umbgen, L. (1991).
\newblock The asymptotic behavior of some nonparametric change-point
  estimators.
\newblock {\em The Annals of Statistics\/}~{\em 19\/}(3), 1471--1495.

\bibitem[\protect\citeauthoryear{Ferger}{Ferger}{1994}]{ferger1994change}
Ferger, D. (1994).
\newblock Change-point estimators in case of small disorders.
\newblock {\em Journal of Statistical Planning and Inference\/}~{\em 40\/}(1),
  33--49.

\bibitem[\protect\citeauthoryear{Ferger}{Ferger}{2001}]{ferger2001exponential}
Ferger, D. (2001).
\newblock Exponential and polynomial tailbounds for change-point estimators.
\newblock {\em Journal of Statistical Planning and Inference\/}~{\em
  92\/}(1-2), 73--109.

\bibitem[\protect\citeauthoryear{Ferger}{Ferger}{2004}]{ferger2004continuous}
Ferger, D. (2004).
\newblock A continuous mapping theorem for the argmax-functional in the
  non-unique case.
\newblock {\em Statistica Neerlandica\/}~{\em 58\/}(1), 83--96.

\bibitem[\protect\citeauthoryear{Frick, Munk, and Sieling}{Frick
  et~al.}{2014}]{frick2014multiscale}
Frick, K., A.~Munk, and H.~Sieling (2014).
\newblock Multiscale change point inference.
\newblock {\em Journal of the Royal Statistical Society - Series B: Statistical
  Methodology\/}~{\em 76\/}(3), 495--580.

\bibitem[\protect\citeauthoryear{Gamboa, Loub\`es, and Maza}{Gamboa
  et~al.}{2007}]{gamboa2007semi}
Gamboa, F., J.-M. Loub\`es, and E.~Maza (2007).
\newblock Semi-parametric estimation of shifts.
\newblock {\em Electronic Journal of Statistics\/}~{\em 1}, 616--640.

\bibitem[\protect\citeauthoryear{Glaz and Balakrishnan}{Glaz and
  Balakrishnan}{2012}]{glaz2012scan}
Glaz, J. and N.~Balakrishnan (2012).
\newblock {\em Scan statistics and applications}.
\newblock Springer Science$+$Business Media.

\bibitem[\protect\citeauthoryear{Glaz, Naus, and Wallenstein}{Glaz
  et~al.}{2001}]{glaz2001scan}
Glaz, J., J.~Naus, and S.~Wallenstein (2001).
\newblock {\em Scan statistics}.
\newblock Springer.

\bibitem[\protect\citeauthoryear{Glaz, Pozdnyakov, and Wallenstein}{Glaz
  et~al.}{2009}]{glaz2009scan}
Glaz, J., V.~Pozdnyakov, and S.~Wallenstein (2009).
\newblock {\em Scan statistics: methods and applications}.
\newblock Springer Science$+$Business Media.

\bibitem[\protect\citeauthoryear{Glaz and Zhang}{Glaz and
  Zhang}{2004}]{glaz2004multiple}
Glaz, J. and Z.~Zhang (2004).
\newblock Multiple window discrete scan statistics.
\newblock {\em Journal of Applied Statistics\/}~{\em 31\/}(8), 967--980.

\bibitem[\protect\citeauthoryear{Haiman and Preda}{Haiman and
  Preda}{2006}]{haiman2006estimation}
Haiman, G. and C.~Preda (2006).
\newblock Estimation for the distribution of two-dimensional discrete scan
  statistics.
\newblock {\em Methodology and Computing in Applied Probability\/}~{\em
  8\/}(3), 373--382.

\bibitem[\protect\citeauthoryear{Hajnal and Hill}{Hajnal and
  Hill}{2001}]{hajnal2001medical}
Hajnal, J.~V. and D.~L. Hill (2001).
\newblock {\em Medical image registration}.
\newblock CRC press.

\bibitem[\protect\citeauthoryear{Hall and Jin}{Hall and
  Jin}{2010}]{hall2010innovated}
Hall, P. and J.~Jin (2010).
\newblock Innovated higher criticism for detecting sparse signals in correlated
  noise.
\newblock {\em The Annals of Statistics\/}~{\em 38\/}(3), 1686--1732.

\bibitem[\protect\citeauthoryear{H\"ardle and Marron}{H\"ardle and
  Marron}{1990}]{hardle1990semiparametric}
H\"ardle, W. and J.~Marron (1990).
\newblock Semiparametric comparison of regression curves.
\newblock {\em The Annals of Statistics\/}~{\em 18\/}(1), 63--89.

\bibitem[\protect\citeauthoryear{He and Severini}{He and
  Severini}{2010}]{he2010asymptotic}
He, H. and T.~A. Severini (2010).
\newblock Asymptotic properties of maximum likelihood estimators in models with
  multiple change points.
\newblock {\em Bernoulli\/}~{\em 16\/}(3), 759--779.

\bibitem[\protect\citeauthoryear{Hinkley}{Hinkley}{1970}]{hinkley1970inference}
Hinkley, D. (1970).
\newblock Inference about the change-point in a sequence of random variables.
\newblock {\em Biometrika\/}~{\em 57\/}(1), 1--17.

\bibitem[\protect\citeauthoryear{Hjelm{\aa}s and Low}{Hjelm{\aa}s and
  Low}{2001}]{hjelmaas2001face}
Hjelm{\aa}s, E. and B.~K. Low (2001).
\newblock Face detection: A survey.
\newblock {\em Computer Vision and Image Understanding\/}~{\em 83}, 236--274.

\bibitem[\protect\citeauthoryear{Jeng, Cai, and Li}{Jeng
  et~al.}{2010}]{jeng2010optimal}
Jeng, X.~J., T.~T. Cai, and H.~Li (2010).
\newblock Optimal sparse segment identification with application in copy number
  variation analysis.
\newblock {\em Journal of the American Statistical Association\/}~{\em
  105\/}(491), 1156--1166.

\bibitem[\protect\citeauthoryear{Jiang}{Jiang}{2002}]{jiang}
Jiang, T. (2002).
\newblock Maxima of partial sums indexed by geometrical structures.
\newblock {\em The Annals of Probability\/}~{\em 30\/}(4), 1854--1892.

\bibitem[\protect\citeauthoryear{Kabluchko}{Kabluchko}{2011}]{kabluchko2011extremes}
Kabluchko, Z. (2011).
\newblock Extremes of the standardized {G}aussian noise.
\newblock {\em Stochastic Processes and their Applications\/}~{\em 121\/}(3),
  515--533.

\bibitem[\protect\citeauthoryear{Kim and Pollard}{Kim and
  Pollard}{1990}]{kim1990cube}
Kim, J. and D.~Pollard (1990).
\newblock Cube root asymptotics.
\newblock {\em The Annals of Statistics\/}~{\em 18\/}(1), 191--219.

\bibitem[\protect\citeauthoryear{Kneip and Engel}{Kneip and
  Engel}{1995}]{kneip1995model}
Kneip, A. and J.~Engel (1995).
\newblock Model estimation in nonlinear regression under shape invariance.
\newblock {\em The Annals of Statistics\/}~{\em 23\/}(2), 551--570.

\bibitem[\protect\citeauthoryear{Kneip and Gasser}{Kneip and
  Gasser}{1988}]{kneip1988convergence}
Kneip, A. and T.~Gasser (1988).
\newblock Convergence and consistency results for self-modeling nonlinear
  regression.
\newblock {\em The Annals of Statistics\/}~{\em 16\/}(1), 82--112.

\bibitem[\protect\citeauthoryear{Kneip and Gasser}{Kneip and
  Gasser}{1992}]{kneip1992statistical}
Kneip, A. and T.~Gasser (1992).
\newblock Statistical tools to analyze data representing a sample of curves.
\newblock {\em The Annals of Statistics\/}~{\em 20\/}(3), 1266--1305.

\bibitem[\protect\citeauthoryear{K{\"o}nig, Munk, and Werner}{K{\"o}nig
  et~al.}{2020}]{konig2020multidimensional}
K{\"o}nig, C., A.~Munk, and F.~Werner (2020).
\newblock Multidimensional multiscale scanning in exponential families: Limit
  theory and statistical consequences.
\newblock {\em The Annals of Statistics\/}~{\em 48\/}(2), 655--678.

\bibitem[\protect\citeauthoryear{Korostelev}{Korostelev}{1988}]{korostelev1988minimax}
Korostelev, A. (1988).
\newblock On minimax estimation of a discontinuous signal.
\newblock {\em Theory of Probability \& Its Applications\/}~{\em 32\/}(4),
  727--730.

\bibitem[\protect\citeauthoryear{Kosorok}{Kosorok}{2008}]{kosorok2008introduction}
Kosorok, M.~R. (2008).
\newblock {\em Introduction to empirical processes and semiparametric
  inference}.
\newblock Springer Science+Business Media.

\bibitem[\protect\citeauthoryear{Kou}{Kou}{2017}]{kou2017identifying}
Kou, J. (2017).
\newblock Identifying the support of rectangular signals in {G}aussian noise.
\newblock {\em arXiv preprint arXiv:1703.06226\/}.

\bibitem[\protect\citeauthoryear{Lan, Banerjee, and Michailidis}{Lan
  et~al.}{2009}]{lan2009change}
Lan, Y., M.~Banerjee, and G.~Michailidis (2009).
\newblock Change-point estimation under adaptive sampling.
\newblock {\em The Annals of Statistics\/}~{\em 37\/}(4), 1752--1791.

\bibitem[\protect\citeauthoryear{Lawton, Sylvestre, and Maggio}{Lawton
  et~al.}{1972}]{lawton1972self}
Lawton, W., E.~Sylvestre, and M.~Maggio (1972).
\newblock Self modeling nonlinear regression.
\newblock {\em Technometrics\/}~{\em 14\/}(3), 513--532.

\bibitem[\protect\citeauthoryear{Lehmann and Romano}{Lehmann and
  Romano}{2005}]{lehmann2005testing}
Lehmann, E.~L. and J.~P. Romano (2005).
\newblock {\em Testing statistical hypotheses}.
\newblock Springer Science$+$Business Media.

\bibitem[\protect\citeauthoryear{Lowe}{Lowe}{1999}]{lowe1999object}
Lowe, D.~G. (1999).
\newblock Object recognition from local scale-invariant features.
\newblock In {\em International Conference on Computer Vision}, Volume~2, pp.\
  1150--1157. IEEE.

\bibitem[\protect\citeauthoryear{Mauer}{Mauer}{2018}]{mauer2018least}
Mauer, R. (2018).
\newblock {\em Least squares estimation in multiple change-point models}.
\newblock Ph.\ D. thesis, Technische Universit\"at Dresden.

\bibitem[\protect\citeauthoryear{Naus and Wallenstein}{Naus and
  Wallenstein}{2004}]{naus2004multiple}
Naus, J.~I. and S.~Wallenstein (2004).
\newblock Multiple window and cluster size scan procedures.
\newblock {\em Methodology and Computing in Applied Probability\/}~{\em
  6\/}(4), 389--400.

\bibitem[\protect\citeauthoryear{Perry, Weed, Bandeira, Rigollet, and
  Singer}{Perry et~al.}{2019}]{perry2019sample}
Perry, A., J.~Weed, A.~S. Bandeira, P.~Rigollet, and A.~Singer (2019).
\newblock The sample complexity of multireference alignment.
\newblock {\em SIAM Journal on Mathematics of Data Science\/}~{\em 1\/}(3),
  497--517.

\bibitem[\protect\citeauthoryear{Perry, Wein, Bandeira, and Moitra}{Perry
  et~al.}{2018}]{perry2018message}
Perry, A., A.~S. Wein, A.~S. Bandeira, and A.~Moitra (2018).
\newblock Message-passing algorithms for synchronization problems over compact
  groups.
\newblock {\em Communications on Pure and Applied Mathematics\/}~{\em
  71\/}(11), 2275--2322.

\bibitem[\protect\citeauthoryear{Pozdnyakov, Glaz, Kulldorff, and
  Steele}{Pozdnyakov et~al.}{2005}]{pozdnyakov2005martingale}
Pozdnyakov, V., J.~Glaz, M.~Kulldorff, and J.~M. Steele (2005).
\newblock A martingale approach to scan statistics.
\newblock {\em Annals of the Institute of Statistical Mathematics\/}~{\em
  57\/}(1), 21--37.

\bibitem[\protect\citeauthoryear{Proksch, Werner, and Munk}{Proksch
  et~al.}{2018}]{proksch2018multiscale}
Proksch, K., F.~Werner, and A.~Munk (2018).
\newblock Multiscale scanning in inverse problems.
\newblock {\em The Annals of Statistics\/}~{\em 46\/}(6B), 3569--3602.

\bibitem[\protect\citeauthoryear{Seijo and Sen}{Seijo and
  Sen}{2011}]{seijo2011change}
Seijo, E. and B.~Sen (2011).
\newblock Change-point in stochastic design regression and the bootstrap.
\newblock {\em The Annals of Statistics\/}~{\em 39\/}(3), 1580--1607.

\bibitem[\protect\citeauthoryear{Serre, Wolf, and Poggio}{Serre
  et~al.}{2005}]{serre2005object}
Serre, T., L.~Wolf, and T.~Poggio (2005).
\newblock Object recognition with features inspired by visual cortex.
\newblock In {\em 2005 IEEE Computer Society Conference on Computer Vision and
  Pattern Recognition (CVPR'05)}, Volume~2, pp.\  994--1000. IEEE.

\bibitem[\protect\citeauthoryear{Shao}{Shao}{1995}]{shao1995conjecture}
Shao, Q.~M. (1995).
\newblock On a conjecture of r{\'e}v{\'e}sz.
\newblock {\em Proceedings of the American Mathematical Society\/}~{\em
  123\/}(2), 575--582.

\bibitem[\protect\citeauthoryear{Sharpnack and Arias-Castro}{Sharpnack and
  Arias-Castro}{2016}]{sharpnack2016exact}
Sharpnack, J. and E.~Arias-Castro (2016).
\newblock Exact asymptotics for the scan statistic and fast alternatives.
\newblock {\em Electronic Journal of Statistics\/}~{\em 10\/}(2), 2641--2684.

\bibitem[\protect\citeauthoryear{Siegmund}{Siegmund}{2013}]{siegmund2013sequential}
Siegmund, D. (2013).
\newblock {\em Sequential analysis: tests and confidence intervals}.
\newblock Springer Science+Business Media.

\bibitem[\protect\citeauthoryear{Siegmund and Venkatraman}{Siegmund and
  Venkatraman}{1995}]{sieg95}
Siegmund, D. and E.~S. Venkatraman (1995).
\newblock Using the generalized likelihood ratio statistic for sequential
  detection of a change-point.
\newblock {\em The Annals of Statistics\/}~{\em 23\/}(1), 255--271.

\bibitem[\protect\citeauthoryear{Sotiras, Davatzikos, and Paragios}{Sotiras
  et~al.}{2013}]{sotiras2013deformable}
Sotiras, A., C.~Davatzikos, and N.~Paragios (2013).
\newblock Deformable medical image registration: A survey.
\newblock {\em IEEE Transactions on Medical imaging\/}~{\em 32\/}(7),
  1153--1190.

\bibitem[\protect\citeauthoryear{Talagrand}{Talagrand}{1996}]{talagrand1996new}
Talagrand, M. (1996).
\newblock New concentration inequalities in product spaces.
\newblock {\em Inventiones Mathematicae\/}~{\em 126\/}(3), 505--563.

\bibitem[\protect\citeauthoryear{Trigano, Isserles, and Ritov}{Trigano
  et~al.}{2011}]{trigano2011semiparametric}
Trigano, T., U.~Isserles, and Y.~Ritov (2011).
\newblock Semiparametric curve alignment and shift density estimation for
  biological data.
\newblock {\em IEEE Transactions on Signal Processing\/}~{\em 59\/}(5),
  1970--1984.

\bibitem[\protect\citeauthoryear{Truong, Oudre, and Vayatis}{Truong
  et~al.}{2020}]{truong2020selective}
Truong, C., L.~Oudre, and N.~Vayatis (2020).
\newblock Selective review of offline change point detection methods.
\newblock {\em Signal Processing\/}~{\em 167}, 107299.

\bibitem[\protect\citeauthoryear{Tsybakov}{Tsybakov}{2009}]{tsybakov2009introduction}
Tsybakov, A.~B. (2009).
\newblock {\em Introduction to nonparametric estimation}.
\newblock Springer Science$+$Business Media.

\bibitem[\protect\citeauthoryear{Turin}{Turin}{1960}]{turin1960introduction}
Turin, G. (1960).
\newblock An introduction to matched filters.
\newblock {\em IRE Transactions on Information Theory\/}~{\em 6\/}(3),
  311--329.

\bibitem[\protect\citeauthoryear{van~der Vaart}{van~der
  Vaart}{1998}]{van2000asymptotic}
van~der Vaart, A.~W. (1998).
\newblock {\em Asymptotic statistics}.
\newblock Cambridge University Press.

\bibitem[\protect\citeauthoryear{van~der Vaart and Wellner}{van~der Vaart and
  Wellner}{1996}]{van1996weak}
van~der Vaart, A.~W. and J.~A. Wellner (1996).
\newblock {\em Weak convergence and empirical processes with applications to
  statistics}.
\newblock Springer.

\bibitem[\protect\citeauthoryear{Vimond}{Vimond}{2010}]{vimond2010efficient}
Vimond, M. (2010).
\newblock Efficient estimation for a subclass of shape invariant models.
\newblock {\em The Annals of Statistics\/}~{\em 38\/}(3), 1885--1912.

\bibitem[\protect\citeauthoryear{Walther}{Walther}{2010}]{MR2604703}
Walther, G. (2010).
\newblock Optimal and fast detection of spatial clusters with scan statistics.
\newblock {\em The Annals of Statistics\/}~{\em 38\/}(2), 1010--1033.

\bibitem[\protect\citeauthoryear{Wang and Gasser}{Wang and
  Gasser}{1999}]{wang1999synchronizing}
Wang, K. and T.~Gasser (1999).
\newblock Synchronizing sample curves nonparametrically.
\newblock {\em The Annals of Statistics\/}~{\em 27\/}(2), 439--460.

\bibitem[\protect\citeauthoryear{Wang and Singer}{Wang and
  Singer}{2013}]{wang2013exact}
Wang, L. and A.~Singer (2013).
\newblock Exact and stable recovery of rotations for robust synchronization.
\newblock {\em Information and Inference: A Journal of the IMA\/}~{\em 2\/}(2),
  145--193.

\bibitem[\protect\citeauthoryear{Wang and Glaz}{Wang and
  Glaz}{2014}]{wang2014variable}
Wang, X. and J.~Glaz (2014).
\newblock Variable window scan statistics for normal data.
\newblock {\em Communications in Statistics - Theory and Methods\/}~{\em
  43\/}(10-12), 2489--2504.

\bibitem[\protect\citeauthoryear{Yao and Au}{Yao and Au}{1989}]{yao1989least}
Yao, Y.-C. and S.-T. Au (1989).
\newblock Least-squares estimation of a step function.
\newblock {\em Sankhy{\=a}: The Indian Journal of Statistics, Series A\/}~{\em
  51\/}(3), 370--381.

\bibitem[\protect\citeauthoryear{Zitova and Flusser}{Zitova and
  Flusser}{2003}]{zitova2003image}
Zitova, B. and J.~Flusser (2003).
\newblock Image registration methods: a survey.
\newblock {\em Image and Vision Computing\/}~{\em 21\/}(11), 977--1000.

\end{thebibliography}
